\documentclass[a4paper,12pt,oneside,reqno]{amsart}
\usepackage{hyperref}
\usepackage[noabbrev,capitalize]{cleveref}
\usepackage[headinclude,DIV=13]{typearea}
\areaset{15.1cm}{25.0cm}
\usepackage{txfonts,amssymb,amsmath,amsthm,bbm}

\usepackage[utf8] {inputenc}
\usepackage{graphicx, psfrag}
\usepackage{xcolor}
\usepackage{subfigure}

\usepackage{enumitem}				
\usepackage{verbatim}
\usepackage{soul} 
\usepackage[noadjust]{cite} 

\newtheorem{theorem}{Theorem}[section]
\newtheorem{lemma}[theorem]{Lemma}
\newtheorem{proposition}[theorem]{Proposition}

\newtheorem{assumption}[theorem]{Assumption}

\theoremstyle{remark}
\newtheorem*{remark}{Remark}

\newtheoremstyle{noparentheses}
{\topsep}   
{\topsep}   
{\itshape}  
{0pt}       
{\bfseries} 
{}         
{5pt plus 1pt minus 1pt} 
{\thmname{#1} \thmnumber{#2} \normalfont{\thmnote{#3}}\textbf{.}}          
\theoremstyle{noparentheses}
\newtheorem{theorem*}[theorem]{Theorem}
\newtheorem{lemma*}[theorem]{Lemma}
\newtheorem{proposition*}[theorem]{Proposition}
\newtheorem{corollary*}[theorem]{Corollary}
\newtheorem{assumption*}[theorem]{Assumption}

\newtheoremstyle{klammern}
{\topsep}   
{\topsep}   
{\itshape}  
{0pt}       
{\bfseries} 
{}         
{5pt plus 1pt minus 1pt} 
{\vspace{-6pt}$ $\\  \thmname{#1} \thmnumber{#2} \normalfont{[\hspace{-4pt}}\normalfont{
		\thmnote{#3}
		}\normalfont{\hspace{-4pt}]}\textbf{.}}          

\theoremstyle{klammern}

\newtheorem{lemmab}[theorem]{Lemma}
\newtheorem{propositionb}[theorem]{Proposition}

\def\paragraph#1{\noindent \textbf{#1}}

\numberwithin{equation}{section}

\newlist{thmlist}{enumerate}{1}
\setlist[thmlist,1]{label=(\roman{thmlisti}), ref=\thetheorem.(\roman{thmlisti}),noitemsep} 

\newlist{thmlistb}{enumerate}{1}
\setlist[thmlistb,1]{label=(\alph{thmlistbi}), ref=\thetheorem.(\alph{thmlistbi}),noitemsep} 

\definecolor{bbm}{RGB}{51,153,0}
\definecolor{above}{RGB}{128,0,128}
\definecolor{below}{RGB}{102,0,204}
\definecolor{cascade}{RGB}{204,0,0}
\definecolor{iid}{RGB}{153,51,0}

\def\a{\alpha}
\def\b{\beta}
\def\g{\gamma}
\def\e{\epsilon}

\def\s{\sigma}

\def\o{\omega}

\def\D{\Delta}

\def\E{{\mathbb E}} 
  
\def\N{{\mathbb N}}  
\def\P{{\mathbb P}} 

\def\R{{\mathbb R}}

\let\cal=\mathcal
\def\AA{{\cal A}}
\def\BB{{\cal B}}
\def\CC{{\cal C}}

\def\EE{{\cal E}}
\def\FF{{\cal F}}
\def\GG{{\cal G}}

\def\LL{{\cal L}}

\def\OO{{\cal O}}
\def\PP{{\cal P}}

\def\SS{{\cal S}}
\def\TT{{\cal T}}

\def\YY{{\cal Y}}
\def\ZZ{{\cal Z}}

\def \zet {{\mathfrak z}}

\def \log{\ln} 

\def\eee{{\mathrm e}}
\def\d{\mathrm{d}}

\def\del{\partial}

\def\mi{\bar{i}}
\newcommand{\multix}[1]{\tilde{x}_{i_{#1}}^{\, \bar{i}_{#1-1}} }
\newcommand{\multixb}[1]{\bar{x}_{i_{#1}}^{\, \bar{i}_{#1-1}} }

\def\1{\mathbbm{1}}

\def\<{\langle}
\def\>{\rangle}

\def \ba {\begin{array}}
\def \ea {\end{array}}

\newcommand{\be}{\begin{equation}}
\newcommand{\ee}{\end{equation}}

\newcommand{\bea}{\begin{eqnarray}}
\newcommand{\eea}{\end{eqnarray}}

\def\TH(#1){\label{#1}}\def\thv(#1){\ref{#1}}
\def\Eq(#1){\label{#1}}\def\eqv(#1){(\ref{#1})}

\def\sfrac#1#2{{\textstyle{\frac{#1}{#2}}}}

\def\BBM{branching Brownian motion }

\newcommand{\kl}[1]{\left(#1\right)}
\newcommand{\bkl}[1]{\big(#1\big)}
\newcommand{\bbkl}[1]{\Big(#1\Big)}
\newcommand{\bbbkl}[1]{\bigg(#1\bigg)}
\newcommand{\bbbbkl}[1]{\Bigg(#1\Bigg)}

\newcommand{\gkl}[1]{\left\{#1\right\}}

\newcommand{\ekl}[1]{\left[#1\right]}

\newcommand{\bbbbekl}[1]{\Bigg[#1\Bigg]}

\newcommand{\okl}[1]{\left\lceil#1\right\rceil}

\newcommand{\abs}[1]{\left| #1 \right|}

\newcommand*\iid{i.i.d.\ }

\DeclareMathOperator{\conc}{conc}

\renewcommand{\ae}{\a_{\mathrm{end}}}
\newcommand{\bend}{b_{\mathrm{end}}}
\newcommand{\ab}{\a_{\mathrm{begin}}}
\newcommand{\bbeg}{b_{\mathrm{begin}}}

\begin{document}

 \title[The log-correction for the maximum of variable-speed BBM]{From 1 to infinity: The log-correction for the maximum of variable-speed branching Brownian motion}

\author[A. Alban]{Alexander Alban}
\address{A. Alban\\
 	Institut f\"ur Mathematik \\Johannes Gutenberg-Universit\"at Mainz\\
	Staudingerweg 9,	
	55128 Mainz, Germany}
\email{aalban@uni-mainz.de}
\author[A. Bovier]{Anton Bovier}
\address{A. Bovier\\
	Institute for Applied Mathematics\\University of Bonn\\ 
	Endenicher Allee 60\\ 
	53115 Bonn, Germany }
\email{bovier@uni-bonn.de}
\author[A. Gros]{Annabell Gros}
\address{A. Gros\\
	Institute for Applied Mathematics\\University of Bonn\\ 
	Endenicher Allee 60\\ 
	53115 Bonn, Germany }
\email{gros@iam.uni-bonn.de}
\author[L. Hartung]{Lisa Hartung}
\address{L. Hartung\\
 	Institut f\"ur Mathematik \\Johannes Gutenberg-Universit\"at Mainz\\	
 	Staudingerweg 9,
 	55128 Mainz, Germany}
\email{lhartung@uni-mainz.de}
\date{\today}

 \begin{abstract}  We study the extremes of variable speed branching Brownian motion (BBM) where the time-dependent ``speed functions", which describe the time-inhomogeneous variance, converge to the identity function.
 	We consider general speed functions lying strictly below their concave hull and piecewise linear, concave speed functions.
 	In the first case, the log-correction for the order of the maximum depends only on the rate of convergence of the speed function near 0 and 1 and exhibits a smooth interpolation between the correction in the i.i.d.\ case, $\sfrac{1}{2\sqrt{2}} \log t$, and that of standard BBM, $\sfrac{3}{2\sqrt{2}} \log t$.
 	In the second case, we describe the order of the maximum in dependence of the form of speed function and show that any log-correction larger than $\sfrac{3}{2\sqrt{2}} \log t$ can be obtained. 
 	In both cases, we prove that the limiting law of the maximum and the extremal process essentially coincide with those of standard BBM, using a first and second moment method which relies on the localisation of extremal particles.
 	This extends the results of  \cite{1to6} for two-speed BBM.
 \end{abstract}

\thanks{
$\!$Published in Electron.\ J.\ Probab.\ 30:1-46, 2025,  \url{https://doi.org/10.1214/25-EJP1287}.\\
\indent This work was partly funded by the Deutsche Forschungsgemeinschaft (DFG, German Research Foundation) under Germany's Excellence Strategy - GZ 2047/1, Project-ID 390685813 and GZ 2151 - Project-ID 390873048,
through Project-ID 211504053 - SFB 1060, through Project-ID 233630050 - TRR 146, through Project-ID 443891315  within SPP 2265, and Project-ID 446173099.
 }


 \maketitle

\section{Introduction}

\subsection {Models and background}
\emph{Variable speed branching Brownian motion [VSBBM]} \cite{1to6,BH14.1,BovHar13, DerriSpohn88,FZ_BM, MZ}  is a class of Gaussian processes $X$ indexed by a continuous-time Galton-Watson-tree with branching  rate one and offspring distribution $(p_k)_{k\in\N}$, where $\sum_{k=1}^{\infty} p_k = 1$, $\sum_{k=1}^{\infty} k p_k =2$ and  $\sum_{k=1}^{\infty} k(k-1)p_k < \infty$. $X$  has mean zero and covariance
\begin{equation}
	\E \left[ x_i(s) x_j(r) \:\mid \FF_t^{\text{tree}} \right]
	= t A (t^{-1} d( x_i(s), x_j(r))),
\end{equation}
where $\FF_t^{\text{tree}}$ denotes the $\s$-algebra generated by the Galton-Watson tree up to time $t$ and $d(x_i(s), x_j(r))$ denotes
the time of the most recent common ancestor of the particles labelled $i$ and $j$ in the tree. $A \colon [0,1] \to [0,1]$ with $A(0)=0$ and $A(1) =1$ is a non-decreasing and right-continuous function, called \emph{speed function}. It is the continuous-time analogon to Derrida's Generalised and Continuous Random Energy Model (GREM/CREM) \cite{GD86b, BK1, BK2, kistler2015,KistSchmi15}. 
VSBBM is a family $(\tilde{X}_t)_{t>0}$ of processes where $\tilde{X}_t = \{\tilde{x}_j^t(s)\colon j \leq n(t), s \leq t \}$ denotes the trajectories of all particles when the time horizon of the process is $t$. 
We write $\tilde{x}_j^t(s)$ for the position at time $s$ of the ancestor of a particle labelled $j$ at time $t$.
For simplicity, we write $\tilde{x}_j(t)\equiv \tilde{x}_j^t(t)$. The trajectory $\{\tilde{x}^t_j(s) \colon s \leq t\}$ for $j \leq n(t)$ is denoted by $\tilde{x}_j$.

The case when  $A(x)\equiv x$ is standard branching Brownian motion (BBM), the primary example of 
so-called log-correlated processes, a class of processes that contains,  among others, branching random walk and the discrete Gaussian free field in dimension two. Also for these models, deformations with
different speed functions analogous to VSBBM have been studied, see \cite{FZ_RW,Mallein,Ouimet} for the branching random walk and \cite{GFF1,GFF2,GFF3} for the discrete Gaussian free field in dimension two.

The extreme value statistics of standard BBM are by now well understood, 
see, e.g. \cite{B_M,B_C, LS, chauvin88, chauvin90, ABK_G, ABK_P, ABK_E}.
The extreme values of variable speed BBM exhibit  different behaviour depending on the properties of the speed function.
\begin{enumerate}[label=(\roman*)]
\item If $A(x) < x$ for all $x \in (0,1)$, then has been shown in \cite{BH14.1} that $A'(0) < 1$ and $A'(1) > 1$ imply that to first sub-leading order,
\begin{equation}
	\label{eq:OrderOfMaxBelow}
	\max_{j \le n(t)} \tilde{x}_j(t) \approx \sqrt{2}t - \frac{1}{2 \sqrt{2}} \log(t).
\end{equation}
The order of the maximum is the same as in the case of independent particles.  
\item If $A(x)=x$ for all $x \in [0,1]$, Bramson \cite{B_M,B_C} has proven that
\begin{equation}
	\max_{j \le n(t)} \tilde{x}_j(t) \approx m(t) \equiv \sqrt{2}t - \frac{3}{2 \sqrt{2}} \log(t).\label{eq:OrderMaxStandardBBM}
\end{equation}
\item If $A(x)>x$ for some $x \in (0,1)$, then, to leading order,
\begin{equation} \label{eq:OrderOfMaxFirstOrder}
	\max_{j \le n(t)} \tilde{x}_j(t) \approx \sqrt{2} t \int_{0}^{1} \!\sqrt{ \conc(A)'(y)} \,\d y,
\end{equation}
where $\conc(A)$ denotes the concave hull of the function $A$. The sub-leading order depends on the specific form of the covariance. If $A$ is a piecewise linear function with slope $\s_1^2$ on the interval $[0,b), b \in (0,1)$ and $\s_2^2$ on $[b,1]$, it has been shown in \cite{FZ_RW, BovHar13} that the log-correction is 
\begin{equation} \label{eq:log_correction_piecewise}
	- \frac{3}{2\sqrt{2}} \left( \s_1 \log(bt) + \s_2 \log((1-b)t) \right).
\end{equation}
If $A$ is strictly concave and continuous, the sub-leading order of the maximum is of order $t^{1/3}$ (see \cite{FZ_BM}; \cite{MZ} for a refinement).
\end{enumerate}
Note that, in the first and second cases, the concave hull of $A$ is the identity function. 
Therefore, \eqref{eq:OrderOfMaxFirstOrder} holds in all three cases. 
We see that standard  BBM is on the borderline  where correlations begin to affect the properties of the extremes.  Moreover, the sub-leading terms are discontinuous at the identity function.  To analyse  these discontinuities in more detail, \cite{1to6} considered  piecewise linear speed functions $A_t$ such that  
\be
A'_t(x) =\begin{cases} 1\pm t^{-\a},&\; x< 1/2,\\
1\mp t^{-\a}, &\; 1/2< x\leq  1.
\end{cases}
\ee
Another example was studied by Kistler and Schmidt
\cite{KistSchmi15}. In the present paper, we generalise the analysis in
\cite{1to6} to a wide class of speed functions.
We distinguish between piecewise linear speed functions converging from above and a general class of 
speed functions converging from below. The case above the identity function is referred to as \emph{Case A},  the other one as case \emph{Case B}. 

As explained above, the dependence of the properties of extremes on the speed function is very different in Cases
A and B. In particular, in case $A$, the techniques of proofs are very different in the case when $A$ is piecewise linear
and when it is strictly concave. Therefore, in this paper we restrict ourselves to the piecewise linear case. The precise conditions on the speed functions considered are given below in Assumption~\ref{as:above}. In Case~B, 
the dependence on the speed function is only on the slopes of the speed function at zero and one, so in this case there is no need to distinguish between piecewise linear and other speed functions. The corresponding conditions are formulated in Assumption~\ref{as:below}.

The following assumptions describe the class of speed functions we consider in Cases A and B.
\begin{assumption}[\textbf{Case A; }$\mathbf{A_t(s) > s}$]\label{as:above}
The family of speed functions $(A_t)_{t>0}$ with $A_t(s) > s$ for all $t > 0, s \in (0,1)$, satisfies:

\begin{thmlist}
	\item The functions $(A_t)_{t>0}$ are piecewise linear and continuous.
	Their derivatives are given by
	\begin{equation}
		A_t'(s) = 
		\sum_{k=1}^{\ell} \s_k^2(t) \1_{ \left( \sum_{j=1}^{k-1} b_j(t), \sum_{j=1}^{k} b_j(t) \right)}(s),
	\end{equation}
	where we call $\s_k\colon\R_+ \to \R_+, 1 \le k \le \ell,$ \emph{velocities} and \mbox{$b_k\colon \R_+ \to (0,1]$},${1 \le k \le \ell}$, are called \emph{interval lengths}. 
	We assume $\sum_{k=1}^{\ell} b_k(t)=1$ and $\sum_{k=1}^{\ell} \s_k^2(t) b_k(t) = 1$. \label{as:above1}
	\item The functions $(A_t)_{t>0}$ are concave and converge to the identity function, as $t\uparrow \infty$. \label{as:above2}
	\item There exists $\b \in (0, 1/2)$ such that, for all $1 \le k < \ell$, \\
	$ \displaystyle \sqrt{ \min \{ b_k(t)t, b_{k+1}(t)t \} } \textstyle \gg ( \s_k(t) - \s_{k+1}(t)  )^{-1} \gg t^{\b}$, as $t \uparrow \infty$.
	\label{as:above3}
\end{thmlist}
\end{assumption}
Here and elsewhere, we use the notation
\begin{align}\label{eq:NotationLLGG}
f(t) \ll g(t), \text{as $t\uparrow\infty$},
\quad\Leftrightarrow\quad 
\exists\, \e > 0\colon \frac{t^\e f(t)}{g(t)} \downarrow 0, \text{as $t\uparrow\infty$},
\end{align}
for functions $f, g \colon \R_+\to\R$.

To illustrate the assumptions in Case A, consider a two-speed BBM with velocities 
$\s_1^2(t) = 1 + t^{-\a_1}$ on the interval $[0,b(t))$ and
$\s_2^2(t) = 1 - t^{-\a_2}$ on $[b(t), 1]$, with $b(t) = 1/( t^{\a_2 - \a_1} + 1 )$ and $\a_1, \a_2 > 0$. One checks that the assumptions are verified if $\a_1+\a_2<1$.
\begin{assumption}[\textbf{Case B; }$\mathbf{A_t(s) < s}$]\label{as:below}
Let $\ab, \ae \in \kl{0, 1/2}$. The family of speed functions $\kl{A_t}_{t>0}$ with $A_t(s) < s$, for all $t>0, s\in (0,1)$, satisfies:
\begin{thmlistb}
	\item \label{as:belowA}
	For each $t>0$, there exist $\bbeg(t) \in (0,1)$ and twice differentiable functions $\underline{B}_t, \overline{B}_{\,t} \colon [0,1]\to[0,1]$ with $\underline{B}_t(0) = \overline{B}_{\,t}(0) = 0$, for which each of the following hold:
	\begin{enumerate}[label=(\roman*)] 
		\item $1 \gg \bbeg(t) \gg t^{\ab\, - \,1/2}$ as $t\uparrow\infty$.
		\item $\underline{B}'_t(0) = \overline{B}'_{\,t}(0) = 1 - t^{-\ab}$. 
		\item On $[0,\bbeg(t)]$, we have $\underline{B}_{\,t} \leq A_t \leq \overline{B}_{t}$ and the second derivatives of $\underline{B}_t, \overline{B}_{\,t}$ are both bounded by $t^{-\ab} \bbeg(t)^{-1}$ in the sense of \eqref{eq:NotationLLGG}.
	\end{enumerate}
	\item \label{as:belowB}
	For each $t>0$, there exist $\bend(t) \in (0,1)$ and twice differentiable functions $\underline{C}_t, \overline{C}_{\,t} \colon [0,1]\to[0,1]$ with $\underline{C}_t(1) = \overline{C}_{\,t}(1) = 1$, 
	such that:
	\begin{enumerate}[label=(\roman*)] 
		\item $1 \gg \bend(t) \gg  t^{\ae\, -\, 1/2}$ as $t\uparrow\infty$.
		\item $\underline{C}'_t(1) = \overline{C}'_{\,t}(1) = 1 + t^{-\ae}$.
		\item On $[1-\bend(t), 1]$, we have $\underline{C}_{\,t} \leq A_t \leq \overline{C}_{t}$ and the second derivatives of $\underline{C}_t, \overline{C}_{\,t}$ are both bounded by $t^{-\ae} \bend(t)^{-1}$ in the sense of \eqref{eq:NotationLLGG}.     
	\end{enumerate}
	\item \label{as:belowC}
	$\min_{s\in[\bbeg(t), 1-\bend(t)]} \bkl{s - A_t(s)} \gg t^{-1/2}$, as $t\uparrow\infty$.
\end{thmlistb}
\end{assumption}
In Case B, the slopes in~$0$ and~$1$ are given by $1-t^{-\ab}$ and $1+t^{-\ae}$.
The assumptions ensure that $A_t$ can be well approximated by piecewise linear functions  in $0$ and $1$, similarly to the assumptions in \cite{BH14.1}.

%
%
%

In this paper, we determine the limiting law of the rescaled maximum and the full extremal process in 
both cases. 
Recall that, for BBM, see \cite{B_M,LS}, 
\begin{equation}\label{eq:lisa1}
\lim_{t \uparrow \infty} \P \left( \max_{j \le n(t)} x_j(t) - m(t) \le y \right)  
= \E \left[ \eee^{- C Z \eee^{-\sqrt{2}y}} \right],
\end{equation}
where  $m(t)$ is the same as in \eqref{eq:OrderMaxStandardBBM}, $Z$ is the limit of the \emph{derivative martingale}
\begin{equation}\label{eq:derivMart}
Z(t) 
\equiv \sum_{j\leq n(t)} \left( \sqrt{2}t - x_j(t) \right) \eee^{-\sqrt{2} \left( \sqrt{2} t - x_j(t) \right)},
\end{equation}
and   $C$ is a positive constant.

The extremal process of standard BBM \cite{ABK_E,ABBS} is of the form
\begin{align}
\lim_{t\uparrow\infty} \sum_{j\leq n(t)} \delta_{x_j(t)-m(t)} = \sum_{k,j} \delta_{\eta_k + \D_j^{(k)}},
\label{eq:ExtremalProessStandardBBM}
\end{align}
where the points
$\eta_k$ are the atoms of a Poisson point process with random intensity measure $CZ\sqrt{2}\eee^{-\sqrt{2}y} \mathrm{d} y$. 
The points $\D_j^{(k)}$ are the atoms of \iid point processes $\D^{(k)}$, which arise as the limit in law as $t\uparrow\infty$ of 
\begin{align}
\sum_{j \leq n(t)}  \delta_{\bar{x}_j (t) - \max_{i \leq n(t)} \bar{x}_i(t)},
\label{eq:ExtremalProessDecoration}
\end{align}
where $\bar{x}(t)$ are the points of standard BBM conditioned on the event  $\max_{i \leq n(t)} x_i(t) \geq \sqrt{2}t$. 
\subsection{Results}
To state our results, we define functions  $m^{\pm}$, where the superscript $+$ corresponds to Case~A and the superscript $-$ to Case~B.
Let
\begin{equation} \label{eq:OrderOfMaxAboveTime}
m^+(t)
\equiv \sqrt{2}t \left( \sum_{k=1}^{\ell} \s_k(t) b_k(t) \! \right) \! 
- \frac{3}{2 \sqrt{2}} \left( \sum_{k=1}^{\ell} \log(b_k(t)t) 
+ 2 \sum_{k=1}^{\ell-1} \log \left( \pi^{1/6} (\s_k(t) - \s_{k+1}(t)) \right) \! \right),
\end{equation}
and 
\begin{equation} \label{eq:OrderOfMaxBelowTime}
m^-(t) \equiv \sqrt{2}t - \frac{1 + 2 (\ab + \ae)}{2\sqrt{2}}  \log(t).
\end{equation}
We notice that $m^+$ depends on the details of the speed functions $(A_t)_{t>0}$ while $m^-$ depends only on the rate of convergence of the speed functions near $0$ and $1$.
%
\noindent The main results of this paper are the following two theorems.
\begin{theorem}\label{thm:lawofmax}
Let $(\tilde{X}_t)_{t>0}$ be a family of variable speed BBMs with speed functions $(A_t)_{t>0}$ satisfying Assumption~\ref{as:above} (Case~A) or Assumption~\ref{as:below} (Case~B). 
Let $C$ be the same positive constant as in \eqref{eq:lisa1}
and	$Z$ the limit of the derivative martingale.
Then, for all $ y \in \R$,
\begin{equation}
	\lim_{t \uparrow \infty} \P \left( \max_{j \le n(t)} \tilde{x}_j(t) - m^{\pm}(t) \le y \right)  
	=\E \left[ \exp\kl{-C
		Z \eee^{-\sqrt{2}y}} \right],
\end{equation} 
where $m^\pm = m^+$ in Case~A and $m^\pm = m^-$ in Case~B.
\end{theorem}

\begin{theorem}\label{thm:extremalprocess}
Let $(\tilde{X}_t)_{t>0}$ and $m^\pm$ be as in Theorem~\ref{thm:lawofmax}.
Let $\D_j^{(k)}$ be the atoms of the \iid copies $\D^{(k)}$ of the limit of the point process described in \eqref{eq:ExtremalProessDecoration}.
Then,
\begin{equation}  \label{eq:ExtrProc}
	\lim_{t\uparrow\infty} \sum_{j\leq n(t)} \delta_{\tilde{x}_j(t)-m^{\pm}(t)} = \sum_{k,j} \delta_{\eta_k + \D_j^{(k)}},
\end{equation}
where $\eta_k$ are the atoms of a Poisson point process with random intensity measure\\ \mbox{$CZ\sqrt{2}\eee^{-\sqrt{2}y} \mathrm{d} y$}.
\end{theorem}

Observe that in Case~A,  we can obtain any factor between $\frac{3}{2 \sqrt{2} } \log(t) $ and $\infty$ in front of the logarithmic correction with an appropriate choice of $A_t$.
More precisely,  the logarithmic correction is of the form
\begin{equation} \label{eq:log_corr_above}
- \frac{3}{2 \sqrt{2} } f(A_t) \log(t),
\end{equation}
where $f$ is a function taking values in $(1, \ell)$. This follows from estimating the terms in $m^+$ from above and below with the bounds in Assumption~\ref{as:above3} and using that $b_k(t) \le 1$ for \mbox{$1 \le k \le \ell$}.
In Case B, the logarithmic correction in \eqref{eq:OrderOfMaxBelowTime} only depends on the slope of $A_t$ near~$0$ and near~$1$, while the behaviour of $A_t$ away from~$0$ and~$1$ is negligible as long as $A_t$ maintains a distance of order $t^{-1/2}$ from the identity function. The prefactor of the logarithmic correction interpolates between $\sfrac{1}{2\sqrt2}$ and $\sfrac{3}{2\sqrt2}$. 

Both theorems above follow from the convergence of a class of Laplace functionals. A very nice characterisation of this fact is the following lemma from \citen{Berestycki}.

\begin{lemmab}[\citen{Berestycki}, Lemma 4.4] \label{lem:jointConvergence}
Let $\PP_t, \PP_\infty$ be point processes on $\R$ such that almost surely, $\PP((0,\infty))<\infty$.
The following are equivalent: As $t\uparrow\infty$,
\begin{thmlist}
	\item $\PP_t \to \PP_\infty$ and $\max\PP_t \to \max\PP_\infty$ in distribution.
	\item $\E \, \left[ \exp \left( - \int \phi(y) \PP_t( \d y) \right) \right] \to \E \, \left[ \exp \left( - \int \phi(y) \PP_\infty( \d y) \right) \right]$  for all $\phi\in \CC^\infty$ which are nondecreasing with support bounded from the left and for which there exists $a\in \R$ such that $\phi(x)$ is constant for $x>a$.
	\label{lem:jointConvergence2}
\end{thmlist}  
\end{lemmab}
Note that the fact that the functions $\phi$ are required to have support bounded only from the left allows to use 
Bramson's results on the convergence of the F-KPP equations directly and the fact that $\phi$ can be chosen to be smooth is 
convenient for applying Gaussian comparison. We prefer, however, to give the proof
of Theorem \ref{thm:lawofmax} without using Laplace functionals, since we find this more easy to follow. The proof 
of Theorem \ref{thm:extremalprocess} is then very similar and we only outline the main differences.

%
%

\subsection{Outline of the paper}

The remainder of this paper is organised as follows.
Section~\ref{sec:notation} provides a collection of relevant notation.
Section~\ref{sec:preliminaries} recalls facts on Brownian bridges and the asymptotics of solutions of the F-KPP equation. 
A crucial step towards the proofs of Theorem~\ref{thm:lawofmax} and \ref{thm:extremalprocess} is to localise the positions of the ancestral paths of extremal particles. This is done in Section~\ref{sec:localisation}.
Section~\ref{sec:mainProofs} contains the proof of Theorem~\ref{thm:lawofmax} and \ref{thm:extremalprocess}. First, we give a  proof of Theorem~\ref{thm:lawofmax}, which is split between Case~A in Subsection~\ref{sec:proof_caseA} and Case~B in Subsection~\ref{sec:proof_CaseBLinear}. In both cases, we prove the claim for piecewise linear speed functions. Some technical details are postponed to  Appendix~\ref{sec:appendix}. In Case~B, we extend the result to general speed functions with Gaussian comparison techniques. 
In Subsection~\ref{sec:proofExtremalProcess}, we describe how to modify the proof of Theorem~\ref{thm:lawofmax} so it extends to the convergence of Laplace functionals. Applying Lemma~\ref{lem:jointConvergence} then completes the proof of Theorem~\ref{thm:extremalprocess}.

\section{Notation} \label{sec:notation}
In this section, we introduce notation for $\ell$-speed BBM which is used throughout the paper.  Denote by
\begin{equation}
a_k(t) \equiv \sum_{j=1}^k b_j(t), \quad a_0(t) = 0,
\end{equation}
for all $1 \le k \le \ell$, the times at which speed changes occur.
In the following we drop the $t$-dependence of the terms $\s_k(t), b_k(t)$ and $a_k(t)$ to shorten the notation. 

It is convenient to express $\ell$-speed BBM using standard BBMs. To do so, let \mbox{$\{ \tilde{x}_{i_k}^{i_1, \dots, i_{k-1}}$}, \mbox{$k\in \N, i_\ell\in\N \}$}  be BBM with variance $\s_k^2$.

We use multiindices 
\begin{equation}
\bar{i}_k \equiv i_1,\dots,i_k
\end{equation}
and write $\tilde{x}_{i_1}^{\mi_0} = \tilde{x}_{i_1}$. With this notation, we can rewrite variable speed BBM, as
\begin{align} \label{eq:multiindex_notation}
&\left\{ \tilde{x}^t_i (s): 1 \le i \le n(t) \right\} \nonumber \\
&=\left\{ \sum_{j=1}^{k-1}  \multix{j}(b_jt) 
+  \multix{k}(s - a_{k-1} t):
1 \le i_1 \le n(b_1t), \dots, 1 \le i_k \le n^{\mi_{k-1}}(b_kt) \right\},
\end{align} 
for $1\leq k < \ell$ and for \mbox{$s \in [a_{k-1} t, a_k t)$}.
The $\s$-algebra which is generated by all particles of $\ell$-speed BBM up to time $s, s \le t$, is called $\FF_s$.  The path of the particle with position $\sum_{j=1}^k \multix{j}(b_jt)$ at time $a_{k} t$ is abbreviated by $\sum_{j=1}^k \multix{j}$.

\section{Preliminaries} \label{sec:preliminaries}
In this section, we recall some results on \BBM and the F-KPP equation.
We start with the fundamental connection between BBM and the F-KPP equation. Let $f:\R \to [0,1]$ be a function and set
\begin{equation}
v(t,x) \equiv \E \left[ \prod_{j=1}^{n(t)} f(x-x_k(t)) \right].
\end{equation}
Then $u(t,x)=1-v(t,x)$ is the unique solution to the F-KPP equation 
\begin{equation} \label{eq:FKPP}
\del_t u = \frac{1}{2} \del_x^2 u + F(u),
\end{equation}
with $F(u)= (1-u) - \sum_{k=1}^{\infty} p_k (1-u)^k$ and with initial condition $u(0,x)=1-f(x)$.

The following proposition describes the asymptotic behaviour of solutions of the F-KPP equation for large times.
\begin{propositionb}[\citen{B_C, 1to6}]
\label{prop:FKPP_tail_estimate}
Let $u$ be a solution to the F-KPP equation with initial condition satisfying
\begin{thmlist}[label=(\roman*)]
	\item $0 \le u(0,x) \le 1$; \label{prop:FKPP_tail_estimate1}
	\item $\exists \, h > 0\colon \lim \sup_{t \to \infty} \frac{1}{t} \log \left( \int_{t}^{t(1+h)} u(0,y) \d y \right) \le  -\sqrt{2}$;
	\item $\exists \, c > 0, M >0, N > 0\colon \int_{x}^{x+N} u(0,y) \d y > c\quad \forall \, x \le -M$;
	\item $\int_{0}^{\infty} u(0,y) y \eee^{2y} \d y < \infty$. \label{prop:FKPP_tail_estimate4}
\end{thmlist}
Then we have, for any function $z \colon \R_+ \to \R_+$ such that $\lim_{t \to \infty} z(t)/t = 0$,
\begin{equation}
	\lim_{t \uparrow \infty} \eee^{\sqrt{2}z(t)} \eee^{(z(t))^2/2t} (z(t))^{-1} u \left( t, z(t) + \sqrt{2}t - \tfrac{3}{2\sqrt{2}} \log(t) \right) = C,
\end{equation}
where $C$ is a strictly positive constant depending on the initial condition $u(0,\cdot)$.
\end{propositionb}
\vspace{0pt}
\begin{lemmab} [\citen{B_C}, Proposition 8.2] \label{lem:F-KPP_finite}
For $z \colon \R_+ \to [1, \infty)$ and $t$ large enough, we have
\begin{align}
	\P \left( \max_{j \le n(t)} x_j(t)> z(t) + \sqrt{2}t - \sfrac{3}{2\sqrt{2}} \log(t) \right) 
	\le C' z(t) \eee^{- \sqrt{2} z(t) },
\end{align}
where $C'$ is a strictly positive constant independent of $t$.
\end{lemmab}

%

Particles of standard BBM are unlikely to cross the barrier function with slope $\pm\! \sqrt{2}$.
\begin{lemma}\label{lem:BarrierLok}
For any $\e > 0$, there exists $r_0 < \infty$ such that for all $r > r_0$, for all $t$ large enough,
\begin{align}\label{eq:barrierSqrt2}
	\P \, \bigg( 
	\exists_{j \le n(t)}  \exists_{ s \in [r, t-r]} \colon 
	| x_j(s) |
	> \sqrt{2} s
	\bigg) < \e,
\end{align}
\end{lemma}
\begin{proof}
In the proof of the convergence of the derivative martingale in \cite{LS}, it is pointed out that
\begin{align}
	\liminf_{t\uparrow\infty}	\min_{j\leq n(t)} \kl{\sqrt{2}t - x_j(t)} \uparrow \infty, \text{ a.s.},
\end{align} 
which implies \eqref{eq:barrierSqrt2}.
\end{proof}

We state a version of Slepian's lemma \cite{Sl} adapted to variable speed BBM.
\begin{lemma}[Slepian's Lemma]\label{lem:slepian}
Let $(\hat{x}_k(t))_{k \leq \hat{n}(t)}$ and $(\bar{x}_k(t))_{k \leq \bar{n}(t)}$ be the particle positions at time $t$ of variable speed BBMs with speed functions $\hat{A}$ and $\bar{A}$.\\
If $\hat{A\,} \leq \bar{A\,}$\!, then
\begin{equation}
	\P\left(\max_{k\leq \hat{n}(t)}\hat{x}_k(t)> y\right) \geq  \P\left(\max_{k\leq \bar{n}(t)}\bar{x}_k(t)> y\right).
\end{equation}
\end{lemma}
\begin{proof}
This follows from \cite[Corollary 3.10]{bbm-book}.
\end{proof}

We also recall two basic facts about  Brownian bridges.
We denote by $\zet_{a,b}^t$ a Brownian bridge starting in~$a$ and ending in~$b$ at time~$t$.
\vspace{2pt}
\begin{lemmab}[\citen{BH14.1}, Lemma 2.2] \label{lem:BB_fluc}
For any $\gamma > 1/2$ and for any $\e > 0$, there exists a constant $r > 0$ such that
\begin{equation}
	\lim_{t \uparrow \infty}
	\P \left( \forall_{r \le s \le t-r} \colon
	\vert \zet_{0,0}^t (s) \vert < (s \land (t-s) )^{\gamma}\right) > 1 - \e.
\end{equation}
\end{lemmab}
\vspace{6pt}
\begin{lemmab}[\citen{B_C}, Lemma 2.2] \label{lem:BB_line}
For any $x,y > 0$ holds
\begin{equation}
	\P \left( \forall_{0 \le s \le t} \colon
	\zet_{0,0}^t (s) \le (sx + (t-s)y/t) \right) 
	= 1 - \eee^{-2xy/t} 
	\leq 2\frac{xy}{t},
\end{equation}
and asymptotic equality holds if $xy = o(t)$.
\end{lemmab}
The next lemma allows to restrict events related to  maxima to \emph{likely} events.
\vspace{2pt}
\begin{lemmab}[\citen{1to6}, Lemma 3.4] \label{lem:localisation}
Let $x_j, j=1,\dots,n,$ be path-valued random variables and $\LL$ be an event on the set of paths such that, for any $\e > 0$,
\begin{equation}
	\P \left( \exists_{j \le n} \colon \{ x_j(t) > y\} \land \{ x_j \in \LL \} \right)
	\ge \P \left( \exists_{j \le n} \colon x_j(t) > y \right) - \e.
\end{equation}
Then
\begin{equation}
	\bigg\vert
	\P \, \bigg( \max_{j \le n} x_j(t) \le y \bigg)
	- \P \, \bigg( \max_{j \le n: x_j \in \LL} x_j(t) \le y \bigg)
	\bigg\vert \le \e.
\end{equation}
\end{lemmab}

\section{Localisation of paths}	\label{sec:localisation}
An essential step in the proof of Theorem~\ref{thm:lawofmax} is the control of the particle positions until time $a_{\ell-1}t$.
\subsection{Localisation of paths in Case A} 
\label{sec:localisationA}\label{sec:above_localisation}
First, we prove localisation results for a VSBBM $(\tilde{X}_t)_{t>0}$  with speed functions $(A_t)_{t>0}$ satisfying Assumption~\ref{as:above}. We introduce the following subsets of the space of continuous paths $X:\R_+ \to \R$.
\begin{align}
\AA_{r_1, r_2, S, \s}
&\equiv \left\{ X \colon \forall_{r_1 \le s \le r_2} \colon X(s) + S \le \sqrt{2} \s s \right\}, \nonumber \\
\GG_{s, k, S, B, D}
&\equiv \left\{  X \colon X(s) - \sqrt{2} \sigma_k s + S 
\in \sfrac{\left[ -B, -D \right] }{\s_k - \s_{k+1}}
\right\},
\end{align}
where $0 \le r_1 \le r_2 \le t$, $S \in \R$, $B > D > 0$ and $\s > 0$.
The set $\AA_{r_1, r_2, S, \s}$ describes shifted paths which do not exceed a linear barrier function on a given time interval.
Paths which lie in a certain interval at a given time are described by $\GG_{s, k, S, B, D}$. 
Figure~\ref{fig:Localisation_above} illustrates the localisation results of this subsection in the case of three-speed VSBBM. 
\begin{figure}
\includegraphics[page=1,width = 12cm]{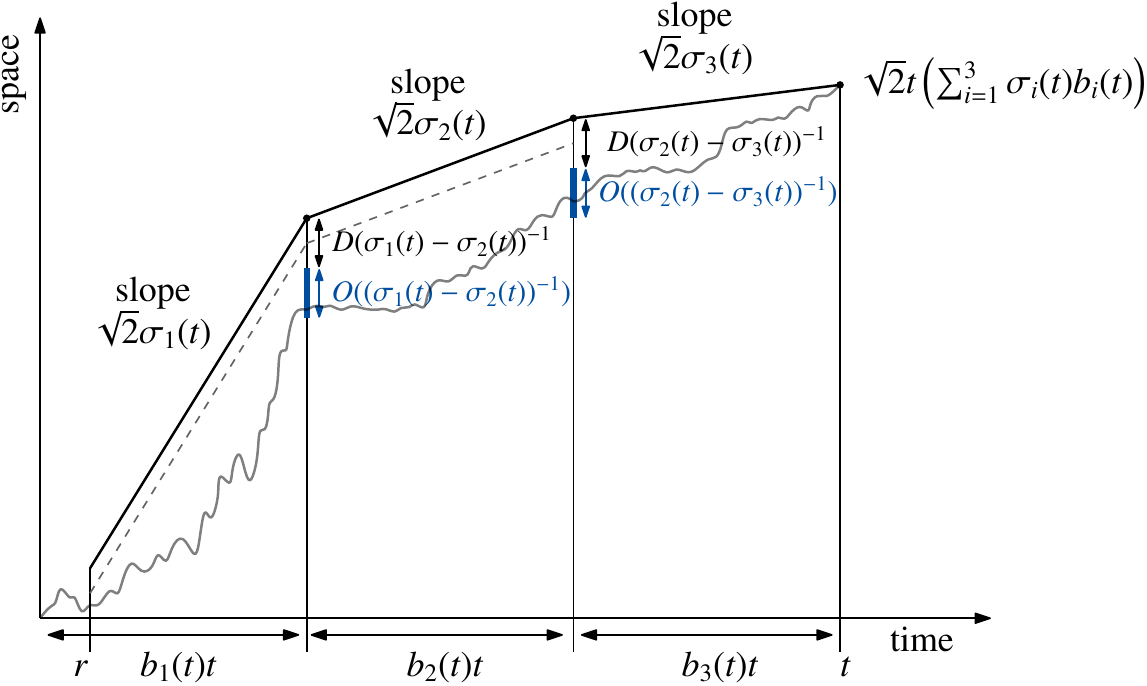}
\caption{Localisation of an extremal particle of three-speed BBM in Case A. 
	The dashed line depicts the effect of entropic repulsion.
} \label{fig:Localisation_above}

\end{figure}


As a first step, we show that all BBM particles stay below a certain barrier function.
\begin{proposition}[Barrier function] \label{prop:barrier}
For any $\e > 0$, there exists $r_0 < \infty$ such that for all $r > r_0$, for all $t$ large enough,
\begin{align} \label{eq:prem_barrier_prop}
	\P  \left( 
	\exists_{j \le n(t)} \exists_{1 \le k < \ell} \exists_{ s \in S_k} \colon 
	\tilde{x}_j(a_{k-1}t + s)
	> \sqrt{2} \sum_{i=1}^{k-1} \s_i b_i t +\sqrt 2 \s_k s 
	\right) < \e,
\end{align}
where $S_1 = [r, b_1t]$ and $S_k = [0,b_kt]$, for $1 < k < \ell$.
\end{proposition}
With the notation introduced in \eqref{eq:multiindex_notation}, the proposition implies that, with high probability, for all $1 < k < \ell$,
$\tilde{x}_{i_1} \in \AA_{r, b_1t, 0, \s_1}$ and 
$\multix{k} \in \AA_{0, b_kt, - \Lambda^{\mi_{k - 1}} (t),  \s_k}$, 
where 
\begin{equation} \label{eq:Lambda_def}
\Lambda^{\mi_{k - 1}}(t) \equiv \sum_{j=1}^{k-1} \sqrt{2} \s_j b_j t - \multix{j}(b_j t).
\end{equation}

\begin{proof}
By a union bound, the probability in \eqref{eq:prem_barrier_prop} is not larger than
\begin{equation} \label{eq:above_barrier_union}
	\sum_{k=1}^{\ell-1} \P \left( \exists_{j \le n(t)} \exists_{s \in S_k} \colon
	\tilde{x}_j( a_{k-1}t + s) > f(a_{k-1}t + s) \right),
\end{equation}
where $f( a_{k-1}t + s) = \sqrt{2} \sum_{i=1}^{k-1} \s_i b_i t + \sqrt{2} \s_k s$, for $s \in S_k, 1 \le k < \ell$. 
By Lemma~\ref{lem:BarrierLok}, the term with $k=1$
satisfies
\begin{align} \label{eq:barrier_first_summand}
	\P \left( \exists_{j \le n(t)} \exists_{s \in S_1} \colon
	\tilde{x}_j( a_{0}t + s) > f(a_{0}t + s) \right)&=\P \left( \exists_{j \le n(t)} \exists_{s \in [r, b_1t]} \colon
	\s_1 x_j(s) > \sqrt{2} \s_1 s \right)\nonumber\\
	&\leq \e /( 2^{\ell-2} \ell),
\end{align}
for any $\e>0$, for $r$ and $t$ large enough.
Controlling  the  terms with $k>1$  in the sum, follows exactly the lines of the 
proof of Theorem 2.2 in \cite{ABK_G}, namely
controlling the paths first at integer times and then showing that fluctuations between those times are small.  We repeat the calculations for completeness.

We decompose that term with $k=2$ in \eqref{eq:above_barrier_union}  as
\begin{align} \label{eq:above_barrier_T1_T2}
	&\P \left( \exists_{j \le n(t)} \exists_{ s \in [0, b_2t] } \colon
	\tilde{x}_j( a_1t + s) > f( a_1t + s) \right)  		\equiv \, (T1) + (T2),\nonumber  
\end{align}
where,  for an arbitrary constant $\hat{C} > 0$, 
\begin{align}
	(T1)= \, & \P \, \bigg( \exists_{ s \in [a_1t, a_2t] } \colon
	\max_{j \le n(s)} \tilde{x}_j(s) > f(s),
	\max_{j \le n( \okl{s} )} \tilde{x}_k( \okl{s}) > f( \okl{s}) - \hat{C} \bigg),  \nonumber  \\
	(T2)= \, &\P \, \bigg( \exists_{ s \in [a_1t, a_2t] } \colon
	\max_{j \le n(s)} \tilde{x}_j(s) > f(s),
	\max_{j \le n( \okl{s} )} \tilde{x}_k( \okl{s}) \le f( \okl{s}) - \hat{C} \bigg).  	\end{align}
By \eqref{eq:barrier_first_summand}, $(T1)$ is bounded from above by
\begin{equation} \label{eq:above_barrier_T1}
	\P \left( \exists_{ s \in [a_1t, a_2t] } \colon \hspace{-0.5em}
	\max_{j \le n( \okl{s} )} \tilde{x}_k( \okl{s}) > f( \okl{s}) - \hat{C};
	\forall_{ s' \in [r, a_1t] } \colon  \hspace{-0.5em}
	\max_{j \le n(s')} \tilde{x}_j(s') < f(s') \right)
	+ \frac{ \e }{ 2^{\ell-2} \ell }.
\end{equation}
The probability in \eqref{eq:above_barrier_T1} 
can be expressed in terms of standard BBMs $x$ as
\begin{align} \label{eq:above_barrier_Markov}
	&\P \, \bigg( \exists_{i_1 \le n(b_1t)} \colon
	\left\{ \forall_{ s' \in [r, b_1t] } \colon x_{i_1}(s') < \sqrt{2} s' \right\}  \nonumber  \\
	& \quad \land \bigg\{ \exists_{ s \in [0, b_2t]} \colon
	\max_{i_2 \le n^{i_1}( \okl{s} )} x_{i_2}^{i_1} ( \okl{s} ) 
	> \sqrt{2} \okl{s} + \sfrac{ \s_1 }{ \s_2 } \left( \sqrt{2} b_1t - x_{i_1}(b_1t) \right) - \hat{C} \bigg\} \bigg), 
\end{align}
which in turn, by Markov's inequality and the many-to-one lemma, is not larger than
\begin{align} \label{eq:above_barrier_Markov2}
	&\eee^{b_1t} \, 
	\E \, \bigg[ \1_{ \{ \forall s' \in [r, b_1t] \colon x(s') < \sqrt{2} s' \}} \nonumber \\
	& \qquad \ \times
	\P \, \bigg( \exists_{ s \in [0, b_2t]} \colon
	\max_{i_2 \le n( \okl{s} )} x_{i_2} ( \okl{s} ) 
	> \sqrt{2} \okl{s} + \sfrac{ \s_1 }{ \s_2 } \left( \sqrt{2} b_1t - x(b_1t) \right) - \hat{C}
	\, \Big\vert \, x \bigg) \bigg].
\end{align}
Writing $x(s') = \frac{s'}{b_1t} x(b_1t) + \zet_{0,0}^{b_1t}(s')$ for $ s' \in [0,b_1t]$ and using that $\zet_{0,0}^{b_1t}$ is independent of $x(b_1t)$, the right-hand side of \eqref{eq:above_barrier_Markov2} is equal to
\begin{align} \label{eq:above_barrier_T1_o}
	&\eee^{b_1t}
	\int_{- \infty}^{ \sqrt{2} b_1t }
	\sfrac{ \d \o}{\sqrt{2 \pi b_1 t} }
	\exp \left( - \sfrac{ \o^2 }{2 b_1 t} \right)
	\P \left( \zet_{0, \o}^{b_1t}(s') \le \sqrt{2}s \, \forall s' \in [r, b_1t] \right)  \nonumber  \\
	&\qquad \times
	\P \, \bigg( \exists_{ s \in [0, b_2t]} \colon
	\max_{i_2 \le n ( \okl{s} )} x_{i_2} ( \okl{s} ) 
	> \sqrt{2} \okl{s} + \sfrac{ \s_1 }{ \s_2 } \left( \sqrt{2} b_1t - \o \right) - \hat{C} \bigg).
\end{align}
Changing  variables $\o = \sqrt{2}b_1t - z$ and using that $\eee^{-z^2/(2b_1t)}\leq 1$, \eqref{eq:above_barrier_T1_o} is not larger than
\begin{equation} \label{eq:above_barrier_T1_z}
	\int_0^{\infty}  \!
	\sfrac{ \d z \, \eee^{\sqrt{2}z }}{\sqrt{2 \pi b_1t} }
	\P \!\left( \zet_{0, -z}^{b_1t}(s') \le 0 \, \forall s' \in [r, b_1t] \right) 
	\P \!\, \bigg( \exists_{ s \in [0, b_2t]}:\!
	\max_{i_2 \le n( \okl{s} )} x_{i_2}( \okl{s} ) 
	> \sqrt{2} \okl{s} + \sfrac{ \s_1 }{ \s_2 } z - \hat{C} \bigg)\!. 
\end{equation}
By Lemma 3.4 in \cite{ABK_G}, 
\begin{align} \label{eq:above_barrier_BB_bound}
	&\P \left( \zet_{0, -z}^{b_1t}(s') \le 0 \,  \forall s' \in [r, b_1t] \right)
	\leq  \frac{2z\sqrt r}{b_1t-r}.
\end{align}

The second probability in \eqref{eq:above_barrier_T1_z} is bounded from above by	
\begin{align} \label{eq:above_barrier_T1_union}
	\sum_{j=1}^{\okl{b_2t}}
	\P \, \bigg(
	\max_{i_2 \le n( j )} x_{i_2}( j) 
	> \sqrt{2} j + \sfrac{ \s_1 }{ \s_2 } z - \hat{C} \bigg).
\end{align}
We choose $j^* \in \{1, \dots, \okl{b_2t}\}$ such that we can apply Lemma~\ref{lem:F-KPP_finite} to the probabilities in \eqref{eq:above_barrier_T1_union} for $j>j^*$ when $\sfrac{3}{2 \sqrt{2}} \log(j) + \sfrac{ \s_1}{\s_2} z - \hat{C} \ge 1$.
If $j \le j^*$ and $\sfrac{3}{2 \sqrt{2}} \log(j) + \sfrac{ \s_1}{\s_2} z - \hat{C} \ge 1$, we bound these probabilities by Lemma 3.4 in \cite{1to6}, otherwise we simply bound them by one.

Therefore, \eqref{eq:above_barrier_T1_union} is smaller than
\begin{align} \label{eq:above_barrier_T1_FKPP} 
	&\eee^{ - \sqrt{2} \frac{ \s_1}{\s_2} z }
	\eee^{ \sqrt{2} \hat{C} }
	\bigg(
	\sfrac{1}{2\sqrt{\pi}} j^*
	+ \sum_{j=j^*+1}^{ \okl{b_2t} }
	j^{-3/2}
	C' \!\left( \sfrac{3}{2 \sqrt{2}} \log(j) + \sfrac{ \s_1}{\s_2} z - \hat{C} \right)\! \bigg)
	\1_{  \sfrac{ \s_1}{\s_2} z > 1+\hat{C}-\sfrac{3}{2 \sqrt{2}} \log(j) } \nonumber\\
	&+ \sum_{j=1}^{ \okl{b_2t} }\1_{  \sfrac{ \s_1}{\s_2} z < 1+\hat{C}-\sfrac{3}{2 \sqrt{2}} \log(j) }. 
\end{align}
We bound the first indicator function by one. Inserting the estimates from \eqref{eq:above_barrier_BB_bound} and \eqref{eq:above_barrier_T1_FKPP} into \eqref{eq:above_barrier_T1_z}, \eqref{eq:above_barrier_T1_z} is bounded form above by
\begin{align} \label{eq:above_barrrier_T1_postFKPP}
	&\sfrac{2 \sqrt{r}}{ (b_1t)^{1/2} (b_1t - r)}
	C' \eee^{\sqrt{2} \hat{C} }
	\sum_{j=1}^{ \okl{b_2t} }
	j^{-3/2}
	\int_0^{\infty} 
	\d z \,
	\eee^{-\sqrt{2} z \frac{ \s_1 - \s_2}{ \s_2} }
	z \left( \sfrac{ \s_1}{\s_2} z + \sfrac{3}{2 \sqrt{2}} \log(j) - \hat{C} 
	+ \sfrac{j^{3/2}}{2 \sqrt{\pi}} \1_{j \le j^*} \right) \nonumber \\
	&+ \,
	\sfrac{2 \sqrt{r}}{ (b_1t)^{1/2} (b_1t - r)}
	\sum_{j=1}^{\okl{b_2t}}
	\int_0^{ \min \left\{ 0, \frac{\s_2}{\s_1} \left( 1+\hat{C} - \frac{3}{2\sqrt{2}} \log(j) \right) \right\} }
	\d z \, 
	\eee^{ \sqrt{2}z} \, z. 
\end{align}
The sum in the second line has only finitely many summands, each of which is bounded by a constant. Thus, 
the second line is bounded from above by a constant times $\frac{ \sqrt{r} }{ (b_1t)^{1/2} (b_1t - r) }$, which converges to zero, as $t\uparrow\infty$.
For the first line, we note
\begin{equation} \label{eq:above_barrier_integrals}
	\int_0^{\infty} 
	\d z \,
	\eee^{-\sqrt{2} z \frac{ \s_1 - \s_2}{ \s_2} }
	z^2
	+ \int_0^{\infty} 
	\d z \,
	\eee^{-\sqrt{2} z \frac{ \s_1 - \s_2}{ \s_2} }
	z
	= 2 \sfrac{ \s_2^3}{ ( \s_1 - \s_2)^3 }
	+ \sfrac{ \s_2^2}{ ( \s_1 - \s_2)^2}.
\end{equation}
Since $\log(j) j^{-3/2}$ is summable, the first term in \eqref{eq:above_barrrier_T1_postFKPP} is bounded from above by
\begin{equation}
	\sfrac{ \sqrt{r} }{ (b_1t)^{1/2} (b_1t - r) } \OO \left( (\s_1 - \s_2)^{-3}  \right).
\end{equation}
By Assumption~\ref{as:above3}, $ (b_1t)^{3/2} ( \s_1 - \s_2 )^3\uparrow \infty$, as $t \uparrow \infty$. 
Therefore, $(T1)$ can be made smaller than $\e / ( 2^{\ell-2} \ell)$
for fixed $\hat{C}$ and $t$ sufficiently large.	

%
%

It remains to control $(T2)$.
By monotonicity,
\begin{equation} \label{eq:above_barrier_T2}
	(T2)
	\le \P  \biggl( \exists_{s \in [r, a_2t]} \colon
	\max_{j \le n(s)} \tilde{x}_j(s) > f(s),
	\max_{j \le n( \okl{s})} \tilde{x}_j( \okl{s}) \le f( \okl{s} ) - \hat{C}
	\biggr).
\end{equation}
We define stopping times $\SS \equiv \inf \{ s \in [r, a_2t] \colon \max_{j \le n(s)} \tilde{x}_j(s) > f(s) \}$.
By conditioning on~$\SS$,
\begin{align} \label{eq:above_barrier_S_integral}
	(T2) \le
	\sum_{k = \okl{r}-1 }^{ \okl{a_2t} }
	\int_k^{k+1}
	\P   ( \SS \in \d s ) \, 
	\P \, \bigg( \max_{j \le n( \okl{s})} \tilde{x}_j( \okl{s}) \le f( \okl{s} ) - \hat{C}
	\, \Big\vert \, \SS=s \bigg).
\end{align}
The conditional probability is bounded by the probability that the offspring of the particle that, at time $s$, was maximal, are, at time $\okl{s}$, all smaller by $f( \okl{s}) - f(s) - \hat{C}$.
The distribution of the offspring positions  depends on~$s$ and ~$\okl{s}$, but, since $\s_1 > \s_2$,  for all $s$ and $\okl{s}$,  the conditional probability is not larger than
\begin{equation} \label{eq:above_barrier_cond_1}
	\P  \left( \max_{ i \le n( \okl{s} - s )} \s_2 x_i ( \okl{s} - s ) \le \sqrt{2} \s_1 ( \okl{s} - s ) - \hat{C} \right).
\end{equation}
Since 
$\{ \forall i \le n(r): \s_2 x_i (r) \le \sqrt{2} \s_1 r - \hat{C} \}
\subseteq \{ \exists i \le n(r): \s_2 x_i (r) \le \sqrt{2} \s_1 (r) - \hat{C} \}$ for all $r>0$, we get by Markov's inequality and the many-to-one lemma that \eqref{eq:above_barrier_cond_1}
is bounded by
\begin{equation} \label{eq:above_barrier_T2_final}
	\eee^{ \okl{s} - s } \, \P \left( x ( \okl{s} - s ) \le \sqrt{2} \sfrac{ \s_1}{\s_2} ( \okl{s} - s ) - \sfrac{ \hat{C} }{\s_1} \right),
\end{equation}
where $x ( \okl{s} - s )$ is a Gaussian random variable with mean zero and variance $ \okl{s} - s$. 
Since $ \okl{s} - s \le 1$, \eqref{eq:above_barrier_T2_final} is smaller than $\e / (2^{\ell-2} \ell)$ if  $\hat{C}$ is large enough.
In conclusion, we have shown that, for $t$ large enough,
\begin{equation}
	\P \left( \exists_{j \le n(t)} \exists_{ s \in [a_1t, a_2t] } \colon
	\tilde{x}_j(s) > f(s) \right)
	< 3 \e / (2^{\ell-2} \ell).
\end{equation}
Each summand in \eqref{eq:above_barrier_union} can be bounded by $(k+1) \e/ (2^{\ell-k} \ell )$ by the same reasoning as for the second summand.
Then \eqref{eq:above_barrier_union} is smaller than $(\ell - 1) \e / \ell$, which concludes the proof.
\end{proof}
The next proposition states that ancestors of extremal particles at time~$t$ are at the time~$a_kt$ of order $ ( \s_k - \s_{k+1})^{-1}$ below the barrier function.
\begin{proposition}[Position at speed changes] \label{prop:above_speed_change}
For any $y \in \R$ and for any $\e > 0$, there are constants $B > D > 0$ such that, for all $t$ large enough,
\begin{align}\label{eq:above_speed_change}
	\P \, \bigg( &\exists_{j \le n(t)} \colon \Big\{ \tilde{x}_j(t) > m^+(t)-y \Big\} 
	\land \Bigg\{ \exists _{1 \le k < \ell }\colon
	\tilde{x}_j(a_kt) - \sqrt{2}t \sum_{i=1}^{k} \s_i b_i 
	\not\in \sfrac{[ - B, - D ]}{\s_k - \s_{k+1}} \Bigg\} \Bigg) 
	< \e.
\end{align}
\end{proposition}
The assertion of the proposition can be restated as saying that
extremal particles have the property that 
$\tilde{x}_{i_1} \in \GG_{b_1t, 1, 0, B, D}$ and 
$\multix{k} \in \GG_{b_kt, k, - \Lambda^{\mi_{k - 1}}(t), B, D}$, with high probability, 
for all $1 < k < \ell$. 
\begin{proof}
By a union bound, the probability in \eqref{eq:above_speed_change} is not larger than
\begin{align} \label{eq:above_speed_union_bound}
	\sum_{k=1}^{\ell-1} 
	\P \, \bigg( \exists_{j \le n(t)} \colon \Big\{ \tilde{x}_j(t) > m^+(t)-y \Big\} 
	\land \Bigg\{ \tilde{x}_j(a_kt) - \sqrt{2}t \sum_{i=1}^{k} \s_i b_i 
	\not\in \sfrac{[ - B, - D ]}{\s_k - \s_{k+1}} \Bigg\} \Bigg).
\end{align}

We prove in detail how the summand for $k=\ell-1$ can be bounded. We then explain how the arguments can be transferred to estimate the other summands.

By Proposition~\ref{prop:barrier}, we can introduce the barrier condition  into the probabilities in \eqref{eq:above_speed_union_bound}. The summand for $k=\ell-1$ is then 
is equal to 
\begin{align} \label{eq:above_speed_special_localisation}
	\P \, \bigg( 
	&\exists_{j \le n(t)} \colon \Big\{ \tilde{x}_j(t) > m^+(t)-y \Big\}
	\land \, \bigg\{
	\tilde{x}_j (a_{\ell-1} t) - \sqrt{2}t \sum_{i=1}^{\ell-1} \s_i b_i 
	\not\in \sfrac{[ - B, - D ] }{\s_{\ell-1} - \s_{\ell}} 
	\bigg\} \nonumber \\
	&\land \, \Bigg\{ 
	\forall _{1 \le k < \ell-1 }\colon
	\tilde{x}_j(a_kt) - \sqrt{2}t \sum_{i=1}^{k} \s_i b_i 
	\in \sfrac{[ - B, - D ]  }{\s_k - \s_{k+1}} 
	\Bigg\} \nonumber \\
	&\land \, \Bigg\{ 
	\forall _{1 \le k < \ell} \forall_{s \in S_k} \colon 
	\tilde{x}_j(a_{k-1}t + s)
	\le r \lor  \sqrt{2} \, \bigg( \sum_{i=1}^{k-1} \s_i b_i t + \s_k s \bigg) 
	\Bigg\} \Bigg)+\OO(\e).
\end{align}
By the many-to-one lemma, the independence of the increments, and the barrier conditions to write everything in a more accessible product form,  this is bounded from above by
\begin{align}\label{lisa.100}
	&\eee^{\sum_{k=1}^{\ell-1} b_k t} \, 
	\E \, \Bigg[
	\1_{ \left\{ \forall _{1 \le k < \ell} \forall_{s \in S_k} \colon 
		\sum_{i=1}^{k-1} \s_ix_i(b_it)+\s_k x_k(s)
		\le r \lor  \sqrt{2} \, \left( \sum_{i=1}^{k-1} \s_i b_i t + \s_k s \right) \right\} } \nonumber\\
	& \times
	\1_{ \left\{ \forall _{1 \le k < \ell-1 }\colon
		\sum_{i=1}^{k}\s_i x_i(b_it) - \sqrt{2}t \sum_{i=1}^{k} \s_i b_i 
		\in \sfrac{[ - B, - D ]  }{\s_k - \s_{k+1}} \right\} }
	\1_{ \left\{ \sum_{i=1}^{\ell-1}\s_ix_i(b_it) - \sqrt{2}t \sum_{i=1}^{\ell-1} \s_i b_i 
		\not\in \sfrac{[ - B, - D ] }{\s_{\ell-1} - \s_{\ell}} \right\}} \nonumber\\
	&\times
	\P \, \bigg( \max_{i_{\ell} \le n (b_{\ell}t)} \tilde{x}_{i_{\ell}}(b_{\ell}t)
	> m^+(t) - \sum_{k=1}^{\ell-1} \s_k x_{k} (b_{k}t) - y
	\, \bigg \vert \, \tilde{F}_{\ell-1} \Bigg) \Bigg],
\end{align}
where $x_k, 1 \le k \le \ell-1,$ denote Brownian motions and $\tilde{F}_{\ell-1} = \s (x_k, 1 \le k \le \ell-1)$.
Decomposing  $x_k$ into a Brownian bridge and its endpoint, \eqref{lisa.100} equals 
\begin{align} \label{eq:above_speed_integral_chain_1}
	\eee^{\sum_{k=1}^{\ell-1} b_k t} 
	&\int_{I_1} 
	\sfrac{\d \o_1}{\sqrt{2 \pi \s_1^2 b_1t}}
	\eee^{- \frac{\o_1^2}{2 \s_1^2 b_1 t}} 
	\int_{I_2} \dots 
	\int_{I_{\ell-2}} 
	\sfrac{\d \o_{\ell-2}}{\sqrt{2 \pi \s_{\ell-2}^2 b_{\ell-2} t}}
	\eee^{- \frac{\o_{\ell-2}^2}{2 \s_{\ell-2}^2 b_{\ell-2} t}}
	\nonumber \\
	\times \, & 
	\int_{\bar I_{\ell-1}} 
	\sfrac{\d \o_{\ell-1}}{\sqrt{2 \pi \s_{\ell-1}^2 b_{\ell-1} t}}
	\eee^{- \frac{\o_{\ell-1}^2}{2 \s_{\ell-1}^2 b_{\ell-1} t}}
	\P \, \bigg( \max_{i_{\ell} \le n (b_{\ell}t)} \tilde{x}_{i_{\ell}} (b_{\ell}t)
	> m^+(t) - \sum_{k=1}^{\ell-1} \o_k - y \Bigg)
	\nonumber \\
	\times \, &\prod_{k=1}^{\ell-1}
	\P \, \bigg( \zet_{
		- \left( \sum_{i=1}^{k-1} \sqrt{2} \s_i b_i t - \o_i \right) ,\,
		\o_k - \left( \sum_{i=1}^{k-1} \sqrt{2} \s_i b_i t - \o_i \right) }^{|S_k|}
	(s) \le \sqrt{2} \s_k s \, \forall s \in S_k \bigg),
\end{align}
where, for $1 \le k \le \ell-1$,
\begin{align} \label{eq:above_speed_Ik}
	I_k 
	&= \sqrt{2} \s_k b_k t + \sum_{i=1}^{k-1} \left( \sqrt{2} \s_i b_i t - \o_i \right) -  \sfrac{ [D, B ]}{\s_k - \s_{k+1}},
	\nonumber \\
	\bar{I}_{\ell-1}
	&= \left(- \infty,\sqrt{2} \s_{\ell-1} b_{\ell-1}t 
	+  \sum_{i=1}^{\ell-2} \left( \sqrt{2} \s_i b_i t - \o_i \right) \right)\setminus I_{\ell-1}.
\end{align}
Changing variables 
$\o_k = - z_k + \sqrt{2} \s_k b_k t + z_{k-1} \1_{k \ge 2}$ for $1 \le k \le \ell-1$ and setting $z_0=0$,
\eqref{eq:above_speed_integral_chain_1}  equals
\begin{align} \label{eq:above_speed_z}
	\eee^{\sum_{k=1}^{\ell-1} b_k t}
	&\int_{J_1} 
	\sfrac{\d z_1}{\sqrt{2 \pi \s_1^2 b_1t}}
	\eee^{- \frac{ \left(\sqrt{2} \s_1 b_1 t - z_1 \right)^2}{2 \s_1^2 b_1 t}} \
	\int_{J_2} \dots
	\int_{J_{\ell-2}} 
	\sfrac{\d z_{\ell-2}}{\sqrt{2 \pi \s_{\ell-2}^2 b_{\ell-2} t}}
	\eee^{- \frac{ \left( \sqrt{2} \s_{\ell-2} b_{\ell-2} t + z_{\ell-3} - z_{\ell-2} \right)^2}{2 \s_{\ell-2}^2 b_{\ell-2} t}} \nonumber \\
	\times &\int_{\bar{J}} 
	\sfrac{\d z_{\ell-1}}{\sqrt{2 \pi \s_{\ell-1}^2 b_{\ell-1} t}}
	\eee^{- \frac{ \left( \sqrt{2} \s_{\ell-1} b_{\ell-1} t + z_{\ell-2} - 
			z_{\ell-1} \right)^2}{2 \s_{\ell-1}^2 b_{\ell-1} t}}
	\prod_{k=1}^{\ell-1}
	\P \, \bigg( \zet_{
		- \frac{ z_{k-1} }{\s_k}, - \frac{z_k }{\s_k}}^{| S_k|} 
	(s) \le 0 \, \forall s \in S_k \bigg) \nonumber \\
	\times \, &\P \, \bigg( \max_{i_{\ell} \le n (b_{\ell}t)} \tilde{x}_{i_{\ell}} (b_{\ell}t)
	> m^+(t) + z_{\ell-1} - \sqrt{2} t  \sum_{k=1}^{\ell-1} \s_k b_k - y \Bigg), 
\end{align}
where, for all $1 \le k \le \ell-1$,
\begin{equation} \label{eq:above_speed_Jk}
	J_k
	= \sfrac{[ D,B] }{ \s_k - \s_{k+1} }, \quad
	\bar{J}
	= (0, \infty) \setminus J_{\ell}.
\end{equation}
The product of all exponential functions in \eqref{eq:above_speed_z} is equal to
\begin{align} \label{eq:above_speed_expo}
	&\exp \left(  \sum_{k=1}^{\ell-1} b_k t 
	- \frac{ (\sqrt{2} \s_1 b_1 t - z_1)^2}{2 \s_1^2 b_1 t}
	- \sum_{k=2}^{\ell-1} \frac{ ( \sqrt{2} \s_k b_k t + z_{k-1} - z_k )^2}{2 \s_k^2 b_k t} \right)  \nonumber  \\
	&\le \exp \left( - \sqrt{2} \sum_{k=1}^{\ell-2} z_k \frac{\s_k - \s_{k+1}}{\s_k \s_{k+1}}
	+ \sqrt{2} \frac{z_{\ell-1}}{\s_{\ell-1}}
	\right).
\end{align}
For $1 \le k \le \ell-1$,  \eqref{eq:above_speed_z} is not larger than
\begin{align} \label{eq:above_speed_zzz}
	\int_D^B& 
	\sfrac{\d y_1}{\sqrt{2 \pi b_1t}}
	\eee^{-\sqrt{2} \frac{y_1}{ \s_1 \s_2}}
	\int_D^B \dots
	\int_D^B 
	\sfrac{\d y_{\ell-2}}{\sqrt{2 \pi b_{\ell-2} t}}
	\eee^{- \sqrt{2} \frac{y_{\ell-2}}{ \s_{\ell-2} \s_{\ell-1}} } 
	\int_{\substack{(0,D) \\ \cup (B, \infty)}} 
	\sfrac{\d y_{\ell-1}}{\sqrt{2 \pi b_{\ell-1} t}}
	\eee^{ \sqrt{2} \frac{ y_{\ell-1} }{ \s_{\ell-1} (\s_{\ell-1} - \s_{\ell}) } }
	\nonumber \\
	\times &
	\prod_{k=1}^{\ell-1}
	\sfrac{1}{\s_k - \s_{k+1}}
	\P \, \bigg( \zet_{
		- \frac{ y_{k-1} }{ \s_k (\s_{k-1} - \s_k) },
		- \frac{ y_k }{ \s_k (\s_k - \s_{k+1}) }
	}^{|S_k|} 
	(s) \le 0 \, \forall s \in S_k \bigg) 
	\nonumber \\
	\times \, &\P \, \bigg( \max_{i_{\ell} \le n(b_{\ell}t)} \tilde{x}_{i_{\ell}}(b_{\ell}t)
	> m^+(t) +  \sfrac{ y_{\ell-1} }{ \s_{\ell-1} - \s_{\ell} } - \sqrt{2} t \sum_{k=1}^{\ell-1} \s_k b_k - y \Bigg).
\end{align}
By Lemma~\ref{lem:BB_line} and \eqref{eq:above_barrier_BB_bound}, 
\begin{align}  \label{eq:above_speed_BB_bound}
	\prod_{k=1}^{\ell-1}&
	\P \, \bigg( \zet_{
		- \frac{ y_{k-1} }{ \s_k (\s_{k-1} - \s_k) },
		- \frac{ y_k }{ \s_k (\s_k - \s_{k+1}) }
	}^{|S_k|} 
	(s) \le 0 \, \forall s \in S_k \bigg) \nonumber  \\
	\le & \ 2^{\ell-1} \sqrt{r}
	\frac{b_1t}{b_1t - r}
	\frac{y_{\ell-1}}{ (\s_{\ell-1} - \s_{\ell} ) b_{\ell-1}t \, \s_{\ell-1} }
	\prod_{k=1}^{\ell-2} 
	\frac{ y_k^2 }{ (\s_k - \s_{k+1})^2 b_kt \, \s_k \s_{k+1} }.
\end{align}
Therefore, \eqref{eq:above_speed_zzz} is bounded by
\begin{align} \label{eq:above_speed_y}
	&\sqrt{r} \left( \frac{2}{\pi} \right)^{(\ell-1)/2}
	\frac{b_1t}{b_1t-r}
	\prod_{k=1}^{\ell-2}
	\frac{1}{(\s_k - \s_{k+1})^3 (b_kt)^{3/2} \s_k^2 \s_{k+1} }
	\int_D^B
	\d y_1 
	\eee^{-\sqrt{2 } \frac{y_1}{\s_1 \s_2} } y_1^2
	\int_D^B
	\dots \nonumber \\
	&\times
	\int_D^B
	\d y_{\ell-2}
	\eee^{-\sqrt{2} \frac{y_{\ell-2}}{ \s_{\ell-2} \s_{\ell-1}} } y_{\ell-2}^2
	\int_{ \substack{ (0,D) \\ \cup (B, \infty)}}
	\d y_{\ell-1}
	\eee^{\sqrt{2} \frac{y_{\ell-1}}{\s_{\ell-1} (\s_{\ell-1} - \s_{\ell} )}}
	\sfrac{ y_{\ell-1} }{ (\s_{\ell-1} - \s_{\ell} )^2 (b_{\ell-1}t)^{3/2} \s_{\ell-1}^2} \nonumber \\
	&\times
	\P \, \bigg( \max_{i_{\ell} \le n(b_{\ell}t)} \tilde{x}_{i_{\ell}}(b_{\ell}t)
	> m^+(t) +  \sfrac{ y_{\ell-1} }{ \s_{\ell-1} - \s_{\ell} } - \sqrt{2} t \sum_{k=1}^{\ell-1} \s_k b_k - y \Bigg).
\end{align}
The probability of the maximum can be bounded using Lemma~\ref{lem:F-KPP_finite}.
We notice that 
\begin{align} \label{eq:above_speed_max_estimate_x}
	&\sfrac{1}{\s_{\ell}}
	\bigg( m^+(t) +  \sfrac{ y_{\ell-1} }{ \s_{\ell-1} - \s_{\ell} } - \sqrt{2}t \sum_{k=1}^{\ell-1} \s_k b_k - y \bigg)
	- \left( \sqrt{2} b_{\ell} t - \sfrac{3}{2 \sqrt{2}} \log(b_{\ell}t) \right) \nonumber \\
	&= \,  \sfrac{1}{\s_{\ell}}
	\left( \sfrac{ y_{\ell-1} }{\s_{\ell-1} - \s_{\ell} } +  L_{\ell-1}(t) - y \right),
\end{align}
where 
\begin{equation} \label{eq:above_def_L}
	L_{\ell-1}(t) \equiv 
	- \sfrac{3}{2 \sqrt{2}} \sum_{k=1}^{\ell-1} \big( \log (b_k t) + 2 \log ( \pi^{1/6} (\s_k - \s_{k+1})) \big).
\end{equation}
This implies that, for $ y_{\ell-1} > ( \s_{\ell-1} - \s_{\ell} ) (y + \s_{\ell} -  L_{\ell-1}(t))$,
\begin{align} \label{eq:above_speed_max_estimate}
	&\P \, \bigg( \max_{i_{\ell} \le n (b_{\ell}t)} \tilde{x}_{i_{\ell}} (b_{\ell}t)
	> m^+(t) + \sfrac{ y_{\ell-1} }{ \s_{\ell-1} - \s_{\ell} }  - \sqrt{2} t \sum_{k=1}^{\ell-1} \s_k b_k - d \Bigg) \nonumber \\
	&\le \, \sfrac{C'}{\s_{\ell}}
	\left( \sfrac{ y_{\ell-1} }{\s_{\ell-1} - \s_{\ell} } +  L_{\ell-1}(t) - y \right)
	\eee^{ - \sqrt{2} \frac{1}{\s_{\ell}}
		\left( \frac{ y_{\ell-1} }{\s_{\ell-1} - \s_{\ell} } +  L_{\ell-1}(t) - y \right)}. 
\end{align}
For $y_{\ell-1} \le ( \s_{\ell-1} - \s_{\ell} ) ( y + \s_{\ell} - L_{\ell-1}(t) )$, 
we bound the probability by 1.
In order to estimate the $y_{\ell-1}$-integral in \eqref{eq:above_speed_y},
we split the domain of integration \mbox{$(0,D) \cup (B, \infty)$} into two parts,
\begin{alignat}{2}
	&J_a
	\equiv \Big( 0, 
	( \s_{\ell-1} - \s_{\ell} ) ( y + \s_{\ell}- L_{\ell-1}(t) ) \Big), \quad
	&&J_b
	\equiv (0, \infty) \setminus (J_a \cup [D,B] ).
\end{alignat}
On  $J_a$, we bound the integrand by its maximum value to get
\begin{align}  \label{eq:above_speed_Ja}
	&\sfrac{1}{ ( \s_{\ell-1} - \s_{\ell} )^2 (b_{\ell-1}t)^{3/2} \s_{\ell-1}^2}
	\int_0^{( \s_{\ell-1} - \s_{\ell} ) \, \big( y + \s_{\ell} - L_{\ell-1}(t)\big)}
	\d y_{\ell-1}
	\eee^{\sqrt{2} \frac{y_{\ell-1}}{\s_{\ell-1} (\s_{\ell-1} - \s_{\ell} )}}
	y_{\ell-1} \nonumber \\
	&\le \, \sqrt{\pi} \e^{\sqrt{2}} (\s_{\ell-1} - \s_{\ell})^3 ( y + \s_{\ell}- L_{\ell-1}(t) )^2
	\eee^{- \sqrt{2} \, \big( L_{\ell-2}(t) - y \big) }
	(1+o(1)).
\end{align}
By Assumption~\ref{as:above3}, $( \s_{\ell-1} - \s_{\ell} )^3 ( y + \s_{\ell} - L_{\ell-2} (t))^2 \to 0$, as  $t \uparrow \infty$.
By \eqref{eq:above_speed_max_estimate}, the integral over $J_b$ is bounded by
\begin{align} \label{eq:above_speed_Jb}
	\sfrac{C' \sqrt{\pi}}{\s_{\ell-1}^2 \s_{\ell}}
	\eee^{- \sqrt{2} \, \big( L_{\ell-2}(t) - y \big) }
	\bigg(
	\int_0^D
	\!\d y_{\ell-1} 
	\eee^{- \sqrt{2} \frac{ y_{\ell-1} }{ \s_{\ell-1} \s_{\ell} }}
	y_{\ell-1}^2
	+ \!\int_B^{\infty}
	\!\d y_{\ell-1} 
	\eee^{- \sqrt{2} \frac{ y_{\ell-1} }{ \s_{\ell-1} \s_{\ell} }}
	y_{\ell-1}^2
	\bigg)
	(1+o(1)).
\end{align}
Both integrals are finite and converge to zero as $ D \downarrow 0, B \uparrow \infty$, respectively.
Inserting the upper bounds from \eqref{eq:above_speed_Ja} and \eqref{eq:above_speed_Jb} into \eqref{eq:above_speed_y} and recalling the definition of $L_{\ell-2}$, we see that \eqref{eq:above_speed_y} is not larger than
\begin{align} \label{eq:above_speed_final_bound}
	&C' \sqrt{r}
	2^{(\ell-1)/2}
	\eee^{\sqrt{2}y}
	\sfrac{b_1t}{b_1t-r}
	\Big( g(t) + f(B,D) \Big) \nonumber \\
	&\times
	\int_D^B \d y_1 
	\eee^{-\sqrt{2 } \frac{y_1}{ \s_1 \s_2 } } y_1^2
	\int_D^B \dots
	\int_D^B \d y_{\ell-2}
	\eee^{-\sqrt{2} \frac{y_{\ell-2}}{ \s_{\ell-2} \s_{\ell-1}} } y_{\ell-2}^2 (1+o(1)),
\end{align}
where $g \colon \R_+ \to \R_+$ is a function with $g(t) \downarrow 0$ as $t \uparrow \infty$ and $f \colon \R_+^2 \to \R_+$ is a function with $f(B,D) \downarrow 0$ as $B \uparrow \infty, D \downarrow 0$.
All integrals in \eqref{eq:above_speed_final_bound} are finite since $\s_k \to 1$ as $t \uparrow \infty$.
Taking first $t \uparrow \infty$ and then $D \downarrow 0, B \uparrow \infty$, the term \eqref{eq:above_speed_final_bound} gets as small as we want,  in particular smaller that $\e/\ell$.

The summands in \eqref{eq:above_speed_union_bound} for $k < \ell-1$ can be bounded by the same reasoning as for the $k=\ell-1$ case:
The integral in \eqref{eq:above_speed_y} for index $k$ runs over  $(0,D) \cup (B, \infty)$ and 
hence can be bounded as in \eqref{eq:above_speed_Jb}. 
Therefore, each summand in \eqref{eq:above_speed_union_bound} is smaller that $\e/\ell$, which concludes the proof.
\end{proof}


The last localisation proposition states that ancestors of particles which are extremal at time $t$ stay a small order below the barrier function.

\begin{proposition}[Entropic repulsion] \label{prop:above_entropic_rep}
Let $\beta \in (0, 1/2)$ such that, for all $1 \le k < \ell$, $t^{\beta} \ll ( \s_k - \s_{k+1} )^{-1} $.
Then, for any $ y \in \R$ and for any $\e >0$, there exists a constant $\delta \in (0,1/2)$ 
such that, for all $t$ large enough,
\begin{align} \label{eq:above_entropic_prob}
	\P \, \bigg( 
	&\exists_{j \le n(t)} \colon \Big\{ \tilde{x}_j(t) > m^+(t)-y \Big\} \nonumber \\
	&\land \, \Bigg\{ \exists_{1 \le k < \ell} \exists_{s \in S_k^*} :
	\tilde{x}_j(a_{k-1}t + s)
	> \sqrt{2} \, \left( \sum_{i=1}^{k-1} \s_i b_i t + \s_k s \right) - t^{\b \delta}
	\Bigg\} \Bigg) < \e,
\end{align}
where $S_1^* = [t^{\beta}, b_1t]$ and $S_k^* = [0, b_k t]$, for all $ 1 < k < \ell$.
\end{proposition}
Due to Assumption~\ref{as:above3}, we know that $\beta$ as specified in Proposition~\ref{prop:above_entropic_rep} exits.
Reformulating the results of the previous proposition, we get that with high probability 
$\tilde{x}_{i_1} \in \AA_{t^{\beta}, b_1t, t^{\b \delta}, \s_1}$ and 
$\multix{k} \in \AA_{0, b_kt, t^{\b \delta} - \Lambda^{\mi_{k - 1}}(t),  \s_k}$ 
for all $1 < k < \ell$.
By monotonicity, we can use the slightly weaker version of Proposition~\ref{prop:above_entropic_rep} where the ancestor of an extremal particle at time $t$ stays $t^{\b \delta/2}$ below the barrier function on $(t^{\b}, a_{\ell-1}t]$ and $t^{\b \delta}$ below the barrier function at time $t^{\b}$.
\begin{proof}
By a union bound, the probability in \eqref{eq:above_entropic_prob} is not larger than
\begin{align} \label{eq:above_entropic_union}
	\sum_{k=1}^{\ell-1}
	\P \, \bigg(
	\exists_{j \le n(t)} \colon& 
	\Big\{ \tilde{x}_j(t) > m^+(t)-y \Big\} \notag\\
	&\land \left\{ \exists_{s \in S_k} :
	\tilde{x}_j(a_{k-1}t + s)
	> \sqrt{2} \, \left(\textstyle \sum\limits_{i=1}^{k-1} \s_i b_i t + \s_k s \right) - t^{\b \delta}
	\right\}
	\Bigg).
\end{align}
We insert the localisation conditions Propositions~\ref{prop:barrier} and  \ref{prop:above_speed_change} into \eqref{eq:above_entropic_union}.
Then, each summand in \eqref{eq:above_entropic_union} is equal to
\begin{align} \label{eq:above_entropic_localisation}
	&\P \, \bigg( 
	\exists_{j \le n(t)} \colon \Big\{ \tilde{x}_j(t) > m^+(t)-y \Big\}
	\land \, \left\{ \exists_{s \in S_k^*} :
	\tilde{x}_j(a_{k-1}t + s)
	> \sqrt{2} \left( \textstyle\sum\limits_{i=1}^{k-1} \s_i b_i t + \s_k s \right) - t^{\b \delta}
	\right\} \nonumber \\
	&\land \, \Bigg\{ \forall_{1 \le k \le \ell-1} \forall_{s \in S_k} :
	\tilde{x}_j(a_{k-1}t + s)
	\le \sqrt{2} \left( \sum_{i=1}^{k-1} \s_i b_i t + \s_k s \right) \Bigg\}
	\nonumber \\
	&\land \, \Bigg\{ \forall_{1 \le k \le \ell-1} :
	\tilde{x}_j(a_k t) - \sqrt{2}t \, \sum_{i=1}^k \s_i b_i
	\in \sfrac{[ -B, -D] }{\s_k - \s_{k+1} } \Bigg\}
	\Bigg)
	+ \OO(\e),
\end{align}
where $S_1 = [r, b_1t]$ and $S_k = S_k^* = [0, b_kt]$, for $1 < k < \ell$.
Proceeding as  in  \eqref{eq:above_speed_special_localisation} -- \eqref{eq:above_speed_Ik}, the probability in \eqref{eq:above_entropic_localisation} is not larger than
\begin{align} \label{eq:above_entropic_integral_chain_1}
	&\eee^{\sum_{k=1}^{\ell-1} b_k t}
	\int_{I_1}
	\sfrac{\d \o_1}{\sqrt{2 \pi \s_1^2 b_1t }}
	\eee^{- \frac{\o_1^2}{2 \s_1^2 b_1t }} \
	\int_{I_2} \dots
	\int_{I_{\ell-1}} 
	\sfrac{\d \o_{\ell-1}}{\sqrt{2 \pi \s_{\ell-1}^2 b_{\ell-1} t}}
	\eee^{- \frac{\o_{\ell-1}^2}{2 \s_{\ell-1}^2 b_{\ell-1} t}} \nonumber \\
	&\times \prod_{ h=0, h \neq k}^{\ell-1}
	\P  \Big(
	\zet_{
		- \left( \sum_{i=0}^{h-1} \sqrt{2} \s_i b_i t - \o_i \right),
		\o_h - \left( \sum_{i=0}^{h-1} \sqrt{2} \s_i b_i t - \o_i \right)}^{|S_h|}
	(s) \le \sqrt{2} \s_h s \, \forall s \in S_h \Big) \nonumber \\
	&\times 
	\Bigg[
	\P  \Big( \zet_{ 
		- \left( \sum_{i=0}^{k-1} \sqrt{2} \s_i b_i t - \o_i \right),
		\o_k - \left( \sum_{i=0}^{k-1} \sqrt{2} \s_i b_i t - \o_i \right)}^{|S_k|} (s) 
	\le \sqrt{2} \s_1 s \, \forall s \in S_k \Big)\nonumber \\
	& \quad
	- \P  \Big( \zet_{ 
		- \left( \sum_{i=0}^{k-1} \sqrt{2} \s_i b_i t - \o_i \right),
		\o_k - \left( \sum_{i=0}^{k-1} \sqrt{2} \s_i b_i t - \o_i \right)}^{|S_k^*|} (s) 
	\le \sqrt{2} \s_1 s - t^{\b \delta} \, \forall s \in S_k^* \Big) 
	\Bigg]\nonumber \\
	&\times \, \P \, \bigg( \max_{i_{\ell} \le n (b_{\ell}t)} \tilde{x}_{i_{\ell}} (b_{\ell}t)
	> m^+(t) - \sum_{i=0}^{\ell-1} \o_i - y \Bigg),
\end{align}
where $I_h$, $1 \le h \le \ell-1$ is defined as in \eqref{eq:above_speed_Ik}.
Proceeding as in \eqref{eq:above_speed_z} -- \eqref{eq:above_speed_zzz} and setting $y_0=0$, we see that \eqref{eq:above_entropic_integral_chain_1} is bounded from above by
\begin{align} \label{eq:above_entropic_y}
	&\int_D^B 
	\sfrac{ \d y_1 }{ \sqrt{2 \pi \s_1^2 b_1t} }
	\sfrac{1}{ \s_1 - \s_2 }
	\eee^{ - \sqrt{2} \frac{y_1}{\s_1 \s_2} } 
	\int_D^B \dots
	\int_D^B
	\sfrac{\d y_{\ell-1}}{\sqrt{2 \pi \s_{\ell-1}^2 b_{\ell-1} t}}
	\sfrac{1}{ \s_{\ell-1} - \s_{\ell} }
	\eee^{ \sqrt{2} \frac{ y_{\ell-1} }{ \s_{\ell-1} (\s_{\ell-1} - \s_{\ell} )} } \nonumber \\
	&\times \,
	\prod_{h=1, h \neq k}^{\ell-1}
	\P \, \bigg( \zet_{
		- \frac{ y_{h-1} }{ \s_h (\s_{h-1} - \s_h) },
		- \frac{ y_h }{ \s_h ( \s_h - \s_{h+1}) } }^{|S_h|} 
	(s) \le 0 \, \forall s \in S_h \bigg) \nonumber \\
	&\times \, 
	\Bigg[
	\P \, \bigg( \zet_{
		- \frac{ y_{k-1} }{ \s_k (\s_{k-1} - \s_k) }, 
		- \frac{ y_k }{ \s_k (\s_k - \s_{k+1}) }}^{|S_k|} (s) 
	\le 0 \, \forall s \in S_k \bigg)\notag\\
	&\quad- \P \, \bigg( \zet_{
		- \frac{ y_{k-1} }{ \s_k (\s_{k-1} - \s_k) }, 
		- \frac{ y_k }{ \s_k (\s_k - \s_{k+1}) }}^{|S_k^*|} (s) 
	\le - \sfrac{t^{\b \delta}}{\s_1} \, \forall s \in S_k^* \bigg)
	\Bigg] \nonumber \\
	&\times \, \P \, \bigg( \max_{i_{\ell} \le n(b_{\ell}t)} \tilde{x}_{i_{\ell}} (b_{\ell}t)
	> m^+(t) + \sfrac{ y_{\ell-1} }{ \s_{\ell-1} - \s_{\ell} } - \sqrt{2} t \, \sum_{i=1}^{\ell-1} \s_i b_i - y \bigg).
\end{align}

Once we show that
\begin{align} \label{eq:above_entropic_diffkey}
	&  \P \, \bigg( \zet_{
		- \frac{ y_{k-1} }{ \s_k (\s_{k-1} - \s_k) }, 
		- \frac{ y_k }{ \s_k (\s_k - \s_{k+1}) }}^{|S_k|} (s) 
	\le 0 \, \forall s \in S_k \bigg)\notag\\
	&\quad- \P \, \bigg( \zet_{
		- \frac{ y_{k-1} }{ \s_k (\s_{k-1} - \s_k) }, 
		- \frac{ y_k }{ \s_k (\s_k - \s_{k+1}) }}^{|S_k^*|} (s) 
	\le - \sfrac{t^{\b \delta}}{\s_1} \, \forall s \in S_k^* \bigg) \nonumber \\
	&\ll
	\begin{cases}
		\frac{ 2 \sqrt{r}}{b_1t - r} 
		\frac{y_1}{\s_1 (\s_1 - \s_2)}
		\quad &\text{if } k=1, \\
		\frac{2}{ \s_k^2 b_kt} 
		\frac{ y_{k-1} }{ \s_{k-1} - \s_k }
		\frac{ y_k }{ \s_k - \s_{k+1}} 
		\quad &\text{if } k \ge 2,
	\end{cases}
\end{align}
the claim of the proposition follows since we have proven (in Proposition~\ref{prop:above_speed_change}) that \eqref{eq:above_entropic_y} is of order 1 when the product in \eqref{eq:above_entropic_y} is taken over all $h$. 

For $k \ge 2$, it follows from Lemma~\ref{lem:BB_line} that
\begin{align} \label{eq:above_entropic_BB_diffC}
	&\P \, \bigg( \zet_{
		- \frac{ y_{k-1} }{ \s_k (\s_{k-1} - \s_k) },
		- \frac{ y_k }{ \s_k (\s_k - \s_{k+1}) }
	}^{|S_k|} 
	(s) \le 0 \, \forall s \in S_k \bigg)\notag\\
	&\quad- \P \, \bigg( \zet_{
		- \frac{ y_{k-1} }{ \s_k (\s_{k-1} - \s_k) },
		- \frac{ y_k }{ \s_k (\s_k - \s_{k+1}) }
	}^{|S_k|} 
	(s) \le - \sfrac{ t^{\beta \delta} }{\s_k} \, \forall s \in S_k \bigg) \nonumber \\
	= \: & \exp \left(
	- \sfrac{2}{\s_k^2 b_kt}
	\sfrac{ y_{k-1} }{ \s_{k-1} - \s_k }
	\sfrac{ y_k}{ \s_k - \s_{k+1} }
	\right)
	\left[
	\exp \left( \sfrac{ 2 t^{\b \delta} }{ \s_k^2 b_kt}
	\left( \sfrac{ y_{k-1} }{ \s_{k-1} - \s_k }
	+ \sfrac{ y_k}{ \s_k - \s_{k+1} }
	- t^{\b \delta} \right) \right)
	-1
	\right].
\end{align}
Since $y_{k-1}, y_k \in [D,B]$ and by Assumption~\ref{as:above3}, the arguments of the exponential functions in \eqref{eq:above_entropic_BB_diffC} converge to zero as $t \uparrow \infty$.	We see that the right-hand side of \eqref{eq:above_entropic_BB_diffC} equals
\begin{equation} \label{eq:above_entropic_diff2_bound}
	2 \sfrac{t^{\b \delta} }{ \s_k^2 b_kt}
	\left( \sfrac{ y_{k-1} }{ \s_{k-1} - \s_k }
	+ \sfrac{ y_k}{ \s_k - \s_{k+1} } \right)
	( 1 + o(1)).
\end{equation}
We see that \eqref{eq:above_entropic_diff2_bound} is smaller than the right-hand side of \eqref{eq:above_entropic_diffkey} for $k \ge 2$ since $t^{\b \delta} \ll ( \s_k - \s_{k+1})^{-1}$ for all $2 \le k \le \ell-1$.

In the case $k=1$, we rewrite the difference of the probabilities as
\begin{align} \label{eq:above_entropic_diff1_integral}
	&\int_{- \infty}^0
	\d z
	\frac{ \exp \bigg({-} 
		\frac{b_1t}{2 (b_1t - t^{\b}) t^{\b}}
		\left( z + \frac{ t^{\b} }{b_1t} \frac{y_1}{ \s_1 (\s_1 - \s_2)} \right)^2 \bigg) }{
		\sqrt{ 2 \pi (b_1t - t^{\b}) t^{\b} / (b_1t)}}
	\,\P \, \bigg( \zet_{0, z}^{t^{\b}} (s) \le 0 \, \forall s \in [r, t^{\b}] \bigg) \\
	&\!\! \times\!\! \bigg[ 
	\P \bigg( \zet_{
		z, 
		- \frac{ y_1 }{ \s_1 (\s_1 - \s_2) }}^{b_1t - t^{\b}}\!(s) 
	\le 0 \: \forall s \!\in\! [0, b_1t - t^{\b}] \bigg)
	- \P \, \bigg( \zet_{
		z, 
		- \frac{ y_1 }{ \s_1 (\s_1 - \s_2) }}^{b_1t - t^{\b}}\!(s) 
	\le - \sfrac{t^{\b \delta}}{\s_1} \: \forall s \!\in\! [0,b_1t - t^{\b}] \bigg) \!
	\bigg]\!. \nonumber
\end{align}
We split the domain of integration into three parts,
\begin{equation} \label{eq:above_entropic_J}
	J_c
	\equiv \left( - \infty, - t^{\b} \right), \quad
	J_d 
	\equiv \left(  - t^{\b}, - t^{\b \delta} / \s_1 \right),\quad
	J_e
	\equiv \left(  - t^{\b \delta} / \s_1, 0 \right).
\end{equation}
On $J_c$, we bound the probabilities of the Brownian bridges by 1. 
Using $t^{\b} \ll b_1t$ and a Gaussian tail estimate, we see that \eqref{eq:above_entropic_diff1_integral} integrated over $J_c$  is smaller than $\exp (-t^{\b}/2)$.
For $z \in J_d$, we get, by Lemma~\ref{lem:BB_line},
\begin{align} \label{eq:above_entropic_BB_diffA}
	&\P \, \bigg( \zet_{
		z, 
		- \frac{ y_1 }{ \s_1 (\s_1 - \s_2) }}^{b_1t - t^{\b}} (s) 
	\le 0 \, \forall s \in [0, b_1t - t^{\b}] \bigg)
	- \P \, \bigg( \zet_{
		z, 
		- \frac{ y_1 }{ \s_1 (\s_1 - \s_2) }}^{b_1t - t^{\b}} (s) 
	\le - \sfrac{t^{\b \delta}}{\s_1} \, \forall s \in [0, b_1t - t^{\b}] \bigg) \nonumber \\
	&= \, \exp \left( - 
	\sfrac{ 2 (-z) y_1 }{ \s_1 (\s_1 - \s_2) (b_1t - t^{\b})}  \right)
	\left[ \exp \left(
	\sfrac{ 2 t^{\beta \delta}}{\s_1 (b_1t - t^{\b} )}
	\left(
	\sfrac{ y_1}{  \s_1 ( \s_1 - \s_2) } 
	- z 
	- \sfrac{ t^{\b \delta} }{ \s_1 } \right) \right)
	-1
	\right].
\end{align}
Since $|z| \le t^{\b}, y_1 \in [D,B]$ and $ t^{\beta} \ll (\s_1 - \s_2)^{-1} \ll \sqrt{b_1t}$, the arguments in both exponential functions in \eqref{eq:above_entropic_BB_diffA} tend to zero as $t \uparrow \infty$.
Thus, the right-hand side of \eqref{eq:above_entropic_BB_diffA} is equal to 
\begin{equation} \label{eq:above_entropic_BB_diffA_bound}
	\sfrac{2 t^{\beta \delta}}{ \s_1 (b_1t - t^{\b})}
	\left( \sfrac{ y_1}{  \s_1  ( \s_1 - \s_2)} - z - \sfrac{ t^{\b \delta} }{ \s_1 }
	\right)
	\, (1+o(1)) \leq \sfrac{2 t^{\beta \delta}}{ \s_1^2 (b_1t - t^{\b})}
	\sfrac{ y_1}{ \s_1 - \s_2} (1+o(1)).
\end{equation}
For $z \in J_e$, the second probability in the left-hand side of \eqref{eq:above_entropic_BB_diffA} is zero. By Lemma~\ref{lem:BB_line}, the first probability is smaller than
\begin{equation} \label{eq:above_entropic_BB_diffB_bound}
	\sfrac{2 (-z) y_1 }{ \s_1 (\s_1 - \s_2) (b_1t - t^{\b})}
	\leq \sfrac{2 t^{\b \delta} y_1 }{ \s_1^2 (\s_1 - \s_2) (b_1t - t^{\b})}.
\end{equation}
We use \eqref{eq:above_barrier_BB_bound} to bound the first probability in \eqref{eq:above_entropic_diff1_integral}. Then, by \eqref{eq:above_entropic_BB_diffA_bound} and \eqref{eq:above_entropic_BB_diffB_bound}, \eqref{eq:above_entropic_diff1_integral} is not larger than
\begin{align} \label{eq:above_entropic_diff1_bound}
	&\eee^{-t^{\b}/2} + 
	\sfrac{ 2\sqrt{r} }{ \s_1^2 (b_1t - t^{\b})}
	\sfrac{y_1}{\s_1 - \s_2}
	t^{\b (\delta - 1/2)}
	\int_{-t^{\b/2} }^{ 0 }
	\d x \, 
	\eee^{ - x^2/2}
	(-x) 
	(1+o(1)) 
	= \OO \left( 
	\sfrac{ \sqrt{r} y_1t^{\b ( \delta - 1/2) } }{ ( \s_1- \s_2) (b_1t - t^{\b})} 
	\right).
\end{align}
Since $\delta \in (0,1/2)$ and $t^{\b \delta} \ll (\s_1 - \s_2)^{-1}$, the the last line of \eqref{eq:above_entropic_diff1_bound} is smaller than the right-hand side of  \eqref{eq:above_entropic_diffkey} for \mbox{$k=1$}.
This concludes the proof.
\end{proof}
\subsection{Localisation of paths in Case B}
\label{sec:below_localisation}
In this subsection, we prove localisations of a VSBBM $(\tilde{X}_t)_{t>0}$  with speed functions $(A_t)_{t>0}$ satisfying Assumption~\ref{as:below}. In  Figure~\ref{fig:Localisation_below}, 
we illustrate these localisation results for three-speed VSBBM.

\begin{figure}
\includegraphics[page=2,width = 12cm]{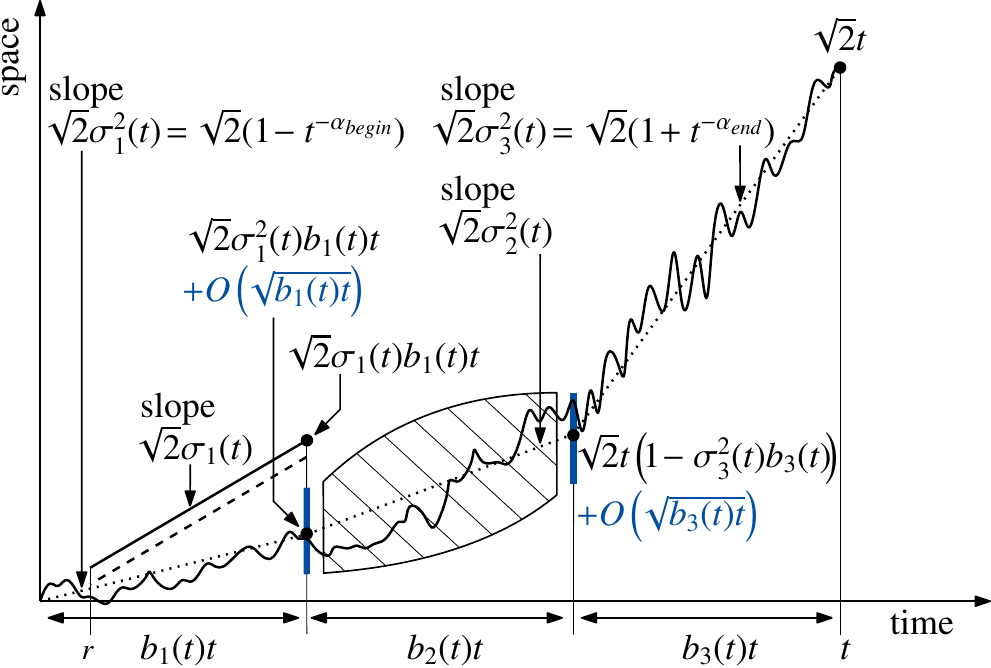}

\caption{Localisation of an extremal particle of three-speed BBM in Case B. The dotted line depicts the function $\sqrt{2} A_t(s/t) t$ for $s \in [0, t]$. The dashed box depicts fluctuations of order $\sqrt{ \min\{ A_t (s/t) t, t - A_t(s/t) t\} }$. The dashed line depicts the effect of entropic repulsion. }\label{fig:Localisation_below}	
\end{figure}
We define the set
\be
\TT_{r_1, r_2, S, \g}
\equiv 
\bigg\{ X \colon \forall_{r_1 \le s \le r_2} \colon 
\big| X(s) + S - \sqrt{2} t A_t\kl{s/t} \big|
<  \bbkl{ A_t\kl{s/t} \wedge \bkl{1 - A_t\kl{s/t} \!} }^\gamma t^\gamma
\bigg\},
\ee
where $0 \le r_1 \le r_2 \le t, S \in \R$ and $\gamma >0$.
Between times $b_1 t$ and $(1-b_\ell) t$, extremal particles fluctuate like a time-inhomogeneous Brownian bridge with time-inhomogeneity controlled by the speed function.
\begin{proposition}[Fluctuations in the middle part] \label{prop:below_bridgeLok}
For any $y\in \R$, any $\e >0$, any $\g > 1/2$ 
and for all $t$ large enough,
\begin{align}
	\label{eq:below_bridgeClaim}
	\P \, \bigg( 
	&\exists_{j \le n(t)} \colon 
	\Big\{ 
	\bar{x}_j(t) > m^-(t)-y
	\Big\}
	\land \, 
	\bigg\{\bar x_j \not\in \TT_{b_1t, (1-b_\ell)t, 0, \g}
	\bigg) < \e.
\end{align}
\end{proposition}
We first control the localisation 
at the times of the first and last speed change.
\begin{lemma}[Position at the last speed change] \label{lem:below_gate2}
For any $y \in \R$, any $\e >0$, any $\g > 1/2$ and for all $t$ large enough,
\begin{align}\label{eq:below_gate2_1}
	\P \, \bigg( 
	&\exists_{j \le n(t)} \colon 
	\Big\{ 
	\bar{x}_j(t) > m^-(t)-y
	\Big\} 
	\land \, 
	\bigg\{ 
	\abs{ 
		\bar{x}_j((1-b_\ell) t) - \sqrt{2} (1-\s_\ell^2 b_\ell) t
	}
	> (\s_\ell^2 b_\ell t)^\g
	\bigg\}
	\bigg) < \e.
\end{align}
\end{lemma}
\begin{lemma}[Position at the first speed change] \label{lem:below_gate1}
For any $y \in \R$, any $\e >0$, any $\g > 1/2$ and for all $t$ large enough,
\begin{align}
	\label{eq:below_gate1_1}
	\P \, \bigg( 
	&\exists_{j \le n(t)} \colon 
	\Big\{ 
	\bar{x}_j(t) > m^-(t)-y
	\Big\} 
	\land \, 
	\bigg\{ 
	\abs{ 
		\bar{x}_j(b_1 t) - \sqrt{2}\s_1^2 b_1 t
	}
	> (\s_1^2 b_1 t)^\g
	\bigg\} 
	\bigg) < \e.
\end{align}
\end{lemma}

\begin{proof}[Proof of Lemma \ref{lem:below_gate2}]
A first moment method as in \eqref{eq:above_speed_special_localisation} -- \eqref{eq:above_speed_integral_chain_1} shows that the probability in \eqref{eq:below_gate2_1} is bounded from above by
\begin{align}
	\label{eq:below_gate2_1.5}
	&\eee^{(1-b_\ell) t}
	\int
	_{			J_1
	}
	\sfrac{\d \omega}{\sqrt{2 \pi (1  -\s_\ell^2 b_\ell)t}}
	\eee^{- \sfrac{\omega^2}{2 (1  -\s_\ell^2 b_\ell)t}}
	\P \, \bigg(
	\max_{i \le n(b_{\ell}t)} \s_\ell x_{i}^{b_\ell t}(b_{\ell}t)
	> m^-(t) - \omega -y
	\, 
	\bigg),
\end{align}
where $\kl{x_{i}^{b_\ell t}(b_{\ell}t)}_{i \le n(b_{\ell}t)}$ denotes a standard BBM at time $b_\ell t$ and
\begin{align}
	J_1 &\equiv 
	\kl{
		-\infty,\sqrt{2}(1-\s_\ell^2 b_\ell)t - (\s_\ell^2 b_\ell t)^\g
	} 
	\cup \kl{ 
		\sqrt{2}(1-\s_\ell^2 b_\ell)t + (\s_\ell^2 b_\ell t)^\g, \infty
	}.
\end{align}
We split the range of integration $J_1$ into $\big[\!\sqrt{2}(1- b_\ell)t, \infty\big)$ and $J_2 \equiv 
J_1 \setminus \big[\!\sqrt{2}(1- b_\ell)t, \infty\big)$.\\ For \mbox{$\o\in\left[\!\sqrt{2}(1- b_\ell)t, \infty\right)$}, we bound the probability in \eqref{eq:below_gate2_1.5} by $1$ and the integral by
\begin{align}\label{eq:below_gate2_1.6}
	\eee^{(1-b_\ell) t}
	\int
	_{
		\sqrt{2}(1- b_\ell)t 
	}^{\infty}
	\sfrac{\d \omega}{\sqrt{2 \pi (1  -\s_\ell^2 b_\ell)t}}
	\eee^{- \sfrac{\omega^2}{2 (1  -\s_\ell^2 b_\ell)t}}
	\leq 
	\eee^{(1-b_\ell) t} \eee^{- \sfrac{(1-b_\ell)^2 t}{ (1  -\s_\ell^2 b_\ell)}} = o(1),
\end{align}
by a Gaussian tail bound.
For the integral over $J_2$, 
we write
\begin{align}
	z(t, \omega) \equiv  \s_\ell^{-1} \kl{ m^-(t) - \omega -y} - \sqrt{2} b_\ell t,
\end{align}  
and 
use \cite[Lemma 2.3]{1to6} to get 
\begin{align}
	\begin{split}
		\P \, \bigg(
		\max_{i \le n(b_{\ell}t)} \s_\ell x_{i}^{b_\ell t}(b_{\ell}t)
		> m^-(t) - \omega -y
		\, 
		\bigg)
		&\leq
		\sfrac{1}{
			\sqrt{2\pi} (\sqrt{2b_\ell t} + z(t, \omega) / \!\sqrt{b_\ell t})
		} 
		\eee^{-\sqrt{2}z(t, \omega) - \sfrac{z^2(t, \omega)}{2b_\ell t}}\\
		&\leq
		\sfrac{1}{2\sqrt{\pi b_\ell t}}
		\eee^{
			b_\ell t -\sfrac{\left( m^-(t) - \omega -y\right)^2}{2\s_\ell^2 b_\ell t}
		}	
		(1+o(1)).
	\end{split}
	\label{eq:below_gate2_2_2}
\end{align}
By \eqref{eq:below_gate2_1.6} and \eqref{eq:below_gate2_2_2},
\eqref{eq:below_gate2_1.5} is not larger than
\begin{align}
	&\sfrac{1}{2\sqrt{\pi b_\ell t}}
	\eee^{ t}
	\int
	_{
		J_2
	}
	\sfrac{\d \omega}{\sqrt{2 \pi (1  -\s_\ell^2 b_\ell)t}}
	\eee^{- \sfrac{\omega^2}{2 (1  -\s_\ell^2 b_\ell)t}}
	\eee^{
		-\sfrac{
			\left( m^-(t) - \omega 	-y\right)^2
		}{
			2\s_\ell^2 b_\ell t
		}
	}(1+o(1))\label{eq:below_gate2_3}\\
	&=\:
	\sfrac{1}{2\sqrt{\pi b_\ell t}}
	\eee^{t-\sfrac{(m^-(t) - y)^2}{2 t}}
	\int
	_{
		J_2
	}
	\sfrac{\d \omega}{\sqrt{2 \pi (1  -\s_\ell^2 b_\ell)t}}
	\exp\kl{
		- \sfrac{
			\left(\omega\,-\,(1-\s_\ell^2 b_\ell) (m^-(t) - y)\right)^2
		}{
			2 (1  -\s_\ell^2 b_\ell)\,\s_\ell^2 b_\ell t
		}
	}
	(1+o(1))\notag\\
	&=\:
	\sfrac{
		\eee^{
			\sqrt{2} d 
		}
		t^{ 1/2 + \a_1+\a_\ell}
	}{
		\sqrt{b_\ell}
	}
	\!\!\!\!\hspace{-4mm}
	\int\limits_{
		\left(-(\s_\ell^2 b_\ell t)^\g\!,\, (\s_\ell^2 b_\ell t)^\g\right)^c
	}
	\!\!\!
	\sfrac{\d \omega}{\sqrt{2 \pi (1  -\s_\ell^2 b_\ell)t}}
	\exp\!\bbbbkl{
		\!{-} \sfrac{
			\kl{
				\omega
				-(1-\s_\ell^2 b_\ell) 
				\kl{m^-(t) - \sqrt{2} t - y
				}
			}^2
		}{
			2 (1  -\s_\ell^2 b_\ell)\,\s_\ell^2 b_\ell t
		}
	}
	\!(1+o(1)),
	\notag
\end{align}
which tends to zero as $t\uparrow\infty$ since 
$m^-(t) \,- \!\sqrt{2} t - y=\OO(\ln(t))$.
\end{proof} 

\begin{proof}[Proof of Lemma~\ref{lem:below_gate1}]
Due to Assumption~\ref{as:belowLSpeed3}, we can choose $\g>1/2$ 
such that
\begin{align}
	\label{eq:below_gate2_tightenedBEllAssumption}
	1 \gg	b_\ell &\gg t^{\frac{\a_\ell + \g- 1}{1-\g}}.
\end{align}
By Lemma~\ref{lem:below_gate2}, the probability in \eqref{eq:below_gate1_1}  equals
\begin{align}
	&\P \, \bigg( 
	\exists_{j \le n(t)} \colon 
	\Big\{ 
	\bar{x}_j(t) > m^-(t)-y
	\Big\} 
	\land \, 
	\bigg\{ 
	\abs{ 
		\bar{x}_j(b_1 t) - \sqrt{2}\s_1^2 b_1 t
	}
	> (\s_1^2 b_1 t)^\g
	\bigg\}\notag\\
	&\quad\quad\qquad\land\,
	\bigg\{ 
	\abs{ 
		\bar{x}_j((1-b_\ell) t) - \sqrt{2} (1-\s_\ell^2 b_\ell) t
	}
	\leq (\s_\ell^2 b_\ell t)^\g
	\bigg\} 
	\bigg) + \OO(\e)\nonumber\\
	&\leq\: \eee^{(1-b_\ell) t}
	\!\!
	\displaystyle
	\int\limits
	_{(I_1)^c}
	\!\!\!\!
	\sfrac{\d \omega_1}{\sqrt{2 \pi \s_1^2 b_1t}}
	\eee^{- \sfrac{\omega_1^2}{2 \s_1^2 b_1 t}} 
	\int\limits
	_{
		I_2
	}
	\!\!\!\!
	\sfrac{\d \omega_2}{\sqrt{2 \pi (1 - \s_1^2 b_1 -\s_\ell^2 b_\ell)t}}
	\eee^{- \sfrac{\omega_2^2}{2 (1 - \s_1^2 b_1 -\s_\ell^2 b_\ell)t}}
	\notag\\
	&\quad\quad\qquad\times 
	\P \, \bigg(
	\max_{i \le n(b_{\ell}t)} \s_\ell x_{i}^{b_\ell t}(b_{\ell}t)
	> m^-(t) - \omega_1 -\omega_2 -y
	\bigg)+\OO(\e),\label{eq:below_gate1_3}
\end{align}  
where
\begin{align}\label{eq:below_DefIntervals}
	I_1 &\equiv \kl{\sqrt{2}\s_1^2 b_1 t-(\s_1^2 b_1 t)^\g, \sqrt{2}\s_1^2 b_1 t+(\s_1^2 b_1 t)^\g}, \notag\\
	I_2 &\equiv 
	\kl{
		\sqrt{2}(1-\s_\ell^2 b_\ell) t - \omega_1 - (\s_\ell^2 b_\ell t)^\g,
		\sqrt{2}(1-\s_\ell^2 b_\ell) t - \omega_1 + (\s_\ell^2 b_\ell t)^\g
	}.
\end{align} 
For $\omega_2 \in I_2$, by \eqref{eq:below_gate2_tightenedBEllAssumption}, we can use
Proposition~\ref{prop:FKPP_tail_estimate} and get
\begin{align}
	\label{eq:below_gate1_4}
	&\P \, \bigg(
	\max_{i \le n(b_{\ell}t)} \s_\ell x_{i}^{b_\ell t}(b_{\ell}t)
	> m^-(t) - \omega_1 -\omega_2 -y
	\, 
	\bigg)\notag\\
	&=\:
	\sfrac{C}{\sqrt{2}}
	b_\ell t^{1-\a_\ell}
	\eee^{			b_\ell t 
		-\sfrac{
			\kl{
				\sqrt{2} t - \omega_1 - \omega_2 + L(t)-y
			}^2
		}{2\s_\ell^2 b_\ell t}}
	(1+o(1)),
\end{align}
where
\begin{equation}
	L(t) \equiv  
	\sfrac{3}{2\sqrt{2}}\s_\ell \log(b_\ell t)
	- \sfrac{1 + 2(\a_1+\a_\ell)}{2\sqrt{2}} \log(t).
	\label{eq:below_defL}
\end{equation}
This implies that
\begin{align}\label{eq:below_gate1_5}
	&\int\limits
	_{
		I_2		}
	\!\!\!\!
	\sfrac{\d \omega_2}{\sqrt{2 \pi (1 - \s_1^2 b_1 -\s_\ell^2 b_\ell)t}}
	\eee^{- \sfrac{\omega_2^2}{2 (1 - \s_1^2 b_1 -\s_\ell^2 b_\ell)t}}
	\P \, \bigg(
	\max_{i \le n(b_{\ell}t)} \s_\ell x_{i}^{b_\ell t}(b_{\ell}t)
	> m^-(t) - \omega_1 -\omega_2 -y
	\, 
	\bigg)\notag\\
	&=
	\sfrac{C}{\sqrt{2}}
	b_\ell t^{1-\a_\ell}
	\eee^{b_\ell t}
	\eee^{
		-\sfrac{
			\kl{
				\sqrt{2} t - \omega_1 + L(t)-y
			}^2
		}{
			2(1-\s_1^2 b_1) t
		}
	}
	\int\limits_{
		I_2
	}
	\sfrac{
		\d \omega_2
	}{
		\sqrt{2 \pi (1 - \s_1^2 b_1 -\s_\ell^2 b_\ell)t}
	}\notag\\
	&\qquad\times
	\exp\bbbbkl{
		{-} (1-\s_1^2 b_1)\sfrac{
			\kl{
				\omega_2 
				- 
				\kl{1-\s_1^2 b_1}^{-1}
				\kl{
					1-\s_1^2 b_1 - \s_\ell^2 b_\ell
				}
				\kl{
					\sqrt{2}t - \omega_1 + L(t)-y
				}
			}^2
		}{
			2 (1 - \s_1^2 b_1 -\s_\ell^2 b_\ell)\s_\ell^2 b_\ell t
		}
	}(1+o(1))\nonumber\\
	&\leq 	\sfrac{C}{\sqrt{2}}
	b_\ell t^{1-\a_\ell}
	\eee^{b_\ell t}
	\eee^{
		-\sfrac{
			\kl{
				\sqrt{2} t - \omega_1 + L(t)-y
			}^2
		}{
			2(1-\s_1^2 b_1) t
		}
	}
	\sfrac{\s_\ell b_\ell^{1/2}}{(1-\s_1^2 b_1)^{1/2}}(1+o(1)).
\end{align}
Therefore,  
\eqref{eq:below_gate1_3} is not larger than
\begin{align}
	&\sfrac{C}{\sqrt{2}}
	b_\ell^{3/2} t^{1-\a_\ell}
	\eee^{-\sqrt 2(L(t)-y)
		-\frac{(L(t)-y)^2}{2t}
	}\notag\\
	&\times\int\limits
	_{\left(-(\s_1^2 b_1 t)^\g\!, \,(\s_1^2 b_1 t)^\g\right)^c}
	\sfrac{\d \omega_1}{\sqrt{2 \pi (1-\s_1^2 b_1)\s_1^2 b_1t}}
	\exp\kl{
		- \sfrac{
			\kl{
				\omega_1 - \s_1^2 b_1 (L(t)-y)
			}^2
		}{
			2 (1-\s_1^2 b_1)\s_1^2 b_1 t
		}
	} (1+o(1)).
	\label{eq:below_gate1_6}
\end{align}
The prefactor is of polynomial order and the integral  is of order 
$\exp\kl{-\frac{(b_1 t)^{2\g-1}}{2}}$. \linebreak
By Assumption~\ref{as:belowLSpeed2},
\eqref{eq:below_gate1_6} tends to zero as $t\uparrow\infty$.
\end{proof}

\begin{proof}[Proof of Proposition~\ref{prop:below_bridgeLok}]
Let $\g$ be close enough to $1/2$ such that
\eqref{eq:below_gate2_tightenedBEllAssumption}
as well as
\begin{align}
	\label{eq:below_bridge_tightenedB1Assumption}
	1 \gg b_1 &\gg t^{\frac{\a_1 + \g- 1}{1-\g}}
\end{align} 
are satisfied.
Choose $\tilde{\g}$ s.t. \mbox{$1/2 < \tilde{\g} <\g$} and $\tilde{\g} < 1$. By Lemmas~\ref{lem:BarrierLok}, \ref{lem:below_gate2}, and \ref{lem:below_gate1}, the probability in \eqref{eq:below_bridgeClaim} is equal to
\begin{align}
	\label{eq:below_bridge1}
	\P \, \bigg( 
	&\exists_{j \le n(t)} \colon 
	\Big\{ 
	\bar{x}_j(t) > m^-(t)-y
	\Big\}
	\land \,
	\gkl{
		\bar{x}_j \in 
		\AA_{r, b_1 t, 0, \s_1}
		\cap
		\TT_{b_1t, b_1t, 0, \tilde{\g}} 
		\cap 
		\TT_{(1-b_\ell)t, (1-b_\ell)t, 0, \tilde{\g}}
	}
	\notag\\ 
	&\land \, 
	\bigg\{
	\exists_{b_1 t \le s \le (1-b_\ell) t} \colon 
	\big|
	\bar{x}_j(s) - \sqrt{2} t A_t(s/t)
	\big|
	> 
	\bbkl{
		A_t(s/t) \wedge \bkl{1 - A_t(s/t) \!}
	}^\gamma t^\gamma
	\bigg\}
	\bigg) + \OO(\e),
\end{align}
for $r,t>0$ large enough.
As in \eqref{eq:above_speed_special_localisation} -- \eqref{eq:above_speed_integral_chain_1}, 
we see that the probability in \eqref{eq:below_bridge1} is bounded from above by
\begin{align}
	&\eee^{(1-b_\ell) t} 
	\,\E\bigg[
	\1_{
		\bar{x}^t_1 \in \AA_{r, b_1 t, 0, \s_1} \cap \TT_{b_1t, b_1t, 0, \tilde{\g}}
	}
	\notag\\
	&\quad \times\E
	\bigg[
	\1_{
		\bar{x}^t_1 \in \TT_{(1-b_\ell)t, (1-b_\ell)t, 0, \tilde{\g}}
	}
	\1_{
		\exists\, {b_1 t \le s \le (1-b_\ell) t} \colon 
		|
		\bar{x}^t_1(s) -  \sqrt{2} t A_t(s/t)
		|
		> 
		\left(
		A_t(s/t) \wedge \kl{1 - A_t(s/t) \!}
		\right)^\gamma t^\gamma
	}
	\notag\\
	&\quad\times 
	\P \, \bigg(
	\max_{i \le n(b_{\ell}t)} \s_\ell x_{i}^{b_\ell t}(b_{\ell}t)
	> m^-(t) - \bar{x}^{t}_1((1-b_\ell)t) -y
	\, 
	\bigg\vert \, 
	\FF_{(1-b_\ell)t} 
	\bigg) 
	\bigg\vert \, 
	\FF_{b_1 t} 
	\bigg]
	\bigg]\nonumber\\
	=\,&\eee^{(1-b_\ell) t}
	\!\!
	\displaystyle
	\int\limits
	_{I_1}
	\!\!
	\sfrac{\d \omega_1}{\sqrt{2 \pi \s_1^2 b_1t}}
	\eee^{\!- \sfrac{\omega_1^2}{2 \s_1^2 b_1 t}} 
	\P \, \bigg(\!
	\zet_{0, \frac{\omega_1}{\s_1}}^{b_1t}\! (s) \!\le \!\!\sqrt{2} s \, \forall\!_{ s \in [r, b_1t]}\! 
	\bigg)\!\!
	\displaystyle
	\int\limits
	_{I_2}
	\!\!
	\sfrac{\d \omega_2}{\sqrt{2 \pi (1 - \s_1^2 b_1 -\s_\ell^2 b_\ell)t}}
	\eee^{\!- \sfrac{\omega_2^2}{2 (1 - \s_1^2 b_1 -\s_\ell^2 b_\ell)t}}
	\notag\\
	&\quad\times
	\P\kl{
		\exists_{s \in [b_1 t, (1-b_\ell) t]} \colon
		\!
		\abs{
			\zet_{\omega_1, \omega_1+\omega_2}^{(1-\s_1^2b_1 -\s_\ell^2 b_\ell) t}
			\scriptstyle
			\!\kl{A_t(s/t) t - \s_1^2b_1t}
			\: - \sqrt{2} t A_t(s/t)
		}
		\scriptstyle
		\,>\,
		\kl{
			A_t(s/t) \wedge \kl{1 - A_t(s/t) }
		}^\gamma t^\gamma\!
	} \notag\\
	&\quad\times 
	\P \, \bigg(
	\max_{i \le n(b_{\ell}t)} \s_\ell x_{i}^{b_\ell t}(b_{\ell}t)
	> m^-(t) - \omega_1 - \omega_2 -y
	\, 
	\bigg),\label{eq:below_bridge3}
\end{align}
with $I_1, I_2$ as in \eqref{eq:below_DefIntervals}. 
We first bound the probabilities involving Brownian bridges. \linebreak
As in \eqref{eq:above_barrier_BB_bound}, 
for  $\g>1/2$ and  $\o_1 \in I_1$,
\begin{align}
	\P \, \bigg(
	\zet_{0, \sfrac{\omega_1}{\s_1}}^{b_1t} (s)
	\le \!\sqrt{2} s \, \forall\!_{ s \in [r, b_1t]} 
	\bigg)
	&\le 
	\sfrac{2\sqrt{r} }{b_1 t-r}\kl{\sqrt{2} b_1 t - \sfrac{\omega_1}{\s_1}} (1 + o(1))\notag\\
	&\le
	2\sqrt{r} \kl{\sqrt{2}(1-\s_1)  + \s_1^{2\g} b_1 t^{1-\g}}(1+o(1)) \notag\\
	&= \sqrt{2r}t^{-\a_1}(1+o(1)).\label{eq:below_bridge7.5}
\end{align}
For the second probability involving a bridge,
the term inside the absolute value equals
\begin{align}
	\zet_{0,0}^{(1-\s_1^2b_1 -\s_\ell^2 b_\ell) t}
	\kl{A_t(s/t) t - \s_1^2b_1t} 
	+ \omega_1
	+ \sfrac{A_t(s/t) -\s_1^2 b_1 }{1-\s_1^2b_1 -\s_\ell^2 b_\ell} \omega_2
	- \sqrt{2} t A_t(s/t)
	.\label{eq:below_bridge4}
\end{align}
For each $s \in [b_1 t, (1-b_\ell) t]$ and $\o_1, \o_2$  in the ranges of integration,
\begin{align}
	&\omega_1
	+ \sfrac{A_t(s/t) -\s_1^2 b_1 }{1-\s_1^2b_1 -\s_\ell^2 b_\ell} \omega_2
	- \sqrt{2} t A_t(s/t)
	=\:
	\OO\kl{
		\sfrac{1-\s_\ell^2 b_\ell - A_t(s/t) }{1-\s_1^2b_1 -\s_\ell^2 b_\ell}
		(b_1 t) ^{\tilde{\g}}
		+
		\sfrac{A_t(s/t) -\s_1^2 b_1 }{1-\s_1^2b_1 -\s_\ell^2 b_\ell}
		(\s_\ell^2 b_\ell t)^{\tilde{\g}}
	}.\label{eq:below_bridge4.5}
\end{align}
We will show that, for $s \in [b_1 t, (1-b_\ell) t]$ and $t$ large enough, 
\begin{align}
	\sfrac{1-\s_\ell^2 b_\ell - A_t(s/t) }{1-\s_1^2b_1 -\s_\ell^2 b_\ell}
	(\s_1^2 b_1 t)^{\tilde{\g}}
	+
	\sfrac{A_t(s/t) -\s_1^2 b_1 }{1-\s_1^2b_1 -\s_\ell^2 b_\ell}
	(\s_\ell^2 b_\ell t)^{\tilde{\g}}
	\leq 
	\bkl{
		A_t(s/t) \wedge \kl{1 - A_t(s/t) }\!
	}^{\tilde{\g}} 
	t^{\tilde{\g}}
	(1+o(1)).\label{eq:below_bridge4.55}
\end{align}
Namely, \eqref{eq:below_bridge4.55} holds for $s= b_1 t$ and $s=(1-b_\ell) t$.
Assumptions~\ref{as:belowLSpeed2}--(iii) imply that \eqref{eq:below_bridge4.55} holds for $s$ s.t. $A_t(s/t) = 1/2$. Concavity of $x\mapsto x^{\tilde{\g}}$ on $\R_+$ and monotonicity of $A_t$ imply that \eqref{eq:below_bridge4.55} holds for all
$s \in [b_1 t, (1-b_\ell) t]$, and so for those $s$,
\eqref{eq:below_bridge4} equals
\begin{align}
	\zet_{0,0}^{(1-\s_1^2b_1 -\s_\ell^2 b_\ell) t}
	\kl{A_t(s/t) t - \s_1^2b_1t} 
	+
	o\bbkl{
		\bkl{
			A_t(s/t) \wedge \kl{1 - A_t(s/t) }\!
		}^{\g} 
		t^{\g}
	}.
\end{align}
Thus the one-but-last probability  in \eqref{eq:below_bridge3} is bounded from above by
\begin{align}
	\begin{split}
		&\P\kl{
			\exists_{s \in [b_1 t, (1-b_\ell) t]} \colon
			\!
			\abs{
				\zet^{
					(1-\s_1^2 b_1 -\s_\ell b_\ell) t
				}_{
					0, 0
				}
				\scriptstyle\kl{A_t(s/t) t - \s_1^2 b_1 t}
			}
			\scriptstyle
			\,>\,
			\kl{
				A_t(s/t) \wedge \kl{1 - A_t(s/t) }\!
			}^\gamma t^\gamma
		}(1+o(1)). 
		\label{eq:below_bridge5}
	\end{split}
\end{align}
For \mbox{$s \in [b_1 t, (1-b_\ell) t]$}, the fluctuations of the Brownian bridge in \eqref{eq:below_bridge5} are bounded by those of 
$
\bbkl{
	\zet^{
		t
	}_{
		0, 0
	}
	\kl{A_t(s/t) t}\!
}_{
	\!s \in [0,t]
}
$, and thus \eqref{eq:below_bridge5} is bounded by
\begin{align}
	\P\kl{
		\exists_{s \in [b_1 t, (1-b_\ell) t]} \colon			\!
		\abs{
			\zet^{
				t
			}_{
				0, 0
			}
			\scriptstyle\kl{A_t(s/t) t}
		}
		\scriptstyle
		\,>\,
		\kl{
			A_t(s/t) \wedge \kl{1 - A_t(s/t) }\!
		}^\gamma t^\gamma
	}(1+o(1)).\label{eq:below_bridge5.5}
\end{align}
This probability is by the monotonicity of $A_t$ not larger than
\begin{align}
	\P\kl{
		\exists_{s \in [\tilde{r}, t - \tilde{r}]} \colon
		\!
		\abs{
			\zet^{
				t
			}_{
				0, 0
			}
			\bkl{s}
		}
		\,>\,
		\bkl{
			s \wedge \kl{t-s}\!
		}^\gamma
	}\leq \e \sqrt{2\pi}/(C\sqrt{r})
	,\label{eq:below_bridge6}
\end{align}
for $\tilde{r}, t>0$ large enough,  by Lemma~\ref{lem:BB_fluc}.
With   \eqref{eq:below_bridge7.5} and \eqref{eq:below_bridge6}  we see that \eqref{eq:below_bridge3} is smaller than
\begin{align}
	\e\sfrac{\sqrt{2}}{C}
	t^{-\a_1}
	\eee^{(1-b_\ell) t}
	\!\!
	\displaystyle
	\int\limits
	_{I_1}
	\!\!
	&\sfrac{\d \omega_1}{\sqrt{2 \pi \s_1^2 b_1t}}
	\eee^{- \frac{\omega_1^2}{2 \s_1^2 b_1 t}} \!
	\int\limits
	_{I_2}
	\sfrac{\d \omega_2}{\sqrt{2 \pi (1 - \s_1^2 b_1 -\s_\ell^2 b_\ell)t}}
	\eee^{- \frac{\omega_2^2}{2 (1 - \s_1^2 b_1 -\s_\ell^2 b_\ell)t}}
	\notag\\
	&\times 
	\P \, \bigg(
	\max_{i \le n(b_{\ell}t)} \s_\ell x_{i}^{b_\ell t}(b_{\ell}t)
	> m^-(t) - \omega_1 -\omega_2 -y
	\, 
	\bigg) (1+o(1)).\label{eq:below_bridge7.6}
\end{align}
For $\o_1 \in I_1$ and $\tilde{\g}>1/2$, the bound in \eqref{eq:below_gate1_5} is asymptotically sharp, and 
so the integral over $\o_2$ in \eqref{eq:below_bridge7.6} is equal to 
\begin{align}
	\label{eq:below_gate7.9}
	\sfrac{C}{\sqrt{2(1-\s_1^2 b_1)}}
	b_\ell^{3/2} t^{1-\a_\ell}
	\eee^{b_\ell t}
	\exp\kl{
		-\sfrac{
			\kl{
				\sqrt{2} t - \omega_1 + L(t)-y
			}^2
		}{
			2(1-\s_1^2 b_1) t
		}
	}(1+o(1)).
\end{align}
\noindent Inserting \eqref{eq:below_gate7.9} into \eqref{eq:below_bridge7.6} and shifting $\o_1$ by $\sqrt{2}\s_1^2b_1 t$, we see that \eqref{eq:below_bridge7.6} is equal to
\begin{align}
	\begin{split}
		&\e
		b_\ell^{3/2} t^{1-\a_1-\a_\ell}
		\eee^{ t}
		\eee^{
			-\frac{
				(\sqrt{2}t+L(t)-y)^2
			}{2t}
		}
		\!\!
		\displaystyle
		\int
		\limits
		_{-(\s_1^2 b_1 t)^{\tilde{\g}}\!}^{\,(\s_1^2 b_1 t)^{\tilde{\g}}}
		\!\!\!
		\sfrac{\d \omega_1}{\sqrt{2 \pi (1-\s_1^2 b_1)\s_1^2 b_1t}}
		\exp\kl{
			{-}\sfrac{
				\kl{
					\omega_1 - \s_1^2 b_1 (L(t)-y)
				}^2
			}{
				2 (1-\s_1^2 b_1)\s_1^2 b_1 t
			}
		} (1+o(1)).
		\label{eq:below_bridge7.95}
	\end{split}
\end{align}
Note that the integrands in \eqref{eq:below_gate1_6} and \eqref{eq:below_bridge7.95} are the same, but the ranges of integration are complementary. We have seen that the integral in \eqref{eq:below_gate1_6} tends to zero as $t\uparrow\infty$, which implies that the Gaussian integral in \eqref{eq:below_bridge7.95} tends to one.
Recalling the definition of $L(t)$ in \eqref{eq:below_defL}, we find that the $t$-dependent factors outside the integral in  
\eqref{eq:below_bridge7.95} tend to one as well, from which the claim follows.	
\end{proof}

The next proposition is the analog to Proposition~\ref{prop:above_entropic_rep}. It describes the effect of entropic repulsion on the first time interval $[0,b_1 t]$ for extremal particles.

\begin{proposition}[Entropic repulsion] \label{prop:below_entropic_rep}
Let $0<\b<\a_1$. 
Then, for any $y \in \R$ and for any $\e >0$, there exists a constant $\delta \in (0,1/2)$ such that, for all $t$ large enough,
\begin{align}
	\label{eq:below_entrop_1}
	\begin{split}
		\P \, \bigg( 
		&\exists_{j \le n(t)} \colon 
		\Big\{ 
		\bar{x}_j(t) > m^-(t)-y
		\Big\} 
		\land \, 
		\bigg\{  
		\exists_{t^\beta \le s \le b_1t} \colon 
		\bar{x}_j(s) > \sqrt{2} \s_1 s - t^{\b \delta} 
		\bigg\} 
		\bigg) < \e.
	\end{split}
\end{align}
\end{proposition}
\noindent This means that the extremal particles of $(\bar{X}_t)_{t>0}$ lie in $\AA_{t^\beta, b_1 t, t^{\beta\delta}, \s_1}$ with high probability. By monotonicity, we can use the superset $\AA_{t^\beta,t^\b, t^{\beta\delta}, \s_1} \cap \AA_{t^\beta, b_1 t, t^{\beta\delta/2}, \s_1}$ of $\AA_{t^\beta, b_1 t, t^{\beta\delta}, \s_1}$ instead.
\begin{proof}
Let $\g>1/2$ be close enough to $1/2$ that \eqref{eq:below_gate2_tightenedBEllAssumption}
and \eqref{eq:below_bridge_tightenedB1Assumption} are satisfied.
By Lemmas~\ref{lem:BarrierLok}, \ref{lem:below_gate2} and  \ref{lem:below_gate1},
the probability in \eqref{eq:below_entrop_1} is equal to
\begin{align}
	\label{eq:below_entrop_2}
	\P \, \bigg( 
	&\exists_{j \le n(t)} \colon 
	\Big\{ 
	\bar{x}_j(t) > m^-(t)-y
	\Big\}
	\land \,
	\gkl{\bar{x}_j \in \TT_{b_1t, b_1t, 0, \g} \cap \TT_{(1-b_\ell)t, (1-b_\ell)t, 0, \g} \cap \AA_{r, b_1 t, 0, \s_1}}
	\notag\\ 
	&
	\land\,
	\bigg\{  
	\bar{x}_j(t^\b) > -\sqrt{2} \s_1  t^{\b} 
	\bigg\} 
	\land \, 
	\bigg\{  
	\exists_{t^\beta \le s \le b_1t} \colon 
	\bar{x}_j(s) > \sqrt{2} \s_1 s - t^{\b \delta} 
	\bigg\} 
	\bigg) + \OO(\e),
\end{align}
for $r,t>0$ large enough.  
As in \eqref{eq:above_entropic_integral_chain_1}, we see that the probability in \eqref{eq:below_entrop_2} is bounded from above by
\begin{align} \label{eq:below_entrop_neu_3}
	&\eee^{(1-b_\ell) t}
	\int\limits_{I_1}
	\!\! 
	\sfrac{\d \o_1}{\sqrt{2 \pi \s_1^2 (b_1t ) }}
	\eee^{-\sfrac{\o_1^2}{2 \s_1^2 (b_1 t ) }}
	\int\limits
	_{I_2}
	\sfrac{\d \omega_2}{\sqrt{2 \pi (1 - \s_1^2 b_1 -\s_\ell^2 b_\ell)t}}
	\eee^{-\sfrac{\omega_2^2}{2 (1 - \s_1^2 b_1 -\s_\ell^2 b_\ell)t}}\nonumber\\
	&\times\bigg[
	\P  \Big(
	\zet_{ 
		0, 
		\frac{\o_1}{\s_1} -  \sqrt{2} b_1 t 
	}^{
		b_1t 
	} (s) 
	\le 
	0 \, \forall_{ s \in [r, b_1t]}
	\Big) 
	-\P  \Big(
	\zet_{ 
		0, 
		\frac{\o_1}{\s_1} -  \sqrt{2}  b_1 t 
	}^{
		b_1t
	} (s) 
	\le 
	- \sfrac{t^{\b \delta}}{\s_1} \, 
	\forall_{s \in [t^\b, b_1t]}
	\Big)
	\bigg] \nonumber \\
	&\times 
	\P  \Big(
	\max_{i \le n(b_{\ell}t)} \s_\ell x_{i}^{b_\ell t}(b_{\ell}t)
	> m^-(t)  - \omega_1 -\omega_2 -y
	\, 
	\Big),
\end{align}
with $I_1, I_2$ as in \eqref{eq:below_DefIntervals}.
We have seen in \eqref{eq:below_bridge7.6}--\eqref{eq:below_bridge7.95} that replacing the difference of probabilities in \eqref{eq:below_entrop_neu_3} with $t^{-\a_1}$ creates a term of order $1$.
Thus, it suffices to show that
\begin{align}
	\P  \Big(
	\zet_{ 
		0, 
		\frac{\o_1}{\s_1} -  \sqrt{2} b_1 t 
	}^{
		b_1t 
	} (s) 
	\le 
	0 \, \forall_{ s \in [r, b_1t]}
	\Big)
	- 
	\P  \Big(
	\zet_{ 
		0, 
		\frac{\o_1}{\s_1} -  \sqrt{2}  b_1 t 
	}^{
		b_1t
	} (s) 
	\le 
	- \sfrac{t^{\b \delta}}{\s_1} \, 
	\forall_{s \in [t^\b, b_1t]}
	\Big)
	\ll
	t^{-\a_1}.  
\end{align}
We deal with the difference of probabilities in \eqref{eq:below_entrop_neu_3} as in \eqref{eq:above_entropic_diff1_integral}--\eqref{eq:above_entropic_diff1_bound}. The only difference is that we replace $\frac{y_1}{\s_1(\s_1-\s_2)}$ by \mbox{$\sqrt{2}b_1 t - \frac{\o_1}{\s_1}$} and use that for $\o_1$ from the integral in \eqref{eq:below_entrop_neu_3},
\begin{align}\label{eq:below_entrop_neu_4}
	\sfrac{\sqrt{2}b_1 t - \s_1^{-1}\o_1}{b_1 t - t^\b} = 2^{-1/2} t^{-\a_1} (1+o(1)).
\end{align}  
Then we get that the difference of probabilities in \eqref{eq:below_entrop_neu_3} is not larger than a term of order 
\begin{align}
	\sqrt{r}\,t^{\b(\delta-1/2)}\sfrac{\sqrt{2}b_1 t - \s_1^{-1}\o_1}{b_1 t - t^\b} \ll t^{-\a_1}.
\end{align}
In the last step, we applied \eqref{eq:below_entrop_neu_4} again and used that $\delta < 1/2$.
\end{proof}

\section[Proofs of Theorem 1.3 and 1.4]{Proofs of Theorem~\ref{thm:lawofmax} and \ref{thm:extremalprocess}} \label{sec:mainProofs}
To prove Theorem~\ref{thm:lawofmax} and \ref{thm:extremalprocess}, it suffices by Lemma~\ref{lem:jointConvergence} 
to prove the convergence of the 
corresponding Laplace functionals
\begin{equation}
\Psi_t ( \phi) 
\equiv \E \, \left[ \exp \left( - \int \phi(z) \EE_t( \d z) \right) \right],
\end{equation}
where $\phi \in \CC^\infty(\R)$ is a nondecreasing test function as in Lemma~\ref{lem:jointConvergence2} and we denote by  \mbox{$\EE_t \equiv \sum_{j \le n(t)} \delta_{ \tilde{x}_j(t) - m^{\pm}(t)}$}  the extremal process of a VSBBM whose speed functions $(A_t)_{t>0}$ satisfy Assumption~\ref{as:above} or \ref{as:below}. 
We set $m^{\pm}(t) = m^+(t)$ if $(A_t)_{t>0}$ satisfies Assumption~\ref{as:above} and $m^{\pm}(t) = m^-(t)$ otherwise.

The proof of the convergence of $(\Phi_t(\phi))_{t>0}$ 
is essentially analogous to the proof of the convergence of the maximum, which formally corresponds to the choice 
$\phi(z) =0$, if  $z\leq y$ and $\phi(z)=+\infty$, if $z>y$, for $y\in \R$. 	
Writing the proof for the maximum is notationally less heavy and easier on the reader, which is why we give the proof 
of Theorem~\ref{thm:lawofmax} in detail in the following two subsections and only indicate the changes needed for the 
proof of Theorem~\ref{thm:extremalprocess} in Subsection~\ref{sec:proofExtremalProcess}.

\subsection[Proof of Theorem 1.3 in Case~A]{Proof of Theorem~\ref{thm:lawofmax} in Case~A}\label{sec:above}\label{sec:proof_caseA} 
This subsection contains the proof of the first part of Theorem~\ref{thm:lawofmax}, i.e.\ we are in Case~A
and show for an $\ell$-speed BBM with speed functions satisfying Assumption~\ref{as:above} that
\begin{equation}
\lim_{t \uparrow \infty}
\P \left( \max_{j \le n(t)} \tilde{x}_j(t) - m^+(t) \le y \right)
= \E \left[ \exp \left( - C 
Z \eee^{-\sqrt{2}y} \right) \right],
\end{equation}
where $m^+(t)$ is given in \eqref{eq:OrderOfMaxAboveTime}.


By Lemma~\ref{lem:BarrierLok} and Propositions~\ref{prop:barrier},~\ref{prop:above_speed_change}, and ~\ref{prop:above_entropic_rep},
\begin{align}
& 
\P \left( \max_{j \le n(t)} \tilde{x}_j(t) - m^+(t) > y \right)\nonumber \\
&= \ \P \, \bigg( \exists_{i_1 \le n(b_1 t), \dots, i_{\ell} \le n^{\mi_{\ell-1}}(b_{\ell} t)} \colon
\bigg\{ \sum_{k=1}^{\ell} \multix{k}(b_k t) - m^+(t) > y \bigg\}
\land \bigg\{ \tilde{x}_{i_1} \in \LL_{b_1t, \s_1, t^{\b}, B, D} \bigg\} \nonumber \\
&\qquad \land \bigg\{ \forall_{1 < k < \ell} \colon
\multix{k} \in \LL_{b_k t, k, B, D}^{\mi_{k-1}} \bigg\} \bigg)
+ \OO(\e),
\end{align}
where
\begin{align}
\LL_{b_1t, 1, t^{\b}, B, D} 
&\equiv \AA_{t^{\b}, t^{\b}, t^{\b \delta}, \s_1}
\cap \BB_{t^{\b}, t^{\b}, 0, \s_1} 
\cap \AA_{t^{\b}, b_1t, t^{\b \delta/2}, \s_1} 
\cap \GG_{b_1t, 1, 0, B, D}, \nonumber \\
\LL_{b_kt, k, B, D}^{\mi_{k-1}}
&\equiv \AA_{0, b_kt, t^{\b \delta/2} - \Lambda^{\mi_{k - 1}}(t), \s_k}
\cap \GG_{b_kt, k, \Lambda^{\mi_{k - 1}}(t), B, D},
\end{align}
and 
\be
\BB_{r_1, r_2, S, \s}
\equiv \left\{ X \colon \forall_{r_1 \le s \le r_2} \colon X(s) + S > -\sqrt{2} \s s \right\}.
\ee
Due to Lemma~\ref{lem:localisation}, it is enough to analyse
\begin{align} \label{eq:above_branch-1}
&\P \, \bigg( \max_{\substack{i_1 \le n(b_1t) \colon \tilde{x}_{i_1} \in \LL_{b_1t, 1, t^{\b}, B, D}, \dots, \\
		i_{\ell} \le n^{\mi_{\ell-1}}(b_{\ell}t) \colon \multix{\ell} \in \LL_{b_{\ell}t, \ell, B, D}^{\mi_{\ell-1}}}} \, 
\sum_{k=1}^{\ell} \multix{k}(b_kt) - m^+(t) \le y \bigg) \nonumber \\
&= \, \P \, \bigg( \max_{\substack{i_1 \le n(b_1t) \\ \tilde{x}_{i_1} \in \LL_{b_1t, 1, t^{\b}, B, D}}} \
\max_{\substack{\substack{i_2 \le n^{i_1}(b_2t) \\ \tilde{x}_{i_2}^{\mi_1} \in \LL_{b_2t, 2, B, D}^{\mi_1}}}} \dots 
\max_{\substack{i_{\ell} \le n^{\mi_{\ell-1}}(b_{\ell}t) \\ \multix{\ell} \in \LL_{b_{\ell}t, \ell, B, D}^{\mi_{\ell-1}}}}
\sum_{k=1}^{\ell} \multix{k}(b_kt) - m^+(t) \le y \bigg). 
\end{align}
Using the branching property, we rewrite this probability as
\begin{align} \label{eq:above_branch}
&\E \, \bigg[\prod_{\substack{i_1 \le n(b_1t) \\ \tilde{x}_{i_1} \in \LL_{b_1t, 1, t^{\beta}, B, D}}}
\E \, \bigg[ \prod_{\substack{i_2 \le n^{i_1}(b_2t) \\ \tilde{x}_{i_2}^{\mi_1} \in \LL_{b_2t, 2, B, D}^{\mi_1}}} \dots \
\E \, \bigg[ \prod_{\substack{i_{\ell-1} \le n^{\mi_{\ell-2}}(b_{\ell-1}t) \\ \tilde{x}_{i_{\ell-1}}^{\mi_{\ell-2}} \in \LL_{b_{\ell-1}t, \ell-1, B, D}^{\mi_{\ell-2}}}} \\
&\!\!\!\times\! 
\bigg( 
1 \!-\!\P \, \bigg(
\max_{ i_{\ell} \le n^{\mi_{\ell-1}}(b_{\ell}t) } 
\multix{\ell}(b_{\ell} t)
> m^+(t) - \sum_ {k=1}^{\ell-1} \multix{k}(b_k t) + y
\bigg\vert \FF_{a_{\ell-1}t} 
\bigg) \!
\bigg) 
\bigg\vert  \FF_{a_{\ell-2}t} \bigg] \dots 
\bigg \vert  \FF_{a_1t} \bigg] \bigg] \nonumber.
\end{align}
We compute this term from within, i.e.\ we first estimate the tail of the maximum with Proposition~\ref{prop:FKPP_tail_estimate} and then iteratively compute the nested expectations.
To apply the tail estimate, we need that $\Delta^{\mi_{\ell-1}}(t)$ is positive and $\lim_{t \to \infty} \Delta^{\mi_{\ell-1}}(t)/(b_{\ell}t) = 0$, where
\begin{align}
&\Delta^{\mi_{\ell-1}}(t)\label{eq:above_def_Delta} \nonumber \\
\equiv  &\sfrac{1}{\s_{\ell}} \left( m^+(t) - \sum_ {k=1}^{\ell-1} \multix{k}(b_k t) + y \right)
- \left( \sqrt{2} b_{\ell} t - \sfrac{3}{2 \sqrt{2}} \log(b_{\ell} t) \right) \nonumber \\
= &\sfrac{1}{\s_{\ell}} \left( \sum_{k=1}^{\ell-1} \sqrt{2} \s_k b_k t - \multix{k}(b_k t) \right)
- \sfrac{3}{2 \sqrt{2}} \left( \sum_{k=1}^{\ell-1}  \log (b_k t) + 2 \log ( \pi^{1/6} (\s_k \!-\! \s_{k+1}) ) \right) \! + y + o(1) \nonumber \\
\equiv &\sfrac{1}{\s_{\ell}} \, \Lambda^{{\mi_{\ell-1}}}(t) + L_{\ell-1}(t) + y + o(1),
\end{align}
$\Lambda^{\bar{i}_{\ell-1}}(t)$ is defined in \eqref{eq:Lambda_def} and $L_{\ell-1}(t)$ is defined in \eqref{eq:above_def_L}.
Since $\tilde{x}_{i_{\ell-1}}^{\mi_{\ell-2}} \in \LL_{b_{\ell-1}t, \ell-1, B, D}^{\mi_{\ell-2}}$, we get $\Delta^{\mi_{\ell-1}}(t)=O\left(( \s_{\ell-1} - \s_{\ell} )^{-1}\right)$. 
By Assumption~\ref{as:above2}, $\Delta^{\mi_{\ell-1}}(t)>0$, and 
by Assumption~\ref{as:above3},  $\lim_{t \to \infty} \Delta^{\mi_{\ell-1}}(t)/(b_{\ell}t) = 0$. 
Proposition~\ref{prop:FKPP_tail_estimate} implies that
\begin{align} \label{eq:above_tail_prob}
&\P \, \bigg( \max_{ i_{\ell} \le n^{\mi_{\ell-1}}(b_{\ell}t) } 
\multix{\ell}(b_{\ell} t)
> m^+(t) - \sum_ {k=1}^{\ell-1}  \multix{k}(b_k t) + y \, 
\bigg\vert \, \FF_{a_{\ell-1}t} \Bigg) \nonumber \\
&= \, C \Delta^{\mi_{\ell-1}}(t) \eee^{-\sqrt{2} \Delta^{\mi_{\ell-1}}(t)} (1+o(1)),
\end{align}
where $o(1)$ is uniform in the range of possible values of $\tilde{x}_{i_{\ell - 1}}^{\, \mi_{\ell - 2}}$.
Inserting \eqref{eq:above_tail_prob} into \eqref{eq:above_branch}, shows that the innermost expectation in \eqref{eq:above_branch} equals
\begin{align}\label{eq:above_exp_approx1}
&\E \, \Bigg[ \prod_{\substack{i_{\ell-1} \le n^{\mi_{\ell-2}}(b_{\ell-1}t) \\ \tilde{x}_{i_{\ell-1}}^{\mi_{\ell-2}} \in \LL_{b_{\ell-1}t, \ell-1, B, D}^{\mi_{\ell-2}}}}
\bigg( 1 - C \Delta^{\mi_{\ell-1}}(t) \eee^{-\sqrt{2} \Delta^{\mi_{\ell-1}}(t)} \bigg)
\, \bigg\vert \, \FF_{a_{\ell-2}t} \Bigg] (1+o(1)) \nonumber \\
&= \, \E \, \Bigg[ 
\exp \bigg( - 
\sum_{\substack{i_{\ell-1} \le n^{\mi_{\ell-2}}(b_{\ell-1}t) \\ \tilde{x}_{i_{\ell-1}}^{\mi_{\ell-2}} \in \LL_{b_{\ell-1}t, \ell-1, B, D}^{\mi_{\ell-2}}}}
C \Delta^{\mi_{\ell-1}}(t) \eee^{-\sqrt{2} \Delta^{\mi_{\ell-1}}(t)} \bigg)
\, \bigg\vert \, \FF_{a_{\ell-2}t} \Bigg] (1+o(1)).
\end{align}
In the last step, we used $\Delta^{\mi_{\ell-1}}(t) > \frac{D}{\s_{\ell-1} - \s_{\ell}}$ and therefore
$\eee^{-\sqrt{2} \Delta^{\mi_{\ell-1}}(t)}\downarrow 0$, as $ \uparrow \infty$.

\begin{proposition} \label{prop:above_iteration}
With the notation from above,
\begin{align} \label{eq:above_interative_1}
	&\E \, \Bigg[ 
	\exp \bigg( - 
	\sum_{\substack{i_{\ell-1} \le n^{\mi_{\ell-2}}(b_{\ell-1}t) \\ \tilde{x}_{i_{\ell-1}}^{\mi_{\ell-2}} \in \LL_{b_{\ell-1}t, \ell-1, B, D}^{\mi_{\ell-2}}}}
	C \Delta^{\mi_{\ell-1}}(t) \eee^{-\sqrt{2} \Delta^{\mi_{\ell-1}}(t)} \bigg)
	\, \bigg\vert \, \FF_{a_{\ell-2}t} \Bigg] \nonumber \\
	= &\exp \left( - C
	\Lambda^{\mi_{\ell-2}}(t)
	\eee^{ - \sqrt{2} \, \left( \frac{1}{\s_{\ell-1}} \Lambda^{\mi_{\ell-2}}(t) + L_{\ell-2}(t) + y \right)}
	\right)  \, (1+o(1)),
\end{align}
where $o(1)$ is uniform in the range of possible values of $\tilde{x}_{i_{\ell - 2}}^{\, \mi_{\ell - 3}}$.
\end{proposition}	 
Based on Proposition~\ref{prop:above_iteration}, we will set up an iterative procedure with which we compute all nested expectations in \eqref{eq:above_branch}.
\begin{proof}
We rewrite the left-hand side of \eqref{eq:above_interative_1} as 
\begin{equation} \label{eq:above_middle_start}
	\E  \Bigg[ \!\exp \!\Bigg(\! - C \sqrt{\pi}
	\eee^{-\sqrt{2} \, \left( \sfrac{1}{\s_{\ell}} \Lambda^{\mi_{\ell-2}}(t) + L_{\ell-2}(t) + y \right)}
	\left(  \sfrac{1}{\s_{\ell}} \Lambda^{\mi_{\ell-2}}(t) \, \YY_{i_{\ell-2}} (t)
	+ \ZZ_{i_{\ell-2}}(t) \!\right)\!
	\Bigg) \, \bigg\vert \, \FF_{a_{\ell-2}t} \Bigg],  
\end{equation}
where
\begin{align} \label{eq:above_middle_YY_ZZ}
	\YY_{i_{\ell-2}}(t) 
	&\equiv \sum_{\substack{i_{\ell-1} \le n^{\mi_{\ell-2}}(b_{\ell-1}t), \\ \tilde{x}_{i_{\ell-1}}^{\mi_{\ell-2}} \in \LL_{b_{\ell-1}t, \ell-1, B, D}^{\mi_{\ell-2}}}}
	\eee^{-\sqrt{2}
		\frac{1}{\s_{\ell}} \left( \sqrt{2} \s_{\ell-1} b_{\ell-1} t - \tilde{x}_{i_{\ell-1}}^{\mi_{\ell-2}}(b_{\ell-1}t)  \right)
		+ \frac{3}{2} \big( \log(b_{\ell-1}t) + 2 \log ( \s_{\ell-1} - \s_{\ell} ) \big)},  
\end{align}
and 
\begin{align}
	\ZZ_{i_{\ell-2}}(t) 
	&\equiv \sum_{\substack{i_{\ell-1} \le n^{\mi_{\ell-2}}(b_{\ell-1}t), \\ \tilde{x}_{i_{\ell-1}}^{\mi_{\ell-2}} \in \LL_{b_{\ell-1}t, \ell-1, B, D}^{\mi_{\ell-2}}}}
	\eee^{-\sqrt{2}
		\frac{1}{\s_{\ell}} \left( \sqrt{2} \s_{\ell-1} b_{\ell-1} t - \tilde{x}_{i_{\ell-1}}^{\mi_{\ell-2}}(b_{\ell-1}t)  \right)
		+ \frac{3}{2} \big( \log(b_{\ell-1}t) + 2 \log ( \s_{\ell-1} - \s_{\ell} ) \big)} \nonumber \\
	& \hspace{2.8cm} \times
	\left( \sfrac{1}{\s_{\ell}} \left( \sqrt{2} \s_{\ell-1} b_{\ell-1} t - \tilde{x}_{i_{\ell-1}}^{\mi_{\ell-2}}(b_{\ell-1}t)  \right)
	+ L_{\ell-1}(t) + y \right).
\end{align}

The exponential in \eqref{eq:above_middle_start} is well approximated by its linear approximation. To see this, we use the inequality
\begin{equation} \label{eq:above_inequality}
	1 - x \le \eee^{-x} \le 1 - x + \frac{x^2}{2}, \quad x > 0,
\end{equation}
with
\begin{equation} \label{eq:above_middle_x}
	x = C \sqrt{\pi} 
	\eee^{-\sqrt{2} \, \left( \frac{1}{\s_{\ell}} \Lambda^{\mi_{\ell-2}}(t) + L_{\ell-2}(t) + y \right)} \
	\left( 
	\sfrac{1}{\s_{\ell}} \Lambda^{\mi_{\ell-2}}(t) \,  \YY_{i_{\ell-2}} (t) + \ZZ_{i_{\ell-2}}(t) 
	\right).
\end{equation}
This gives the bounds
\begin{align} \label{eq:above_middle_inequality_E}
	1 - \E \left[ x \, \vert \, \FF_{a_{\ell-2}t} \right]
	\le \E \left[ \eee^{-x} \, \vert \, \FF_{a_{\ell-2}t}\right]
	\le 1 - \E \left[ x \, \vert \, \FF_{a_{\ell-2}t} \right] 
	\left( 1 - \frac{1}{2} \frac{ \E \left[ x^2 \, \vert \, \FF_{a_{\ell-2}t} \right]}{\E \left[ x \, \vert \, \FF_{a_{\ell-2}t} \right]} \right).
\end{align}
The next two lemmas provide the necessary bounds on the conditional first and second moments of $x$.
Their proofs are postponed to Appendix~\ref{sec:above_proofMoments}. 
\begin{lemma} \label{lem:above_middle_first_moment}
	For all $i_1 \le n(b_1t), \dots, i_{\ell-2} \le n^{\mi_{\ell-3}} (b_{\ell-2}t) $ with
	\be 
	\tilde{x}_{i_1} \in \LL_{b_1t, 1, t^{\b}, B, D }, \dots,
	\tilde{x}_{i_{\ell-2}}^{\mi_{\ell-3}} \in \LL_{b_{\ell-2}, \ell-2, B, D}^{\mi_{\ell-3}},\ee 
	\begin{align}
		\E \, \Big[ \YY_{i_{\ell-2}}(t) \, \big\vert \, \FF_{a_{\ell-2}t} \Big]
		&= \sfrac{1}{\sqrt{\pi}}
		\Lambda^{\mi_{\ell-2}}(t)
		\eee^{ \sqrt{2} \Lambda^{\mi_{\ell-2}}(t) \frac{\s_{\ell-1} - \s_{\ell}}{\s_{\ell-1} \s_{\ell}}} 
		\sfrac{ \s_{\ell-1} - \s_{\ell} }{\s_{\ell-1}^3 \sqrt{2}}
		\, (1+o(1)), \\
		\E \, \Big[ \ZZ_{i_{\ell-2}}(t) \, \big\vert \, \FF_{a_{\ell-2}t} \Big]
		&= \sfrac{1}{\sqrt{\pi}}
		\Lambda^{\mi_{\ell-2}}(t)
		\eee^{ \sqrt{2} \Lambda^{\mi_{\ell-2}}(t) \frac{\s_{\ell-1} - \s_{\ell}}{\s_{\ell-1} \s_{\ell}}}
		\!\left( \sfrac{1}{ \s_{\ell-1}^3 \s_{\ell} } -
		\Lambda^{\mi_{\ell-2}}(t) \sfrac{ \s_{\ell-1} - \s_{\ell} }{ \s_{\ell-1}^3 \s_{\ell} \sqrt{2}} \right)
		\! (1+o(1)) \nonumber,
	\end{align}
	where $o(1)$ tends to $0$ as first $t \uparrow \infty$ and then $B \uparrow \infty, D \downarrow 0$.
	
\end{lemma} 
To bound the conditional second moment of $x$, we write it as
\begin{align} \label{eq:above_middle_second_moment}
	\E &\left[ x^2 \, \vert \, \FF_{a_{\ell-2}t} \right] \nonumber \\
	= \, &C^2 \pi
	\eee^{- 2\sqrt{2} \, \left( \frac{1}{\s_{\ell}} \Lambda^{\mi_{\ell-2}}(t) + L_{\ell-2}(t) + y \right)} \
	\E \left[
	\left( 
	\frac{1}{\s_{\ell}} \Lambda^{\mi_{\ell-2}}(t) \,  \YY_{i_{\ell-2}} (t) + \ZZ_{i_{\ell-2}}(t) 
	\right)^2 \bigg\vert \, \FF_{a_{\ell-2}t} \right] \nonumber \\
	\le \, &C^2 \pi
	\eee^{- 2\sqrt{2} \, \left( \frac{1}{\s_{\ell}} \Lambda^{\mi_{\ell-2}}(t) + L_{\ell-2}(t) + y \right)} \notag\\
	&\times\E \, \Bigg[ \Bigg(
	\sum_{\substack{i_{\ell-1} \le n^{\mi_{\ell-2}}(b_{\ell-1}t), \\ \tilde{x}_{i_{\ell-1}}^{\mi_{\ell-2}} \in \LL_{b_{\ell-1}t, \ell-1, B, D}^{\mi_{\ell-2}}}}
	\eee^{
		-\sqrt{2} \frac{1}{\s_{\ell}} \left( \sqrt{2} \s_{\ell-1} b_{\ell-1} t - \tilde{x}_{i_{\ell-1}}^{\mi_{\ell-2}}(b_{\ell-1}t) \right)
	}
	\Bigg)^2 \, 
	\bigg\vert \, \FF_{a_{\ell-2}t} \Bigg] \nonumber \\
	&\times
	( b_{\ell-1} t )^3 ( \s_{\ell-1} - \s_{\ell} )^6
	\left(
	\sfrac{ B }{ \s_{\ell} (\s_{\ell-1} - \s_{\ell}) }
	+ L_{\ell-1}(t) + y
	\right)^2.
\end{align}
In the last inequality, we bounded the non-exponential factors in  $\ZZ_{i_{\ell-2}}(t) $ with the help of the localisation $\tilde{x}_{i_{\ell-1}}^{\mi_{\ell-2}} \in \LL_{b_{\ell-1}t, \ell-1, B, D}^{\mi_{\ell-2}}$. 
By Assumption~\ref{as:above3}, the last line in \eqref{eq:above_middle_second_moment} is at most of order $t^3$.
The next lemma provides a bound on the expectation in \eqref{eq:above_middle_second_moment}.
\begin{lemma} \label{prop:above_middle_second_moment}
	For all $i_1 \le n(b_1t), \dots, i_{\ell-2} \le n^{\mi_{\ell-3}} (b_{\ell-2}t) $ as in Lemma 
	\ref{lem:above_middle_first_moment}
	\begin{align}
		\E \,& \Bigg[ \Bigg(
		\sum_{\substack{i_{\ell-1} \le n^{\mi_{\ell-2}}(b_{\ell-1}t), \\ \tilde{x}_{i_{\ell-1}}^{\mi_{\ell-2}} \in \LL_{b_{\ell-1}t, \ell-1, B, D}^{\mi_{\ell-2}}}}
		\hspace{-1em}
		\exp \left(
		{\textstyle
			-\sqrt{2} \frac{1}{\s_{\ell}} \left( \sqrt{2} \s_{\ell-1} b_{\ell-1} t - \tilde{x}_{i_{\ell-1}}^{\mi_{\ell-2}}(b_{\ell-1}t) \right) }
		\right)
		\Bigg)^2 \, 
		\bigg\vert \, \FF_{a_{\ell-2}t} \Bigg] \nonumber \\
		\le \, & P(t)
		\exp \left( - \tilde{C} t^{\beta} \right)
		\exp \left( \sqrt{2} \Lambda^{\mi_{\ell-2}}(t)
		\left( \sfrac{\s_{\ell-1} - \s_{\ell}}{\s_{\ell-1} \s_{\ell}}
		+ \sfrac{1}{\s_{\ell}} \right) 
		\right), 
	\end{align}
	with  $\tilde{C} > 0$ a  constant and $P$ is a function satisfying $P(t) \le t^c$ for some  $c>0$.
\end{lemma}
Comparing the first and second-order terms, we obtain with Lemma~\ref{lem:above_middle_first_moment} and ~\ref{prop:above_middle_second_moment}
\begin{align} \label{eq:above_middle_compare_moments}
	\frac{ \E \, \left[ x^2 \, \big\vert \, \FF_{a_{\ell-2}t} \right]}{\E \, \left[ x \, \big\vert \, \FF_{a_{\ell-2}t} \right]} 
	\le \frac{1}{\Lambda^{\mi_{\ell-2}}(t)} P(t) \eee^{- \tilde{C} t^{\beta} }
	\le D ( \s_{\ell-2} - \s_{\ell-1} ) P(t) \eee^{- \tilde{C} t^{\beta} }.
\end{align}
We used the localisation $x_{i_{\ell-2}}^{\mi_{\ell-3}} \in \LL_{b_{\ell-2}t, \ell-2. B, D}^{\mi_{\ell-3}}$ to bound $\Lambda^{\mi_{\ell-2}}(t)$.  By Assumption~\ref{as:above3},
the right-hand side of \eqref{eq:above_middle_compare_moments} converges to zero, as $t \uparrow \infty$.
Together with  \eqref{eq:above_middle_inequality_E}, this yields
\begin{align} \label{eq:above_middle_end}
	\E \, \Big[ \eee^{-x} \, \vert \,  \FF_{a_{\ell-2}t} \Big]
	= \, &\Big(
	1 - \E \left[ x \, \vert \, \FF_{a_{\ell-2}t} \right]
	\Big) \, (1+o(1)) \nonumber  \\
	= \, &\left(
	1 - C 
	\Lambda^{\mi_{\ell-2}}(t)
	\eee^{-\sqrt{2} \, \left( \frac{1}{\s_{\ell-1}} \Lambda^{\mi_{\ell-2}}(t) + L_{\ell-2}(t) + y \right)} 
	\right) \, (1+o(1)) \nonumber  \\
	= \, &\exp \left( -
	C 
	\Lambda^{\mi_{\ell-2}}(t)
	\eee^{ - \sqrt{2} \, \left( \frac{1}{\s_{\ell-1}} \Lambda^{\mi_{\ell-2}}(t) + L_{\ell-2}(t) + y \right)}
	\right)  \, (1+o(1)).
\end{align}
The second equality follows from Lemma~\ref{lem:above_middle_first_moment}.
In the third equality, we used \eqref{eq:above_inequality} together with $\Lambda^{\mi_{\ell-2}}(t) > D (\s_{\ell-2} - \s_{\ell-1})^{-1}$ and Assumption~\ref{as:above3}.
\end{proof}

Inserting the right-hand side of \eqref{eq:above_interative_1} into the nested expectations in \eqref{eq:above_branch}, we see that the innermost expectation in \eqref{eq:above_branch} now is of the form
\begin{equation} \label{eq:above_iterative_2}
\E \, \Bigg[ \exp \bigg(-
\sum_{\substack{ i_k \le n^{\mi_{k-1}} (b_kt) \\ \tilde{x}_{i_k}^{\mi_{k-1}} \in \LL_{ b_kt, k, B, D}^{\mi_{k-1}}}}
C 
\Lambda^{\mi_k} (t)
\eee^{ -\sqrt{2} \, \left( \frac{1}{\s_k} \Lambda^{\mi_k}(t) + L_k(t) + y \right)}
\bigg) \, \bigg\vert \, \FF_{a_{k-1}t} \Bigg]
\end{equation}
with $k = \ell-2$.
By the same arguments as in the proof of Proposition~\ref{prop:above_iteration}, we get that \eqref{eq:above_iterative_2} equals
\begin{equation} \label{eq:above_iterative_3}
\exp \left( -
C 
\Lambda^{\mi_{k-1}}(t)
\eee^{ - \sqrt{2} \, \left( \frac{1}{\s_{\ell}} \Lambda^{\mi_{k-1}}(t) + L_{k-1}(t) + y \right)}
\right)  \, (1+o(1)),
\end{equation}
where the error is uniform in the range of possible values of $\tilde{x}_{i_{k-1}}^{\mi{k_2}}$ since $\tilde{x}_{i_{k-1}}^{\mi{k_2}} \!\in \LL_{b_{k-1}t, k-1, B, D}^{\mi_{k-2}}$.
Iterating this procedure for $k=\ell-3, \dots, 2$, we see that \eqref{eq:above_branch} is equal to
\begin{align} \label{eq:above_first_start}
\E \, \Bigg[ \prod_{\substack{i_1 \le n(b_1t), \\ \tilde{x}_{i_1} \in \LL_{b_1t, 1, t^{\b}, B, D}}}&
\exp \, \bigg( {-} C 
\eee^{- \sqrt{2} y} \, 
\bigr( \sqrt{2} \s_1 b_1 t - \tilde{x}_{i_1}(b_1t) \bigr)  \\
&\times\eee^{-\sqrt{2} \frac{1}{\s_2} \left( \sqrt{2} \s_1 b_1 t - \tilde{x}_{i_1}(b_1t) \right)
	+ \frac{3}{2} \big( \log(b_1t) + 2 \log( \pi^{1/6} (\s_1 - \s_2)) \big) }
\bigg) \Bigg] (1 + o(1)). \nonumber 
\end{align}
To compute the expectation, we split the range of integration  $[0, b_1t)$ into $[0,t^{\b})$ and $[t^{\b}, b_1 t)$, and abbreviate the length of the second part by $t^* \equiv b_1t - t^{\beta}$. 
We choose $\b \in (0,1/2)$ such that Assumption~\ref{as:above3} is satisfied.
Denote by $\{ \tilde{x}_j(t^{\b}) \colon 1 \le j \le n(t^{\b})\}$ the particles of a BBM with variance $\s_1^2$ at time $t^{\b}$, and, for each $j$, denote by \mbox{$\{ \tilde{x}_{j^*}^j (t^*) \colon 1 \le j^* \le n(t^*) \}$} the particles of independent BBMs with variance $\s_1^2$ at time $t^*$.
We set
\begin{align}
\widetilde{\LL}_{t^*, 1, \tilde{x}_j \left( t^{\b} \right) - \sqrt{2} \s_1 t^{\b}, B, D}
&\equiv \AA_{0, t^*, t^{\b \delta/2} + \tilde{x}_j \left( t^{\b} \right) - \sqrt{2} \s_1 t^{\b}, \s_1}
\cap \GG_{t^*, 1, \tilde{x}_j(t^{\b}) - \sqrt{2} \s_1 t^{\b}, B, D}, \nonumber \\
\widetilde{\LL}_{t^{\beta}, t^{\beta \delta}, \s_1}
&\equiv 
\AA_{t^{\beta}, t^{\beta}, t^{\beta \delta}, \s_1}
\cap \BB_{t^{\beta}, t^{\beta}, 0, \s_1}.
\end{align}
By conditioning on $\FF_{t^{\b}}$, we rewrite \eqref{eq:above_first_start} as
\begin{align} \label{eq:above_first_split}
\E\,  \Bigg[ \prod_{\substack{j \le n(t^{\b}) \\ \tilde{x}_j \in \widetilde{\LL}_{t^{\beta}, t^{\beta \delta}, \s_1}}}
\E \,& \Bigg[ \exp \bigg(- C \sqrt{\pi}
\eee^{-\sqrt{2} \,  \left( \frac{1}{\s_2} \left(\sqrt{2} \s_1 t^{\b} - \tilde{x}_j \left( t^{\b} \right) \right) + y \right)} \nonumber \\
& \times
\left( \left(\sqrt{2} \s_1 t^{\b} - \tilde{x}_j \left( t^{\b} \right) \right) \widetilde{\YY}_j(t) + \widetilde{\ZZ}_j(t) \right)\! \bigg) \,  \bigg\vert \,  \FF_{t^{\b}} \Bigg] \Bigg] (1+o(1)),
\end{align}
where
\begin{align} \label{eq:above_first_splitted}
\widetilde{\YY}_j(t) 
&\equiv \hspace{-1.5em}
\sum_{\substack{j^* \le n^j(t^*), \\ \tilde{x}_{j^*}^{\, j} \in \widetilde{\LL}_{t^*, 1, \tilde{x}_j(t^{\b}) - \sqrt{2} \s_1 t^{\b}, B, D}}} 
\hspace{-2.5em}
\eee^{-\sqrt{2} \, \frac{1}{\s_2} \left( \sqrt{2} \s_1 t^* - \tilde{x}_{j^*}^{\, j}(t^*) \right)
	+ \frac{3}{2} \big( \log(b_1t) + 2 \log ( \s_1 - \s_2 ) \big)}, \\
\widetilde{\ZZ}_j(t) 
&\equiv \hspace{-1.5em}
\sum_{\substack{j^* \le n^j(t^*), \\ \tilde{x}_{j^*}^{\, j} \in \widetilde{\LL}_{t^*, 1, \tilde{x}_j(t^{\b}) - \sqrt{2} \s_1 t^{\b}, B, D}}}
\hspace{-2.5em}
\left(\sqrt{2} \s_1 t^* - x_{j^*}^{\, j}(t^*) \right)
\eee^{-\sqrt{2} \frac{1}{\s_2} \left(\sqrt{2} \s_1 t^* - \tilde{x}_{j^*}^{\, j}(t^*) \right)
	+ \frac{3}{2} \big( \log(b_1t) + 2 \log ( \s_1 - \s_2 ) \big)}\!. \nonumber 
	\end{align}
	The proofs of the following lemmas are very similar to the corresponding results in \cite{1to6} and use the same technique as in \eqref{eq:above_middle_start} -- \eqref{eq:above_middle_end}. We give the details in Appendix A for completeness. 
	We set
	\begin{equation} \label{eq:above_first_x}
z
=  C \sqrt{\pi}
\eee^{-\sqrt{2} \, \left( \frac{1}{\s_2} \left(\sqrt{2} \s_1 t^{\b} - \tilde{x}_j \left( t^{\b} \right)\right) + y \right) }
\left( \left(\sqrt{2} \s_1 t^{\b} - \tilde{x}_j(t^{\b})\right) \widetilde{\YY}_j(t) + \widetilde{\ZZ}_j(t) \right).
\end{equation}
In the first step, we compute the first moments of $\widetilde{\YY}_j(t)$ and $\widetilde{\ZZ}_j(t)$.
\begin{lemma} \label{lem:above_first_moment_first}
For all $j \le n(t^{\b})$ with $\tilde{x}_j \in \LL_{t^{\b}, t^{\b \delta}, \s_1}$,
\begin{align}
	\E \, \Big[ \widetilde{\YY}_j(t) \, \vert \, \FF_{t^{\b}} \Big]
	&= \sfrac{1}{ \sqrt{ \pi } }
	\left( \sqrt{2} \s_1 t^{\b} - \tilde{x}_j ( t^{\b} ) \right) 		\sfrac{ \s_1 - \s_2 }{\s_1^3 \sqrt{2}}(1+o(1)), \\
	\E \, \Big[ \widetilde{\ZZ}_j(t) \, \vert \, \FF_{t^{\b}} \Big]
	&=  \sfrac{1}{\sqrt{\pi}}
	\left(\sqrt{2} \s_1 t^{\b} - \tilde{x}_j ( t^{\b} ) \right)		\sfrac{1}{\s_1^3}
	\left( 1
	-
	\left(\sqrt{2} \s_1 t^{\b} - \tilde{x}_j ( t^{\b} ) \right)
	\sfrac{ \s_1 - \s_2 }{ \sqrt{2}}
	\right) (1+o(1)), \nonumber 
\end{align}
where $o(1)$ tends to $0$ as first $t \uparrow \infty$ and then $B \uparrow \infty, D \downarrow 0$.
\end{lemma}
We bound  $\E[z^2\, \vert \, \FF_{t^{\b}}]$ from above by 
\begin{align} \label{eq:above_first_second_moment}
&
\E \, \bigg[ \bigg( 
\hspace{-0.7em}
\sum_{\substack{j^* \le n^j(t^*), \\ \tilde{x}_{j^*}^{\, j} \in \widetilde{\LL}_{t^*, 1,  \tilde{x}_j(t^{\b}) - \sqrt{2} \s_1 t^{\b}, B, D}}}
\hspace{-1.8em}
\exp \left(
{\textstyle-\sqrt{2} \frac{1}{\s_2} \left( \sqrt{2} \s_1 t^* - \tilde{x}_{j^*}^{\,  j}(t^*) \right) } 
\right)
\bigg)^2 \, \Big\vert \, \FF_{t^{\b}} \bigg] \\
&\times
\eee^{-2\sqrt{2} \, \left( \frac{1}{\s_2} \left(\sqrt{2} \s_1 t^{\b} - \tilde{x}_j \left( t^{\b} \right)\right) + y \right) } 
C^2 \pi 
( b_1t)^3
( \s_1 - \s_2)^6
\left( \left( \sqrt{2} \s_1 t^{\b} - \tilde{x}_j ( t^{\b}) \right) 
+ \sfrac{ B }{ \s_1 - \s_2 } \right)^2, \nonumber 
\end{align}
where we used 
the localisation $\tilde{x}_{j^*}^{\, j} \in \widetilde{\LL}_{t^*, 1, \tilde{x}_j(t^{\b}) - \sqrt{2} \s_1 t^{\b}, B, D}$ to bound the non-exponential factors in $\widetilde{\ZZ}_j(t)$ by $( \s_1 - \s_2 )^{-1} B$. 
Due to $\tilde{x}_j \in \LL_{t^{\b}, t^{\b \delta}, \s_1}$, $\sqrt{2} \s_1 t^{\b} - \tilde{x}_j(t^\b) \in (t^{\b \delta}, 2 \sqrt{2} t^{\b})$. 
Together with Assumption~\ref{as:above3}, we see that the non-exponential factors in the last line of \eqref{eq:above_first_second_moment} are bounded by $C^2 \pi t^{3+2\b}$.
It remains to estimate the conditional expectation in 
\eqref{eq:above_first_second_moment}.
\begin{lemma} \label{prop:above_second_moment_first}
For all $j \le n(t^{\b})$ with $\tilde{x}_j \in \LL_{t^{\b}, t^{\b \delta}, \s_1}$,
\begin{align}
	\E \, \bigg[ \bigg( 
	\hspace{-0.5em}
	\sum_{\substack{j^* \le n^j(t^*), \\ \tilde{x}_{j^*}^{\, j} \in \widetilde{\LL}_{t^*, 1, \tilde{x}_j(t^{\b}) - \sqrt{2} \s_1 t^{\b}, B, D}}}
	\hspace{-1.5em}
	\exp \left(
	{ \textstyle-\sqrt{2} \frac{1}{\s_2} \left( \sqrt{2} \s_1 t^* - \tilde{x}_{j^*}^{\, j}(t^*) \right) } 
	\right)\!
	\bigg)^2 \, \bigg\vert \, \FF_{t^{\b}} \bigg]
	\le \,  P(t) \exp \left( - \tilde{C} t^{\b} \right)\!,
\end{align}
where we use the notation from Lemma~\ref{prop:above_middle_second_moment}.
\end{lemma}
Comparing the first and second moments terms of $z$, we obtain with Lemma~\ref{lem:above_first_moment_first} and~\ref{prop:above_second_moment_first}
\begin{equation}
\frac{ \E \left[ z^2 \, \big\vert \, \FF_{t^{\beta}} \right] }{ \E \left[ z \, \big\vert \, \FF_{t^{\beta}} \right]}
\le \frac{1}{ \sqrt{2} \s_1 t^{\b} - \tilde{x}_j \left( t^{\b} \right) }
P(t)
\eee^{ - \tilde{C} t^{\b} }
\le t^{-\b \delta} P(t) \eee^{ - \tilde{C} t^{\b} }.
\end{equation}
The upper bound is due to $\tilde{x}_j \in \LL_{t^{\b}, t^{\b \delta}, \s_1}$ and converges to zero as $t \uparrow \infty$.
Next, we  approximate $\E \left[ \eee^{-z} \, \vert \, \FF_{t^{\beta}} \right]$ as in \eqref{eq:above_middle_end}:
\begin{align} \label{eq:above_first_approx}
\E  \Big[ \eee^{-z}  \Big\vert  \FF_{t^{\beta}} \Big]
=  &\exp \!\bigg( {-}  C 
\left(\sqrt{2} \s_1 t^{\b} - \tilde{x}_j (t^{\b}) \right) 	\eee^{-\sqrt{2}  \left( \frac{1}{\s_2} \left(\sqrt{2} \s_1 t^{\b} - \tilde{x}_j \left( t^{\b} \right) - t^{\b \delta} \right) + y \right) }
\bigg)  (1 + o(1)).
\end{align}
Inserting the right-hand side of \eqref{eq:above_first_approx} into \eqref{eq:above_first_split}, we see that \eqref{eq:above_first_split} is equal to
\begin{align} \label{eq:above_towards_start}
\E \, \Bigg[ 
\exp \Bigg( 
\hspace{-0.2em}- 
\hspace{-0.2em} C 
\eee^{- \sqrt{2} y }
\hspace{-1.3em}
\sum_{\substack{j \le n(t^{\b}) \\  \tilde{x}_j \in \widetilde{\LL}_{t^{\beta}, t^{\beta \delta}, \s_1}}}
\hspace{-1.2em}
\left(\sqrt{2} \s_1 t^{\b} - \tilde{x}_j ( t^{\b} ) \right) 			\eee^{-\sqrt{2} \frac{1}{\s_2} \left(\sqrt{2} \s_1 t^{\b} - \tilde{x}_j \left( t^{\b} \right) \right)}
\Bigg) \Bigg] (1+o(1)).
\end{align}

Finally, the sum in \eqref{eq:above_towards_start} converges to the limit of the derivative martingale.
\begin{lemma} \label{prop:above_derivative_martingale}
With the notation above,
\begin{align} \label{eq:above_towards_claim}
\sum_{j \le n(t^{\b})} 
\1_{ \left\{ \tilde{x}_j \in \widetilde{\LL}_{t^{\beta}, t^{\beta \delta}, \s_1} \right\} }
\left(\sqrt{2} \s_1 t^{\b} - \tilde{x}_j ( t^{\b} ) \right) 		\eee^{-\sqrt{2} \frac{1}{\s_2} \left( \sqrt{2} \s_1 t^{\b} - \tilde{x}_j \left( t^{\b} \right) \right)} 
\to Z, 
\end{align}
in probability, as $t \uparrow \infty$, where $Z$ is the limit of the derivative martingale.
\end{lemma}

\begin{proof}
This lemma is a generalisation of Lemma 3.7 in \cite{1to6} where the case \mbox{$\s_1^2 = 1 + t^{-\a},$} \mbox{$\s_2^2 = 1- t^{-\a},$} $\a \in (0, 1/2),$ was considered. 
The proof of Lemma~\ref{prop:above_derivative_martingale} is a rerun this proof provided $\s_1/\s_2 \ge 1$ and $t^{\b}(\s_1-\s_2) \downarrow 0$ as $t \uparrow \infty$.
But these properties follow directly from Assumption~\ref{as:above} and we are done.
\end{proof}
Due to Lemma~\ref{prop:above_derivative_martingale} and $Z>0$ a.s., \eqref{eq:above_towards_start} converges to
\begin{equation}
\E \left[ \exp \left( - C 
Z \eee^{-\sqrt{2} y} \right) \right],
\end{equation}
as $t \uparrow \infty$.
This proves Theorem~\ref{thm:lawofmax} in Case~A.	

\subsection[Proof of Theorem 1.3 in Case~B]{Proof of Theorem~\ref{thm:lawofmax} in Case~B}\label{sec:below}\label{sec:proof_CaseBLinear}
In this subsection, we prove Theorem~\ref{thm:lawofmax} in Case~B.
We begin  with the case of \mbox{$\ell$-speed} BBM and write $\a_1,\a_\ell$ instead of $\ab, \ae$, and $b_1, b_\ell$ instead of $\bbeg, \bend$ to be consistent with Case~A.
To complete the proof for the more general setting of Theorem~\ref{thm:lawofmax}, we use Gaussian comparison techniques. 
\begin{assumption}\label{as:belowLSpeed}
Let $\a_1, \a_\ell \in \kl{0, 1/2}$. We assume that the family of speed functions $\kl{A_t}_{t>0}$ with $A_t(s) < s$ for all $t>0, s\in (0,1)$, satisfies:
\begin{thmlist}
\item \label{as:belowLSpeed1}
\emph{The functions $(A_t)_{t>0}$ are piecewise linear and continuous}:
Their derivatives are given by
\begin{equation}
	A_t'(s) = 
	\sum_{k=1}^{\ell} \s_k^2 \1_{ \left( \sum_{j=1}^{k-1} b_j, \sum_{j=1}^{k} b_j \right)}(s),
\end{equation}
with velocities $\s_k\colon\R_+ \to \R_+, 1 \le k \le \ell,$ and interval lengths \mbox{$b_k\colon \R_+ \to (0,1]$}, \mbox{$1 \le k \le \ell$}. 
We assume for all $t>0$ that $\sum_{k=1}^{\ell} b_k=1$ and $\sum_{k=1}^{\ell} \s_k^2 b_k = 1$.
\item  \label{as:belowLSpeed2}
The velocity on the first interval satisfies $\s_1^2 = 1 - t^{-\a_1} + o(t^{-\a_1})$ and the interval length $b_1$ satisfies $1 \gg b_1 \gg t^{\a_1-1/2}$ as $t \uparrow \infty$.
\item \label{as:belowLSpeed3}
The velocity on the last interval satisfies $\s_\ell^2 = 1 + t^{-\a_\ell} + o(t^{-\a_\ell})$ and the interval length $b_\ell$ satisfies $1 \gg b_\ell \gg t^{\a_\ell-1/2}$ as $t \uparrow \infty$.
\item \label{as:belowLSpeed4}
\emph{Minimum distance to the identity function:} $\min_{s\in[b_1, 1-b_\ell]} \bkl{s - A_t(s)} \gg t^{-1/2}$ as $t\uparrow\infty$.
\end{thmlist}

\end{assumption}

From now on, we 
write $(\bar{X}_t)_{t>0}$ for $\ell$-speed BBM with speed functions $(A_t)_{t>0}$ satisfying Assumption~\ref{as:belowLSpeed}. 
\begin{proposition}\label{prop:BelowLSpeedLimitingLaw}
Let $C$ be the positive constant from Proposition~\ref{prop:FKPP_tail_estimate}
and	denote by $Z$ the limit of the derivative martingale. Then, for all $ y \in \R$,
\begin{equation}
\lim_{t \uparrow \infty} \P \left( \max_{j \le n(t)} \bar{x}_j(t) - m^{-}(t) \le y \right)  
= \E \left[ \eee^{- C Z \eee^{-\sqrt{2}y}} \right],
\end{equation} 
where
\begin{align}
m^-(t) = \sqrt{2}t - \sfrac{1 + 2 (\a_1 + \a_\ell)}{2\sqrt{2}}  \log(t).
\end{align}
\end{proposition}


\noindent\emph{Proof of Proposition~\ref{prop:BelowLSpeedLimitingLaw}.} The structure of this proof is identical to that in Case A, which is given in Subsection~\ref{sec:proof_caseA}. 
Let $\e > 0$, $r>0$, $\g>1/2$ and $0 < \b < \a_1$  be such that $t^\b \ll b_1 t$.
Lemma~\ref{lem:BarrierLok} and Propositions~\ref{prop:below_bridgeLok} and~\ref{prop:below_entropic_rep} imply
that
\begin{equation}
\P \, \bigg( 
\exists_{j \le n(t)} \colon 
\Big\{ 
\bar{x}_j(t) > m^-(t)-y 
\Big\} \land \, 
\bigg\{  
\bar{x}_j \notin
\LL
\bigg\} 
\bigg) < \e,
\end{equation}
for all $r, t$ large enough, where
\begin{align}
\LL &= \LL_{b_1 t} \cap \TT_{b_1t, (1-b_\ell)t, 0, \g},\notag\\
\LL_{b_1 t} &=  \BB_{t^\b,t^\b,0,\s_1} \cap \AA_{t^\beta,t^\b, t^{\beta\delta}, \s_1} \cap \AA_{t^\beta, b_1 t, t^{\beta\delta/2}, \s_1} \cap  \TT_{b_1 t, b_1 t, 0, \g},
\end{align}
recalling  
\begin{align}
\AA_{r_1, r_2, S, \s}
&= \left\{ X \colon \forall_{r_1 \le s \le r_2} \colon X(s) + S \le \sqrt{2} \s s \right\},\notag\\
\BB_{r_1, r_2, S, \s}
&= \left\{ X \colon \forall_{r_1 \le s \le r_2} \colon X(s) + S > -\sqrt{2} \s s \right\}.
\end{align}
By Lemma~\ref{lem:localisation}, Proposition~\ref{prop:BelowLSpeedLimitingLaw} will follow from 
\begin{align}
&\lim_{t\uparrow\infty}\P \left( \max_{j \le n(t)\colon \bar{x}_j \in \LL} \bar{x}_j(t) - m^{-}(t) \le y \right) = \E \Big[
\exp\bbkl{
{-}CZ \eee^{-\sqrt{2} y} 
}\Big].
\label{eq:bRec2}
\end{align}
Proceeding as in \eqref{eq:above_branch-1}--\eqref{eq:above_branch}, we write the probability in \eqref{eq:bRec2} as
\begin{align}
&	\E \, \Bigg[ 
\prod_{\substack{
	i_1 \le n(b_1t), \\ 
	\bar{x}_{i_1} \in\LL_{b_1 t}
	}}
	\E \, \Bigg[ 
	\prod_{\substack{
	i_2 \le n^{i_1}(b_2t), \dots, i_{\ell-1} \le n^{\mi_{\ell-2}}(b_{\ell-1}t), \\
	\sum_{k=1}^{\ell-1} \multixb{k} \in \TT_{b_1t, (1-b_\ell)t, 0, \g}
	}}\label{eq:below_branch}\\
	&\qquad\times
	\Bigg( 
	1 - 
	\P \, 
	\Bigg( 
	\max_{ i_{\ell} \le n^{\mi_{\ell-1}}(b_{\ell}t) } 
	\multixb{\ell}(b_{\ell} t)
	> 
	m^-(t)
	- \sum_{k=1}^{\ell-1} \multixb{k}(b_k t)
	+ y \, 
	\bigg\vert \, \FF_{(1-b_\ell)t} 
	\Bigg)
	\Bigg) \, 
	\bigg \vert \, \FF_{b_1t} \Bigg] \Bigg]. \notag
\end{align}
Since $\sum_{k=1}^{\ell-1} \multixb{k} \in \TT_{b_1t, (1-b_\ell)t, 0, \g}$, the F-KPP asymptotics from Proposition~\ref{prop:FKPP_tail_estimate} imply
\begin{align}
\P  &\,
\Bigg( 
\max_{ i_{\ell} \le n^{\mi_{\ell-1}}(b_{\ell}t) } 
\multixb{\ell}(b_{\ell} t)
> 
m^-(t)
- \sum_{k=1}^{\ell-1} \multixb{k}(b_k t)
+ y \, 
\bigg\vert \, \FF_{(1-b_\ell)t} 
\Bigg)\notag\\
=\:&
C \Delta^{\mi_{\ell-1}}(t)
\exp\kl{
b_\ell t
-\sfrac{
	\kl{\Delta^{\mi_{\ell-1}}(t) + \sqrt{2} b_\ell 	t}^2
}{
	2b_\ell t
}
} \label{eq:below_tailAsymptotics}
(1+o(1)),
\end{align}
with 
\begin{align}
\label{eq:below_delta_def_Ende}
\Delta^{\mi_{\ell-1}}(t)
\equiv 
\sfrac{1}{\s_{\ell}} 
\left(
m^-(t) - \sum\limits_{k=1}^{\ell-1} \multixb{k}(b_k t) + y
\right)
- \left( 
\sqrt{2} b_{\ell} t - \sfrac{3}{2 \sqrt{2}} \log(b_{\ell} t)
\right).
\end{align}
Inserting \eqref{eq:below_tailAsymptotics} into \eqref{eq:below_branch} yields
\begin{align} \label{eq:below_afterFKPP}
&\E  \Bigg[ 
\prod_{\substack{
	i_2 \le n^{i_1}(b_2t), \dots, i_{\ell-1} \le n^{\mi_{\ell-2}}(b_{\ell-1}t), \\
	\sum_{k=1}^{\ell-1} \multixb{k} \in \TT_{b_1t, (1-b_\ell)t, 0, \g}
	}} 
	\Bigg(
	1 -
	C \Delta^{\mi_{\ell-1}}\!(t)
	\exp\bbbkl{
b_\ell t
-\sfrac{
	\kl{\Delta^{\mi_{\ell-1}}(t) + \sqrt{2} b_\ell 	t}^2
}{
	2b_\ell t
}
}(1+o(1))\!
\Bigg)
\bigg \vert \FF_{b_1t} \Bigg] 
\notag\\
&\!\!=\!
\E \Bigg[ 
\!\prod_{\substack{
	i_2 \le n^{i_1}(b_2t), \dots, i_{\ell-1} \le n^{\mi_{\ell-2}}(b_{\ell-1}t), \\
	\sum_{k=1}^{\ell-1} \multixb{k} \in \TT_{b_1t, (1-b_\ell)t, 0, \g}
	}}\!
	\exp\Bigg(
	{-} 
	C \Delta^{\mi_{\ell-1}}\!(t)
	\exp\!\bbbkl{\!
b_\ell t
-\sfrac{\!
	\kl{\!\Delta^{\mi_{\ell-1}}\!(t) + \sqrt{2} b_\ell 	t}^{\! 2}
}{
	2b_\ell t
}
}(1+o(1))\!
\Bigg)
\bigg \vert  \FF_{b_1t} \Bigg] 
\notag\\
&\!\!=\!
\E  \Big[ 
\exp\Big({-} 
\YY_{i_1}(t)
\Big) \, 
\big \vert \, \FF_{b_1t} \Big] 
(1+o(1)), 
\end{align}
where
\begin{align}\label{eq:below_DefY}
\YY_{i_1}(t) \equiv 
\sum_{\substack{
	i_2 \le n^{i_1}(b_2t), \dots, i_{\ell-1} \le n^{\mi_{\ell-2}}(b_{\ell-1}t), \\
	\sum_{k=1}^{\ell-1} \multixb{k} \in \TT_{b_1t, (1-b_\ell)t, 0, \g}
	}}
	C \Delta^{\mi_{\ell-1}}(t)
	\exp\kl{
b_\ell t
-\sfrac{
	\kl{\Delta^{\mi_{\ell-1}}(t) + \sqrt{2} b_\ell 	t}^2
}{
	2b_\ell t
}
}.
\end{align}
The first equality in \eqref{eq:below_afterFKPP} holds since, for $i_2, \dots, i_{\ell}$ from that product, the right-hand side of \eqref{eq:below_tailAsymptotics} tends to zero.
The error terms in the second line of \eqref{eq:below_afterFKPP} are uniform and we will prove later that $\YY_{i_1}(t)$ converges as $t\uparrow\infty$ for $x_{i_1} \in \LL_{b_1 t}$. Thus it is justified in the last line of \eqref{eq:below_afterFKPP} to write the error term outside the conditional expectation.
As in \eqref{eq:above_middle_inequality_E}, we use the inequality \eqref{eq:above_inequality}
to get 
that
\begin{align}
\label{eq:below_ExpEstimate}
1 \!- \!
\E
\ekl{
\mathcal{Y}_{i_1}(t) \,\big\vert\, \FF_{b_1 t}
}
\leq\E
\ekl{
\exp\big({-} 
\YY_{i_1}(t)
\big)
\,\big\vert\, 
\FF_{b_1 t}
}
\leq 
1 \!-\!   \E
\ekl{
\mathcal{Y}_{i_1}(t) \,\big\vert\, \FF_{b_1 t}
}
\kl{\!
1 \!-\!
\sfrac{1}{2}
\sfrac{
	\E\ekl{
		\mathcal{Y}^2_{i_1}(t) \,\big\vert\, \FF_{b_1 t}
	}
}{
	\E\ekl{
		\mathcal{Y}_{i_1}(t) \,\big\vert\, \FF_{b_1 t}
	} 
}
}\!.
\end{align} 
We postpone the proofs of the following (conditional) first and second moment estimates for $\mathcal{Y}_{i_1}(t) $ to Appendix~\ref{sec:below_first_second_moments}.
\begin{lemma}
\label{lem:below_FirstMomentsMiddlePart}
Let $\g>1/2$ be such that
\begin{align}
(\s_\ell^2 b_\ell t)^\g \ll b_\ell t^{1-\a_\ell}.
\end{align} 
For all $i_1$ appearing in the product in \eqref{eq:below_branch}, we have
\begin{align}
\label{eq:below_FirstMomentsMiddlePartClaim}
\E\ekl{
	\mathcal{Y}_{i_1}(t) \,\big\vert\, \FF_{b_1 t}
}
=
\sfrac{C}{\sqrt{2(1-\s_1^2 b_1)}}
b_\ell^{3/2} t^{1-\a_\ell}
\exp\kl{
	(1- b_1)t
	-
	\sfrac{
		\kl{
			\sqrt{2 t} - \bar{x}_{i_{1}}(b_{1}t) + L(t)+y
		}^2
	}{
		2 (1- \s_1^2 b_1) t
	}
}
(1+o(1)),
\end{align}
where $L(t)$ is as in \eqref{eq:below_defL}.
\end{lemma}
\begin{lemma}
\label{lem:below_SecondMomentsMiddlePart}
Let $\g>1/2$ be such that
\begin{align}\label{eq:below_midsm-1}
\min_{s\in[b_1, 1-b_\ell]} \kl{s - A_t(s)} t \gg t^\g.
\end{align}
For all $i_1$ appearing in the product in \eqref{eq:below_branch}, we have
\begin{align}
\label{eq:below_midsm0}
\E\ekl{
	\mathcal{Y}^2_{i_1}(t) \,\big\vert\, \FF_{b_1 t}
}
\leq
P(t) 
\exp\kl{
	(1- b_1)t
	-
	\sfrac{
		\kl{
			\sqrt{2 t} - \bar{x}_{i_{1}}(b_{1}t) + L(t)+y
		}^2
	}{
		2 (1- \s_1^2 b_1) t
	}
}
\eee^{-t^{1/2}}(1+o(1))
,
\end{align}
with $L(t)$ as in \eqref{eq:below_defL}. We write $P$ for any term satisfying $P(t) \leq t^c$ for some constant $c\in \R$ and for all $t>0$ large enough.
\end{lemma}
By 
Lemmas~\ref{lem:below_FirstMomentsMiddlePart} and \ref{lem:below_SecondMomentsMiddlePart},
\begin{align}
\label{eq:below_FractionSecondFirstMomentMiddle}
\lim_{t\uparrow\infty}
\sfrac{
\E\ekl{
	\mathcal{Y}^2_{i_1}(t) \,\big\vert\, \FF_{b_1 t}
}
}{
\E\ekl{
	\mathcal{Y}_{i_1}(t) \,\big\vert\, \FF_{b_1 t}
} 
}
=0.
\end{align}
From this and \eqref{eq:below_ExpEstimate} follows that
\begin{align}\label{eq:below_AfterAveragingMiddleInner}
\E
\ekl{
\eee^{-\mathcal{Y}_{i_1}(t) }\,\big\vert\, \FF_{b_1 t}
}
=
1 -   \E
\ekl{
\mathcal{Y}_{i_1}(t) \,\big\vert\, \FF_{b_1 t}
}(1+o(1)),
\end{align}
so \eqref{eq:below_branch} is equal to
\begin{align}\label{eq:below_AfterAveragingMiddle}
&
\E \, \Bigg[ 
\prod_{\substack{
	i_1 \le n(b_1t), \\ 
	\bar{x}_{i_1} \in\LL_{b_1 t}
	}}
	\bbkl{1 -   \E
\ekl{
	\mathcal{Y}_{i_1}(t) \,\big\vert\, \FF_{b_1 t}
	}}
	\Bigg] (1+o(1))\\
	&\!=\,
	\E  \Bigg[ 
	\prod_{\substack{
	i_1 \le n(b_1t), \\ 
	\bar{x}_{i_1} \in\LL_{b_1 t}
	}}\!\!
	\exp\!\bbbbkl{\!
{-}
\sfrac{C}{\sqrt{2(1-\s_1^2 b_1)}}
b_\ell^{3/2} t^{1-\a_\ell}
\!\exp\!\kl{\!
	(1- b_1)t
	\!-\!
	\sfrac{
		\kl{
			\sqrt{2 t} - \bar{x}_{i_{1}}(b_{1}t) + L(t)+y
		}^2
	}{
		2 (1- \s_1^2 b_1) t
	}\!
}\!\!
}\!
\Bigg] (1+o(1)).\notag
\end{align}
In the last step, we used Lemma~\ref{lem:below_FirstMomentsMiddlePart} and that, for $\bar{x}_{i_1} \in\LL_{b_1 t}$, the right-hand side of \eqref{eq:below_FirstMomentsMiddlePartClaim} tends to zero as $t\uparrow\infty$. 
As in 
\eqref{eq:above_first_split}, we rewrite the expectation on the right-hand side of \eqref{eq:below_AfterAveragingMiddle} by conditioning on $\FF_{t^{\beta}}$ as
\begin{align}
\E \, \Bigg[ 
\prod_{\substack{
	j \le n(t^{\b}) \\ 
	\bar{x}_j \in \LL_{t^\b}
	}}
	\E
	\ekl{
\eee^{-\tilde{\YY}_{j}(t) }\,\big\vert\, \FF_{t^\b}
}
\Bigg],\label{eq:below_first1}
\end{align}
where
\begin{align}
\tilde{\YY}_{j}(t) 
& \equiv
\sum_{\substack{
	j^* \le n^j(t^*), \\
	\bar{x}_{j^*}^{\, j} \in \LL_{t^*}
	}}
	\sfrac{C}{\sqrt{2(1-\s_1^2 b_1)}}
	b_\ell^{3/2} t^{1-\a_\ell}
	\exp\kl{
(1- b_1)t
-
\sfrac{
	\kl{
		\sqrt{2 t} - \bar{x}_{j^*}^{\, j}(t^*) - \bar{x}_j(t^\b) + L(t)+y
	}^2
}{
	2 (1- \s_1^2 b_1) t
}
},\notag\\
\LL_{t^\b} & \equiv
\AA_{t^{\b}, t^{\b}, t^{\b\delta}, \s_1} \cap \BB_{t^\b,t^\b,0,\s_1},\notag\\
\LL_{t^*} & \equiv \AA_{0, t^*, t^{\b\delta/2} + \bar{x}_j(t^\b) - \sqrt{2}\s_1 t^\b, \s_1} \cap \TT_{t^*,t^*, \bar{x}_j(t^\b) - \sqrt{2}\s_1 t^\b, \g}.
\label{eq:below_DefY2}
\end{align}
Recall that $t^* = b_1 t - t^\b$ and, independently for each $j$, $(\bar{x}_{j^*}^{j}(t^*))_{j^*\leq n^j(t^*)}$ are the particles of a BBM with variance $\s_1^2$ at time $t^*$. 
We postpone the proofs of the estimates of the conditional first and second moments to Appendix~\ref{sec:below_first_second_moments}.
\begin{lemma}
\label{lem:below_FirstMomentsFirstPart}
Let $\g>1/2$ be such that
\begin{align}
(\s_1^2 b_1 t)^\g \ll b_1 t^{1-\a_1}.
\end{align} 
For all $j$ appearing in the product in \eqref{eq:below_first1}, we have
\begin{align}
\E\ekl{
	\tilde{\YY}_{j}(t) \,\big\vert\, \FF_{t^\b}
}
=
C \eee^{-\sqrt{2} y}
\left(\sqrt{2} \s_1 t^{\b} - \bar{x}_j(t^{\b}) -t^{\b\delta/2} \right)
\eee^{-\sqrt{2} \, \left(\sqrt{2} \s_1 t^{\b} - \bar{x}_j(t^{\b}) \right)}
(1+o(1)).
\label{eq:below_FMFP_Result}
\end{align}
\end{lemma}

\begin{lemma}\label{lem:below_SecondMomentsFirstPart}
Let $\g$ be as in Lemma~\ref{lem:below_FirstMomentsFirstPart}.
For all $j$ appearing in the product in \eqref{eq:below_first1}, we have
\begin{align}
\E\ekl{
	\tilde{\YY}^2_{j}(t) \,\big\vert\, \FF_{t^\b}
}
\leq  P(t) \,
\eee^{
	-\sqrt{2} \, 
	\left(\sqrt{2} \s_1 t^{\b} - \bar{x}_j(t^{\b}) \right)
}
\eee^{
	-\sqrt{2} t^{\beta\delta/2}
}
(1+o(1)).
\end{align}
\end{lemma}
By Lemmas~\ref{lem:below_FirstMomentsFirstPart} and \ref{lem:below_SecondMomentsFirstPart},
\begin{align}
\sfrac{\E\ekl{
	\tilde{\YY}^2_{j}(t) \,\big\vert\, \FF_{t^\b}
	}}{\E\ekl{
	\tilde{\YY}_{j}(t) \,\big\vert\, \FF_{t^\b}
	}} 
	\leq C^{-1} \eee^{\sqrt{2}y}P(t)
	\left(
	\sqrt{2} \s_1 t^{\b} - \bar{x}_j(t^{\b}) -t^{\b\delta/2} 
	\right)^{-1} \eee^{-\sqrt{2}t^{\b\delta/2}},
\end{align}
which, for $\bar{x}_j \in \LL_{t^\b}$, converges to $0$ as $t\uparrow\infty$.
We proceed as in \eqref{eq:below_FractionSecondFirstMomentMiddle} -- \eqref{eq:below_AfterAveragingMiddle}, to get that 
\begin{align}
&\E \, \Bigg[ 
\prod_{\substack{
	j \le n(t^{\b}) \\ 
	\bar{x}_j \in \LL_{t^\b}
	}}
	\E
	\ekl{
\eee^{-\tilde{\YY}_{j}(t) }\,\big\vert\, \FF_{t^\b}
}\!
\Bigg]
\label{eq:below_beforeDerivMart}\\
&=\:
\E \, \Bigg[
\exp\bbbkl{
{-}C \eee^{-\sqrt{2} y} 
\sum_{\substack{
		j \le n(t^{\b}) \\ 
		\bar{x}_j \in \LL_{t^\b}
}}
\left(
\sqrt{2} \s_1 t^{\b} - \bar{x}_j(t^{\b}) 	-t^{\b\delta/2} 
\right)
\eee^{
	-\sqrt{2} \, \left(\sqrt{2} \s_1 t^{\b} - 	\bar{x}_j(t^{\b}) \right)
}
}
\Bigg] (1+o(1)).
\notag
\end{align}
The sum in \eqref{eq:below_beforeDerivMart} converges to the limit of the derivative martingale.
\begin{lemma}\label{prop:belowDerivMartLimit}
With the notation from above,
\begin{align}
\sum_{\substack{
		j \le n(t^{\b}) \\ 
		\bar{x}_j \in \LL_{t^\b}
}}
\left(\sqrt{2} \s_1 t^{\b} - \bar{x}_j(t^{\b}) -t^{\b\delta/2}\right)
\eee^{-\sqrt{2} \, \left(\sqrt{2} \s_1 t^{\b} - \bar{x}_j(t^{\b}) \right)}
&\to Z,
\end{align}
in probability as $t\uparrow\infty$.
\end{lemma}
\begin{proof}
The proof is the same as for Lemma~\ref{prop:above_derivative_martingale} since
$
(\s_1-1) t^\b = o(1)
$.
\end{proof}

\noindent As almost surely, 
\begin{align}
-C \eee^{-\sqrt{2} y} 
\sum_{\substack{
	j \le n(t^{\b}) \\ 
	\bar{x}_j \in \LL_{t^\b}
	}}
	\left(\sqrt{2} \s_1 t^{\b} - \bar{x}_j(t^{\b}) -t^{\b\delta/2} \right)
	\eee^{-\sqrt{2} \, \left(\sqrt{2} \s_1 t^{\b} - \bar{x}_j(t^{\b}) \right)} <0,
\end{align}
Lemma~\ref{prop:belowDerivMartLimit} implies that 
\begin{align}
& \lim_{t\uparrow\infty} \P \left( \max_{j \le n(t)\colon \bar{x}_j \in \LL} \bar{x}_j(t) - m^{-}(t) \le y \right)\notag\\
&=\:
\lim_{t\uparrow\infty}
\E \, \Bigg[
\exp\bbbbkl{
{-}C \eee^{-\sqrt{2} y} 
\sum_{\substack{
		j \le n(t^{\b}) \\ 
		\bar{x}_j \in \LL_{t^\b}
}}
\left(
\sqrt{2} \s_1 t^{\b} - \bar{x}_j(t^{\b}) -t^{\b\delta/2} 
\right)
\eee^{-\sqrt{2} \, \left(\sqrt{2} \s_1 t^{\b} - \bar{x}_j(t^{\b}) \right)}
}
\Bigg]\notag\\
&=\:
\E \Big[
\eee^{
{-}CZ \eee^{-\sqrt{2} y} 
}\Big].
\end{align}
This completes the proof of Proposition~\ref{prop:BelowLSpeedLimitingLaw} up to the proofs of Lemmas~\ref{lem:below_FirstMomentsFirstPart} and
\ref{lem:below_SecondMomentsFirstPart}.
\null\hfill\qedsymbol

To complete the proof of Theorem~\ref{thm:lawofmax} in Case~B, we use the fact that we can approximate the speed functions in that setting from above and below by the speed functions described in the following lemma. We deduce from Proposition~\ref{prop:BelowLSpeedLimitingLaw} that variable speed BBMs with these speed functions have the same limiting law of the maximum.

\begin{lemma}\label{lem:belowPiecewiseLinearApprox}
Let $\ab,\ae \in \kl{0, 1/2}$. Let $\kl{A_t}_{t>0}$ be a family of speed functions which satisfy Assumption~\ref{as:below} for $\ab,\ae$. 
Then, there exist $\ell\in\N$ and families of speed functions $(\underline{A}_t)_{t>0}, (\overline{A}_t)_{t>0}$ which satisfy for each $t>0$ large enough:
\begin{thmlist}
\item \label{lem:belowPiecewiseLinearApprox1} $\underline{A}_t(s) \leq A_t(s) \leq \overline{A}_t(s) $ for all $s \in [0,1]$.
\item \label{lem:belowPiecewiseLinearApprox2} The functions $\underline{A}_t$ and $\overline{A}_t$ are piecewise linear and continuous
with slopes
\begin{align}
	\underline{A}_t'(s) &= 
	\sum_{k=1}^{\ell} \underline{\sigma}_k^2 \1_{ \left( \sum_{j=1}^{k-1} \underline{b}_j, \sum_{j=1}^{k} \underline{b}_j \right)}(s), \quad
	\text{ for } s \in (0,1),
	\notag\\
	\overline{A}_t'(s) &= 
	\sum_{k=1}^{\ell} \overline{\sigma}_k^2 \1_{ \left( \sum_{j=1}^{k-1} \overline{b}_j, \sum_{j=1}^{k} \overline{b}_j \right)}(s), \quad
	\text{ for } s \in (0,1),
\end{align}
where $\underline{\sigma}_k,  \underline{b}_k$ (and analogously   $\overline{\sigma}_k,  \overline{b}_k$) depend on $t$ and satisfy
\begin{align}
	\label{eq:below_sigma_unten}
	\sum_{k=1}^{\ell} \underline{\sigma}_k^2\, \underline{b}_k = 1,
	\qquad
	\sum_{k=1}^{\ell}  \underline{b}_k = 1
	\quad \text{ and } \quad
	\underline{b}_k > 0 \text{ for all } k = 1,\dots, \ell.
\end{align}
\item \label{lem:belowPiecewiseLinearApprox3}
$\underline{A}_t$ and  $\overline{A}_t$ satisfy Assumption~\ref{as:belowLSpeed2}--(iv) with $\a_1 = \ab, \a_\ell = \ae$.
\end{thmlist}
\end{lemma}
\begin{proof}
Let $t>0$. We start with the construction of the first piece of the piecewise linear speed functions $\underline{A}_t$ and $\overline{A}_t$.
$A_t$ satisfies Assumption~\ref{as:belowA} with corresponding lower and upper bounds $\underline{B}_t$ and $\overline{B}_t$ on $[0,\bbeg]$, so we get for all $s \in [0, \bbeg]$ that
\begin{align}\label{eq:below_PiecewiseLinearApprox1}
A_t(s) \leq \overline{B}_t(s) 
&\leq \overline{B}_t(0) + s\overline{B}'_t(0) + \sfrac{s^2}{2} \max_{z \in [0, \bbeg]} \overline{B}''_t(z) \notag\\
&\leq \kl{1 - t^{-\ab} + \sfrac{\bbeg}{2} \max_{z \in [0, \bbeg]} \overline{B}''_t(z)}s\equiv \overline{\s}_1^2 s, 
\end{align}
and 
\begin{align}\label{eq:below_PiecewiseLinearApprox2}
A_t(s) \geq \kl{1 - t^{-\ab} - \sfrac{\bbeg}{2} \min_{z \in [0, \bbeg]} \underline{B}''_t(z)}s \equiv \underline{\s}_1^2 s.
\end{align}
We set $\underline{b}_1 =\overline{b}_1 = \bbeg/2$ and $\underline{A}_t(s) = \underline{\s}_1^2 s, \overline{A}_t(s) = \overline{\s}_1^2 s$ for all $s\in[0, \bbeg/2]$.
$\underline{A}_t$ and $\overline{A}_t$ satisfy Assumption~\ref{as:belowLSpeed2} and  $\underline{A}_t(s) \leq A_t(s) \leq \overline{A}_t(s) $ for all $s \in [0,\bbeg/2]$.
Analogously, $A_t$ satisfies Assumption~\ref{as:belowB} with corresponding lower and upper bounds $\underline{C}_t$ and $\overline{C}_t$ on $[1-\bend, 1]$, so we set $\underline{b}_\ell =\overline{b}_\ell = \bend/2$ and, for all $s\in [1-\bend/2, 1]$,
\begin{align}
1 -\underline{A}_t(s) = \underline{\s}_\ell^2 (1-s), \notag\\
1-\overline{A}_t(s) = \overline{\s}_\ell^2 (1-s),
\end{align}
where
\begin{align}
\underline{\s}_\ell^2 &\equiv 1 + t^{-\ae} - \sfrac{\bend}{2} \min_{z \in [1-\bend, 1]} \underline{C}''_t(z), \notag\\
\overline{\s}_\ell^2 &\equiv 1 + t^{-\ae} + \sfrac{\bend}{2} \max_{z \in [1-\bend, 1]} \overline{C}''_t(z).
\end{align}
Then, $\underline{A}_t$ and $\overline{A}_t$ satisfy Assumption~\ref{as:belowLSpeed3} and 
$\underline{A}_t(s) \leq A_t(s) \leq \overline{A}_t(s) $ for all $s \in [1-\bend/2, 1]$.
We turn to the construction of $\underline{A}_t$ and $\overline{A}_t$ on $[\bbeg/2, 1-\bend/2]$.
For $k=2, \dots, \ell-1$, we can choose $\underline{\sigma}_k,  \underline{b}_k$ freely, as long as \eqref{eq:below_sigma_unten} is satisfied and $\underline{A}_t(s) \leq  A_t(s)$ for $s \in [\bbeg/2, 1-\bend/2]$. The latter condition  implies that $\underline{A}_t$ satisfies Assumption~\ref{as:belowLSpeed4}.
Since $A_t$ satisfies Assumption~\ref{as:belowC}, there exists $\e>0$ such that
\begin{align}
\min_{s\in [\bbeg,1-\bend]} \bkl{s - A_t(s)} \geq t^{-1/2 + \e},
\end{align}
for $t$ large enough. For $s\in [\bbeg,1-\bend]$, we set
\begin{align}
\overline{A}_t(s) = s - t^{-1/2 + \e}.
\end{align}
This ensures that on $[\bbeg,1-\bend]$, $\overline{A}_t$ is a piecewise linear upper bound of $A_t$, which satisfies Assumption~\ref{as:belowLSpeed4}.
Then, we set $\overline{A}_t$ to be continuous in $\bbeg/2$ as well as $\bbeg$ with
$\overline{A}'_t(s) = \overline{\s}_2^2$  for  $s \in (\bbeg/2, \bbeg)$, where
\begin{align}
\overline{\s}_2^2 \equiv \sfrac{2}{\bbeg}\kl{\overline{A}_t(\bbeg) - \overline{A}_t(\bbeg/2)} = \sfrac{2}{\bbeg} \kl{(1-\overline{\s}_1^2/2) \bbeg - t^{-1/2 + \e}}.
\end{align}
By Assumption~\ref{as:belowA}.(i), $\overline{\s}_2^2 \in (1,\infty)$  for $\e$ close to $0$ and $t$ large enough. Analogously, we set the slope $\overline{\s}_4^2$ on $(1-\bend, 1-\bend/2)$ as
\begin{align}
\overline{\s}_4^2 \equiv \sfrac{2}{\bend}\kl{\overline{A}_t(1-\bend/2) - \overline{A}_t(1-\bend)},
\end{align}
which lies in $(0,1)$ for $\e$ close to $0$ and $t$ large.
Thus,  $\overline{A}_t$ is a piecewise linear and continuous upper bound of $A_t$, which satisfies Assumption~\ref{as:belowLSpeed4} on $[\bbeg/2,1-\bend/2]$.
\end{proof}

\begin{remark}
In the proof of Lemma~\ref{lem:belowPiecewiseLinearApprox}, it becomes clear that it is always possible to choose $\ell=5$.
\end{remark}

\begin{proof}[Proof of Theorem~\ref{thm:lawofmax}]
Let  $(\underline{A}_t)_{t>0}$ and $(\overline{A}_t)_{t>0}$ be the speed functions from Lemma~\ref{lem:belowPiecewiseLinearApprox} corresponding to $(A_t)_{t>0}$, $\ab$ and $\ae$.
Let $(\underline{X}_t)_{t>0}, (\overline{X}_t)_{t>0}$ be variable speed BBMs with speed functions $(\underline{A}_t)_{t>0}, (\overline{A}_t)_{t>0}$.
Since  $(\underline{A}_t)_{t>0}$ and $(\overline{A}_t)_{t>0}$ satisfy Assumption~\ref{as:belowLSpeed}, Proposition~\ref{prop:BelowLSpeedLimitingLaw} implies that
\begin{align}\label{eq:below_BeforeGaussComp}
\lim_{t \uparrow \infty} \P \left( \max_{j \le n(t)} \underline{x}_j(t) - m^{-}(t) \le y \right)  
= \lim_{t \uparrow \infty} \P \left( \max_{j \le n(t)} \overline{x}_j(t) - m^{-}(t) \le y \right)  
= \E \left[ \eee^{- C Z \eee^{-\sqrt{2}y}} \right].
\end{align}
Since $\underline{A}_t \leq A_t \leq \overline{A}_t$ for all $t>0$, we get
with Gaussian comparison (see Lemma~\ref{lem:slepian}) and \eqref{eq:below_BeforeGaussComp} that
\begin{align}
\lim_{t \uparrow \infty} \P \left( \max_{j \le n(t)} \tilde{x}_j(t) - m^{-}(t) \le y \right)  
= \E \left[ \eee^{- C Z \eee^{-\sqrt{2}y}} \right].
\end{align}
This completes the proof of Theorem~\ref{thm:lawofmax}  in Case~B.
\end{proof}

\subsection[Proof of Theorem 1.4]{Proof of Theorem~\ref{thm:extremalprocess}}\label{sec:proofExtremalProcess}

Let $\tilde X$ be a VSBBM with piecewise linear speed functions satisfying either Assumption~\ref{as:above} or~\ref{as:belowLSpeed}, respectively. 
We compute, for $y \in \R$ and for $\phi \in \CC^\infty(\R)$ which are nondecreasing with support bounded from the left and for which there exists $a\in \R$ such that $\phi(x)$ is constant for $x>a$,
\begin{align}\label{eq:ExtrProc_0}
&\Psi_t ( \phi \, ( \cdot - y )) \\
&= \E  \Bigg[ \eee^{- 
\sum_{j=1}^{n(t)} \phi \left( \tilde{x}_j(t) - m^{\pm}(t) - y\right) 
} \Bigg] \nonumber  \\
&= \, \E  \Bigg[ \exp \Bigg( {-} 
\sum_{i_1 \le n(b_1t)} 
\sum_{i_2 \le n^{i_1}(b_2t)}
\dots
\sum_{i_{\ell} \le n^{\mi_{\ell-1}}(b_{\ell}(t))}
\phi \, \Bigg(
\sum_{j=1}^{\ell}  \s_j \, x_{i_j}^{\mi_{j-1}} (b_jt) - m^{\pm}(t) - y
\Bigg)
\Bigg) \Bigg] \nonumber \\
&= \, \E  \Big[ \textstyle\prod\limits_{i_1 \le n(b_1t)} 
\!\E  \Big[ \!\prod\limits_{i_2 \le n^{i_1}(b_2t)} \hspace{-0.7em} \dots
\E \Big[  \!\prod\limits_{i_{\ell} \le n^{\mi_{\ell-1}}(b_{\ell}(t))}
\!\!\eee^{- \phi  \left(
\sum\limits_{j=1}^{\ell}  \s_j  x_{i_j}^{\mi_{j-1}} (b_jt) - m^{\pm}(t) -y
\right) }
\!\Big\vert \FF_{a_{\ell-1}t} \Big]
\!\dots
\Big\vert \FF_{a_1t} \Big]
\Big], \nonumber 
\end{align}
where $x_{i_j}^{\mi_{j-1}}, 1 \le j \le \ell$, denote standard BBMs. 
We want to show that the innermost conditional expectation is a solution to the \mbox{F-KPP} equation and use the asymptotics of these solutions.
Recalling that the velocities $\s_j, 1 \le j \le \ell$, depend on $t$, we set
\begin{equation}
f^t(z) = \eee^{- \phi \, ( - \s_{\ell}(t) z ) }, \qquad 
v^t ( s, z )
= \E \, \bigg[  \prod_{i_{\ell} \le n^{\mi_{\ell-1}}(s)}
f^t\left( z - x_{i_{\ell}}(s) \right) 
\bigg].
\end{equation}
For fixed $t$, the function $1 - v^t$ is a solution to the \mbox{F-KPP} equation with initial conditions $1 - v^t(0,x) = 1 - f^t(x)$.
Since the initial conditions depend on the time horizon $t$, we need the following generalisation of Proposition~\ref{prop:FKPP_tail_estimate}.
\begin{propositionb}[\citen{1to6}, Proposition 5.2] \label{prop:F-KPP_tail_estimate_general}
Let $u^t$ be a family of solutions to the F-KPP equation with initial data satisfying
\begin{equation}
u^t(0,x) \to u(0,x),
\end{equation}
pointwise and monotone, for $x \in \R$ as $t \uparrow \infty$, where $u(0,x)$ satisfies the conditions in Proposition~\ref{prop:FKPP_tail_estimate}(i)--(iv). 
Then, for any function $z \colon \R_+ \to \R_+$ such that $\lim_{t \uparrow \infty} z(t)/t = 0$,
\begin{equation}
\lim_{t \uparrow \infty} \eee^{\sqrt{2} z(t)} \eee^{ (z(t))^2/(2t) } (z(t))^{-1}
u^t \left( t, z(t) + \sqrt{2}t - \sfrac{3}{2\sqrt{2}} \log(t) \right)
= C,
\end{equation}
where $C$ is the constant from Proposition~\ref{prop:FKPP_tail_estimate} and $u$ is the solution of the F-KPP equation with initial condition $u(0,\cdot)$.
\end{propositionb}
For $t \uparrow \infty$, $1 - f^t(x) \to 1 - \eee^{ - \phi (-x) }$  pointwise and monotone for $x \ge 0$.
Therefore, we can apply Proposition~\ref{prop:F-KPP_tail_estimate_general} to $u^t \equiv 1 - v^t$. 
Hence,
\begin{align}
\E \,& \Bigg[  \hspace{-0.1em}
\prod_{i_{\ell} \le n^{\mi_{\ell-1}}(b_{\ell}(t))}
\hspace{-0.6em}
\eee^{- \phi \, \left(
\sum_{j=1}^{\ell}  \s_j \, x_{i_j}^{\mi_{j-1}} (b_jt) - m^{\pm}(t) - y
\right) }
\bigg\vert \FF_{a_{\ell-1}t} \Bigg]\notag\\
= \, & v^t \, \Bigg( b_{\ell}t, \sfrac{1}{ \s_{\ell}(t) } \Bigg( m^{\pm}(t) + y - \sum_{j=1}^{\ell} \tilde{x}_{i_j}^{\mi_{j-1}} (b_jt) \Bigg) \Bigg)\notag \\
= \, & \left( 1 - C(\phi) \Delta^{\mi_{\ell-1}}(t)
\,\eee^{- \sqrt{2} \, \Delta^{\mi_{\ell-1}}(t)
- \sfrac{ \Delta^{\mi_{\ell-1}}(t)^2 }{b_{\ell} t}} \right)
( 1 + o(1)),
\end{align}
where $C(\phi)$ is 
the constant from Proposition~\ref{prop:FKPP_tail_estimate} for the initial condition $u(0,x)=1-\eee^{-\phi(-x)}$ and $ \Delta^{\mi_{\ell-1}}$ is defined for $m^+(t)$ in \eqref{eq:above_def_Delta} and for $m^-(t)$ in  \eqref{eq:below_delta_def_Ende}. 

From now on, following the same computations as in the proof of Theorem~\ref{thm:lawofmax}, we get that for variable speed BBMs with piecewise linear speed functions
\begin{align}\label{eq:ExtrProc_1}
\lim_{t \uparrow \infty} \Psi_t (\phi \,  ( \cdot - y )) =
\E \, \bigg[
\eee^{
- C(\phi) Z  \eee^{ - \sqrt{2} y} 
}
\bigg]. 
\end{align}

The Laplace functional in \eqref{eq:ExtrProc} corresponds to the point processes in the right-hand side of \eqref{eq:ExtrProc}, see \cite{ABK_E}.
Thus, the proof of 
Theorem~\ref{thm:extremalprocess} in Case~A is completed.

It remains to show  convergence of the extremal process for 
variable speed BBM $(\tilde{X}_t)_{t>0}$ with 
general (not necessarily piecewise linear) speed functions $(A_t)_{t>0}$, which satisfy Assumption~\ref{as:below}.
For each $t>0$, given the underlying Galton-Watson tree of $\tilde{X}_t$, let $\underline{X}_t$ and $\overline{X}_t$ be independent Gaussian processes on this tree with mean $0$ and speed functions $\underline{A}_t, \overline{A}_t$ from Lemma~\ref{lem:belowPiecewiseLinearApprox}.
By \eqref{eq:ExtrProc_1},
\begin{align}
\lim_{t \to \infty} \underline{\Psi}_t (\phi \, ( \cdot - y)) 
= \lim_{t \to \infty} \overline{\Psi}_t (\phi \, ( \cdot - y)) 
= 
\E\ekl{
\eee^{-C(\phi) Z \eee^{-\sqrt{2}y}}
},
\end{align}
where  
\begin{align}
\underline{\Psi}_t ( \phi \, ( \cdot - y)) 
&\equiv
\E \, \Bigg[ 
\exp \Bigg(
{-} \sum_{j=1}^{n(t)} \phi \left( \underline{x}_j(t) - m^{-}(t) -y \right) 
\Bigg)
\Bigg],\notag\\
\overline{\Psi}_t (\phi \, ( \cdot - y)) 
&\equiv
\E \, \Bigg[ 
\exp \Bigg(
{-} \sum_{j=1}^{n(t)} \phi \left( \overline{x}_j(t) - m^{-}(t) -y \right) 
\Bigg)
\Bigg].
\end{align}
Thus, it suffices to prove, for $t>0$, that
\begin{align}
\label{eq:ExtrProc_GaussCompare}
\E\ekl{
F\bbkl{
	(\underline{x}_j(t))_{j\leq n(t)}
}
\Big\vert
\FF_t^{\text{tree}}
}
\leq
\E\ekl{
F\bbkl{
	(\tilde{x}_j(t))_{j\leq n(t)}
}
\Big\vert
\FF_t^{\text{tree}}
}
\leq
\E\ekl{
F\bbkl{
	(\overline{x}_j(t))_{j\leq n(t)}
}
\Big\vert
\FF_t^{\text{tree}}
},
\end{align}
where $\FF_t^{\text{tree}}$ is the $\sigma$-algebra of the Galton-Watson tree of $\tilde{X}_t$, $\underline{X}_t$ and $\overline{X}_t$ and
\begin{align}
F\colon \R^{n(t)} \to \R,\quad 
z \mapsto \exp\bbbbkl{
{-} \sum_{j=1}^{n(t)} \phi \left( z_j - m^{-}(t) -y \right)
}. 
\end{align}
Note that inside the conditional expectations in \eqref{eq:ExtrProc_GaussCompare}, we view $n(t)$ as a constant. 
Since $\phi$ is nondecreasing, we get for all $x \in \R^{n(t)}$ and \mbox{$1 \leq i_1, i_2 \leq n(t)$}, $i_1\neq i_2$, that 
\begin{align}\label{eq:ExtrProc_Partial}
\sfrac{\partial^2 F(x)}{\partial x_{i_1} \partial x_{i_2}} 
= \phi'(x_{i_1} -  m^{-}(t) -y)\,\phi'(x_{i_2} -  m^{-}(t) -y) \, F(x) \geq 0.
\end{align}
By Kahane's theorem, see for example \cite[Theorem 3.5]{bbm-book}, \eqref{eq:ExtrProc_GaussCompare} follows from \eqref{eq:ExtrProc_Partial}.
This completes the proof of  Theorem~\ref{thm:extremalprocess} in Case~B. \hfill\qedsymbol

\appendix

\section{Appendix} \label{sec:appendix}
\label{sec:above_proofMoments}\label{sec:below_first_second_moments}
We prove the moment estimates from Subsection~\ref{sec:proof_caseA} and Subsection~\ref{sec:below}.
\begin{proof}[Proof of Lemma \ref{lem:above_middle_first_moment}]
By the many-to-one lemma,
\begin{align} \label{eq:above_middle_first_many}
\E \,& \Big[ \YY_{i_{\ell-2}}(t) \, \big\vert \, \FF_{a_{\ell-2}t} \Big] \nonumber \\
= \, &\eee^{b_{\ell-1}t} \, \E \, \bigg[ \!
\exp \left( -\sqrt{2}
\sfrac{1}{\s_{\ell}} \left( \sqrt{2} \s_{\ell-1} b_{\ell-1} t - \tilde{x}_{i_{\ell-1}}(b_{\ell-1}t) \right)
+ \sfrac{3}{2} \Big( \! \log(b_{\ell-1}t) + 2 \log (\s_{\ell-1} - \s_{\ell} ) \Big) \right) \nonumber \\
& \times
\1_{ \Big\{
	\tilde{x}_{i_{\ell-1}} (b_{\ell-1}t) - \sqrt{2} \s_{\ell-1} b_{\ell-1} t - \Lambda^{\mi_{\ell-2}}(t) \in \sfrac{ [-B, -D] }{ \s_{\ell-1} - \s_{\ell} }
	\Big\}}	\nonumber \\
& \times
\1_{ \Big\{ 
	\tilde{x}_{i_{\ell-1}} (s) \le \sqrt{2} \s_{\ell-1}s + \Lambda^{\mi_{\ell-2}}(t) - t^{\b\delta/2} \, \forall s \in [0, b_{\ell-1} t] \Big\}}
\, \Big\vert \, \FF_{\a_{\ell-2}t} \bigg] \nonumber \\
= \, &
( b_{\ell-1}t )^{3/2} ( \s_{\ell-1} - \s_{\ell} )^3
\eee^{b_{\ell-1}t}
\! \int_I \sfrac{\d \o}{\sqrt{2 \pi \s_{\ell-1}^2 b_{\ell-1} t}}
\exp \left({-} \sfrac{\o^2}{2 \s_{\ell-1}^2 b_{\ell-1} t} 
	- \sqrt{2} \sfrac{1}{\s_{\ell}} \!\left( \sqrt{2} \s_{\ell-1} b_{\ell-1} t - \o \right)\! \right) \nonumber\\
	& \times
	\P \, \bigg( \zet_{
		- \Lambda^{\mi_{\ell-2}}(t),
		\o - \Lambda^{\mi_{\ell-2}}(t) }^{b_{\ell-1}t}
	(s) \le \sqrt{2} \s_{\ell-1} s - t^{\b \delta/2} \, \forall s \in [0, b_{\ell-1} t] 
	\, \Big\vert \, \Lambda^{\mi_{\ell-2}}(t) \bigg),
\end{align} 
where
\begin{equation} \label{eq:above_middle_first_I}
	I = \sqrt{2} \s_{\ell-1} b_{\ell-1} t + \Lambda^{\mi_{\ell-2}}(t) 
	- \sfrac{[ D,B] }{\s_{\ell-1} - \s_{\ell}}.
\end{equation}
With $\o = - y (\s_{\ell-1} - \s_{\ell})^{-1} + \sqrt{2} \s_{\ell-1} b_{\ell-1} t  + \Lambda^{\mi_{\ell-2}}(t)$,
the right-hand side of \eqref{eq:above_middle_first_many} equals
\begin{align} \label{eq:above_middle_first_shift}
	& ( b_{\ell-1}t )^{3/2} ( \s_{\ell-1} - \s_{\ell} )^2
	\eee^{b_{\ell-1}t} 
	\int_D^B
	\sfrac{\d y}{\sqrt{2 \pi \s_{\ell-1}^2 b_{\ell-1} t}}
	\exp \left( {-} \sqrt{2} \sfrac{1}{\s_{\ell}}\! \left( \sfrac{y}{\s_{\ell-1} \!-\! \s_{\ell}} - \Lambda^{\mi_{\ell-2}}\!(t) \!\right)\!\right) \nonumber\\
	&\!\!\times\! 
	\exp \left( - \sfrac{ \big( \sqrt{2} \s_{\ell-1} b_{\ell-1}t + \Lambda^{\mi_{\ell-2}}(t) - y (\s_{\ell-1} - \s_{\ell})^{-1} \big)^2 }{2 \s_{\ell-1}^2 b_{\ell-1} t} \right) \nonumber\\
	&\!\!\times
	\!\,\P\!\,  \bigg(\! \zet_{
		- \frac{1}{\s_{\ell-1}} \Lambda^{\mi_{\ell-2}}(t),
		- \frac{y}{ \s_{\ell-1} ( \s_{\ell-1} - \s_{\ell} )}}^{b_{\ell-1}t}
	\!(s) \le \!{-} \sfrac{t^{\b \delta/2}}{\s_{\ell-1}} \, \forall_{s \in [0, b_{\ell-1} t]} 
	\, \Big\vert \, \Lambda^{\mi_{\ell-2}}(t) \bigg).  
\end{align}
Since $\tilde{x}_{i_1} \in \LL_{b_1t, 1, t^{\b}, B, D}, \dots, \tilde{x}_{i_{\ell-2}}^{\mi_{\ell-3}} \in \LL_{ b_{\ell-2}t, \ell-2, B, D}^{\mi_{\ell-3}}$,  $\Lambda^{\mi_{\ell-2}}(t)$ is of order $( \s_{\ell-2} - \s_{\ell-1} )^{-1}$. Thus, 
\begin{equation} \label{eq:above_middle_first_exponential}
	\sfrac{ \big(\Lambda^{\mi_{\ell-2}}(t) - y \, (\s_{\ell-1} - \s_{\ell})^{-1} \big)^2 }{2 \s_{\ell-1}^2 b_{\ell-1} t}
	= \OO \left( \left( \min\{ (\s_{\ell-2} - \s_{\ell-1})^2,
	(\s_{\ell-1} - \s_{\ell})^2 \} b_{\ell-1}t \right)^{-1} \right).
\end{equation}
By Assumption~\ref{as:above3}, the right-hand side of \eqref{eq:above_middle_first_exponential} tends to zero as $t \uparrow \infty$.
By Lemma~\ref{lem:BB_line},
\begin{align} \label{eq:above_middle_first_BB}
	&\P \, \bigg( \zet_{
		- \frac{1}{\s_{\ell-1}} \Lambda^{\mi_{\ell-2}}(t),
		- \frac{y}{ \s_{\ell-1} ( \s_{\ell-1} - \s_{\ell} )}}^{b_{\ell-1}t}
	(s) \le - \sfrac{t^{\b \delta/2}}{\s_{\ell-1}} \, \forall s \in [0, b_{\ell-1} t] 
	\, \Big\vert \, \Lambda^{\mi_{\ell-2}}(t)  \bigg)\nonumber \\
	= \, & \sfrac{2}{ \s_{\ell-1}^2 b_{\ell-1}t} 
	\left( \Lambda^{\mi_{\ell-2}}(t) - t^{\b \delta/2} \right)
	\left( \sfrac{y}{ \s_{\ell-1} - \s_{\ell} }- t^{\b \delta/2} \right)
	(1+o(1)),
\end{align}
since, by  Assumption~\ref{as:above3}, the second line of \eqref{eq:above_middle_first_BB} converges to zero, as $t \uparrow \infty$.
Since $t^{\b} \ll ( \s_k - \s_{k+1} )^{-1}$ and $\delta \in (0,1/2)$, we can drop the terms $ t^{\b \delta/2}$ in \eqref{eq:above_middle_first_BB} and make only an error of order $o(1)$.
We conclude that \eqref{eq:above_middle_first_shift} is equal to
\begin{equation}
	\sqrt{ \sfrac{2}{\pi}}
	\, \Lambda^{\mi_{\ell-2}}(t)
	\eee^{\sqrt{2} \Lambda^{\mi_{\ell-2}}(t) \frac{\s_{\ell-1} - \s_{\ell}}{\s_{\ell-1} \s_{\ell}}}
	\sfrac{\s_{\ell-1} - \s_{\ell}}{\s_{\ell-1}^3}
	\int_D^B \d y \, \eee^{-\sqrt{2} \frac{y}{\s_{\ell-1} \s_{\ell}} } y \, (1+o(1)).
\end{equation}
As first $t \uparrow \infty$ and then $D \downarrow 0, B \uparrow \infty$, the integral in the last line converges to $1 / 2$. \linebreak
The conditional expectation of $\ZZ_{i_{\ell-2}}(t)$ is computed analogously. 
\begin{align}
	\E \,& \Big[ \ZZ_{i_{\ell-2}}(t) \, \big\vert \, \FF_{a_{\ell-2}t} \Big] \nonumber \\
	= \, &
	\sqrt{\sfrac{2}{\pi}} 
	\,  \Lambda^{\mi_{\ell-2}}(t)
	\eee^{\sqrt{2} \Lambda^{\mi_{\ell-2}}(t) \frac{\s_{\ell-1} - \s_{\ell}}{\s_{\ell-1} \s_{\ell}}}
	\sfrac{ \s_{\ell-1} - \s_{\ell}}{ \s_{\ell-1}^3 }
	\int_D^B
	\d y \, 
	\eee^{- \sqrt{2} \frac{y}{\s_{\ell-1} \s_{\ell}}} \nonumber \\
	& \quad \times
	\left( \sfrac{1}{\s_{\ell}}
	\left( \sfrac{y}{ \s_{\ell-1} \s_{\ell} } - \Lambda^{\mi_{\ell-2}}(t) \right)
	- \sfrac{3}{2 \sqrt{2}} \left( \ln (b_{\ell-1}t) + 2 \log (\s_{\ell-1} - \s_{\ell}) \right) 
	\right) \, (1+o(1)) \nonumber \\
	= \, & \sqrt{ \sfrac{2}{\pi}} 
	\, \Lambda^{\mi_{\ell-2}}(t)
	\eee^{\sqrt{2} \Lambda^{\mi_{\ell-2}}(t) \frac{\s_{\ell-1} - \s_{\ell}}{\s_{\ell-1} \s_{\ell}}}
	\sfrac{1}{ \s_{\ell-1}^3 \s_{\ell} }
	\nonumber \\
	&\quad \times \, \left[
	\int_D^B \d y \, \eee^{-\sqrt{2} \frac{y}{\s_{\ell-1} \s_{\ell}}} y^2 
	+ \Lambda^{\mi_{\ell-2}}(t) ( \s_{\ell-1} - \s_{\ell} )
	\int_D^B \d y \, \eee^{-\sqrt{2}  \frac{y}{\s_{\ell-1} \s_{\ell}} } y  
	\right] \, (1+o(1)).
\end{align} 
As first $t \uparrow \infty$ and then $D \downarrow 0, B \uparrow \infty$, the integrals in the last line converge to $1/ \sqrt{2}$ and $1/2$, respectively.
\end{proof}

\begin{proof}[Proof of Lemma \ref{prop:above_middle_second_moment}]
We split the expectation value in two summands: the first one, $(T1)$, contains the squared terms and the second one, $(T2)$, contains the mixed terms.
The first summand can be controlled by dropping the barrier condition and applying the many-to-one lemma:
\begin{align} \label{eq:above_middle_T1}
	(T1)
	= \, &\E \, \Bigg[ 
	\sum_{\substack{i_{\ell-1} \le n^{\mi_{\ell-2}}(b_{\ell-1}t), \\ \tilde{x}_{i_{\ell-1}}^{\mi_{\ell-2}} \in \LL_{b_{\ell-1}t, \ell-1, B, D}^{\mi_{\ell-2}}}}
	\exp \Bigg(
	{\textstyle - 2 \sqrt{2} \frac{1}{\s_{\ell}} \left( \sqrt{2} \s_{\ell-1} b_{\ell-1} t - \tilde{x}_{i_{\ell-1}}^{\mi_{\ell-2}}(b_{\ell-1}t) \right) }
	\Bigg) \, 
	\bigg\vert \, \FF_{a_{\ell-2}t} \Bigg] \nonumber \\
	\le \, & \eee^{b_{\ell-1}t}
	\int_I
	\sfrac{\d \o}{ \sqrt{2 \pi \s_{\ell-1}^2 b_{\ell-1} t} }
	\exp \left( - \sfrac{\o^2}{2 \s_{\ell-1}^2 b_{\ell-1} t}
	- \sqrt{2} \sfrac{1}{\s_{\ell}} \left( \sqrt{2} \s_{\ell-1} b_{\ell-1} t - \o \right) \right),
\end{align}
where $I$ is defined as in \eqref{eq:above_middle_first_I}.
With the same arguments as in \eqref{eq:above_middle_first_shift} -- \eqref{eq:above_middle_first_exponential}, we see that
\begin{align} \label{eq:above_middle_second_T1_asymp}
	(T1)
	\le 
	\frac{1}{\sqrt{2 \pi}} 
	\frac{
		\eee^{\sqrt{2} \Lambda^{\mi_{\ell-2}}(t) \frac{\s_{\ell-1} - \s_{\ell}}{\s_{\ell-1} \s_{\ell}}}}{
		(\s_{\ell-1} - \s_{\ell} )(b_{\ell-1})^{1/2} \s_{\ell-1}}
	\int_D^B
	\!\d y \, 
	\eee^{- \sqrt{2} \frac{y}{ \s_{\ell-1} \s_{\ell} } }
	\eee^{- \sqrt{2} \frac{1}{\s_{\ell}} \big(  \frac{y}{ \s_{\ell-1} - \s_{\ell} } - \Lambda^{\mi_{\ell-2}}(t) \big)}
	\!(1+o(1)).
\end{align}
Bounding the last exponential in \eqref{eq:above_middle_second_T1_asymp} by its maximal value, we obtain
\begin{equation} \label{eq:above_middle_second_bound_T1}
	(T1) \le 
	\frac{1}{\sqrt{2 \pi}} 
	\frac{
		\eee^{- \sqrt{2} \frac{1}{\s_{\ell}} \big(  \frac{D}{ \s_{\ell-1} - \s_{\ell} } - \Lambda^{\mi_{\ell-2}}(t) \big)} }{
		(\s_{\ell-1} - \s_{\ell} )(b_{\ell-1})^{1/2} \s_{\ell-1}}
	\eee^{\sqrt{2} \Lambda^{\mi_{\ell-2}}(t) \frac{\s_{\ell-1} - \s_{\ell}}{\s_{\ell-1} \s_{\ell}}}
	\int_D^B
	\d y \, 
	\eee^{- \sqrt{2} \frac{y}{ \s_{\ell-1} \s_{\ell} } }
	(1+o(1)).
\end{equation}
The integral converges to $1/\sqrt{2}$ as first $t \uparrow \infty$ and then $D \downarrow 0, B \uparrow \infty$.
We compute the second summand with help of the many-to-two lemma \cite[Lemma~10]{B_M}.
\begin{align} \label{eq:above_middle_second_many}
	&(T2)\notag\\
	= \, &\E  \Bigg[ 
	\hspace{-0.7em}
	\sum_{\substack{i_{\ell-1}, j_{\ell-1} \le n^{\mi_{\ell-2}}(b_{\ell-1}t); \\  i_{\ell-1} \neq j_{\ell-1};  \\ \tilde{x}_{i_{\ell-1}}^{\mi_{\ell-2}},  \tilde{x}_{j_{\ell-1}}^{\mi_{\ell-2}} \in \LL_{b_{\ell-1}t, \ell-1, B, D}^{\mi_{\ell-2}}}}
	\hspace{-2em}
	\exp \left(
	{\textstyle {-} \sqrt{2} \frac{1}{\s_{\ell}} \!\left( 2 \sqrt{2} \s_{\ell-1} b_{\ell-1} t \!-\! \tilde{x}_{i_{\ell-1}}^{\mi_{\ell-2}}(b_{\ell-1}t) \!-\!  \tilde{x}_{j_{\ell-1}}^{\mi_{\ell-2}}(b_{\ell-1}t) \right)}
	\right) \! 
	\bigg\vert  \FF_{a_{\ell-2}t} \Bigg] \nonumber \\
	= \, & K 
	\int_0^{b_{\ell-1}t} \d s \, \eee^{2 b_{\ell-1}t - s}
	\int_{- \infty}^{\sqrt{2} \s_{\ell-1}s + \Lambda^{\mi_{\ell-2}}(t)  - t^{\b \delta/2} }
	\sfrac{ \d \o}{\sqrt{2 \pi \s_{\ell-1}^2 s} }
	\eee^{ - \sfrac{\o^2}{2 \s_{\ell-1}^2 s} } \nonumber \\
	&  \times 
	\Bigg( \int_{I - \o }
	\sfrac{\d z}{\sqrt{2 \pi \s_{\ell-1}^2 ( b_{\ell-1} t - s )}}
	\eee^{- \sfrac{z^2}{2 \s_{\ell-1}^2 \left( b_{\ell-1}t - s \right) } }
	\eee^{ - \sqrt{2} \frac{1}{\s_{\ell}} \left( \sqrt{2} \s_{\ell-1} b_{\ell-1}t - ( \o + z ) \right) }
	\Bigg)^2,
\end{align}
where $K > 0$ is some constant and $I$ is defined as in \eqref{eq:above_middle_first_I}.
Shifting the $\o$-integral, we get that \eqref{eq:above_middle_second_many} is equal to
\begin{align} \label{eq:above_middle_second_shift}
	&K 
	\int_0^{b_{\ell-1}t} \d s \, \eee^{2 b_{\ell-1}t - s}
	\int_{- \infty}^{- t^{\b \delta/2} }
	\sfrac{ \d \tilde{\o}}{\sqrt{2 \pi \s_{\ell-1}^2 s} }
	\eee^{- \sfrac{ \left( \tilde{\o} + \sqrt{2} \s_{\ell-1}s + \Lambda^{\mi_{\ell-2}}(t) \right)^2}{2 \s_{\ell-1}^2 s} } \\
	&  \quad \times 
	\Bigg( \int_{I_1}
	\sfrac{\d z}{\sqrt{2 \pi \s_{\ell-1}^2 ( b_{\ell-1} t - s )}}
	\eee^{ - \sfrac{z^2}{2 \s_{\ell-1}^2 \left( b_{\ell-1}t - s \right) } }
	\eee^{- \sqrt{2} \frac{1}{\s_{\ell}} \left( \Lambda^{\mi_{\ell-2}}(t) + \sqrt{2} \s_{\ell-1} \left( b_{\ell-1}t - s \right) - ( \tilde{\o} + z ) \right) }
	\Bigg)^2, \nonumber
\end{align}
where
\begin{equation}
	I_1 = \sqrt{2} \s_{\ell-1} \left( b_{\ell-1}t - s \right)  - \tilde{\o} - \sfrac{[ D,B] }{\s_{\ell-1} - \s_{\ell}}.
\end{equation}
With a shift of the $z$-integral, we see that the last line in \eqref{eq:above_middle_second_shift} equals
\begin{align} \label{eq:above_middle_second_squared}
	& \Bigg(
	\int_{ - \frac{B}{\s_{\ell-1} - \s_{\ell}}}^{- \frac{D}{\s_{\ell-1} - \s_{\ell}}}
	\sfrac{ \d \tilde{z}}{\sqrt{2 \pi \s_{\ell-1}^2 ( b_{\ell-1} t - s )}}
	\eee^{- \sfrac{ \left( \tilde{z} - \tilde{\o} + \sqrt{2} \s_{\ell-1} ( b_{\ell-1}t - s) \right)^2}{2 \s_{\ell-1}^2 ( b_{\ell-1}t - s)}}
	\eee^{ \sqrt{2} \frac{1}{\s_{\ell}}\left( \Lambda^{\mi_{\ell-2}}(t)+\tilde{z}\right) }
	\Bigg)^2 \nonumber \\
	&= \,  
	\eee^{ 2 \sqrt{2} \frac{1}{\s_{\ell}} \Lambda^{\mi_{\ell-2}}(t)- 2 (b_{\ell-1}t - s)+2 \sqrt{2} \frac{1}{\s_{\ell-1}} \tilde{\o}} \notag\\
	&\quad\times
	\Bigg(
	\int_{ - \frac{B}{\s_{\ell-1} - \s_{\ell}}}^{- \frac{D}{\s_{\ell-1} - \s_{\ell}}}
	\sfrac{ \d \tilde{z}}{\sqrt{2 \pi \s^2_{\ell-1} ( b_{\ell-1} t - s )}}
	\eee^{- \frac{ \left( \tilde{z} - \tilde{\o} \right)^2}{2 \s_{\ell-1}^2 ( b_{\ell-1}t - s)}+\sqrt{2} \tilde{z} \frac{ \s_{\ell-1} - \s_{\ell}}{\s_{\ell-1} \s_{\ell}} }
	\Bigg)^2.
\end{align}
The last exponential inside the $\tilde{z}$-integral can be uniformly bounded by $\exp \left( - \frac{\sqrt{2}D}{ \s_{\ell-1} \s_{\ell}}\right)$. 
The remaining $\tilde{z}$-integral is smaller than 1.
Inserting this bounds into \eqref{eq:above_middle_second_shift}, we get that \eqref{eq:above_middle_second_shift} is not larger than
\begin{equation} \label{eq:above_middle_second_int}
	K 
	\eee^{- 2 \sqrt{2} \frac{D}{\s_{\ell-1} \s_{\ell}} - \sqrt{2} \Lambda^{\mi_{\ell-2}}(t) \left( \frac{1}{\s_{\ell-1}} - \frac{2}{\s_{\ell}} \right)}
	\int_0^{b_{\ell-1}t} \hspace{-0.5em} \d s
	\int_{- \infty}^{- t^{\b \delta/2}}
	\sfrac{ \d \tilde{\o}}{\sqrt{2 \pi \s_{\ell-1}^2 s} }
	\eee^{ - \sfrac{ \left( \tilde{\o} + \Lambda^{\mi_{\ell-2}}(t) \right)^2}{2 \s_{\ell-1}^2 s} 
		+ \sqrt{2} \frac{1}{\s_{\ell-1}} \tilde{\o}}.
\end{equation}
We bound the last exponential inside the $\tilde{\o}$-integral uniformly by $\exp ( - \sqrt{2} t^{\beta \delta} / \s_{\ell-1} )$. 
The remaining $\tilde{\o}$-integral can be bounded by 1.
Thus, \eqref{eq:above_middle_second_int} is not larger than
\begin{equation} \label{eq:above_middle_second_bound_T2}
	K
	\eee^{- 2 \sqrt{2} D}
	b_{\ell-1}t
	\eee^{ - \sqrt{2} \frac{1}{\s_{\ell-1}} t^{\b \delta/2} }
	\eee^{ \sqrt{2} \Lambda^{\mi_{\ell-2}}(t) 
		\left( \frac{\s_{\ell-1} - \s_{\ell}}{\s_{\ell-1} \s_{\ell}} + \frac{1}{\s_{\ell}} \right) }.
\end{equation}
Finally, we add the upper bounds on $(T1)$, \eqref{eq:above_middle_second_bound_T1}, and on $(T2)$, \eqref{eq:above_middle_second_bound_T2}.
\begin{align}
	&\eee^{ \sqrt{2} \Lambda^{\mi_{\ell-2}}(t) 
		\left( \frac{\s_{\ell-1} - \s_{\ell}}{\s_{\ell-1} \s_{\ell}} + \frac{1}{\s_{\ell}} \right) }\notag\\
	&\quad\times\left( 
	\eee^{ - \sqrt{2} \frac{D}{ \s_{\ell-1} ( \s_{\ell-1} - \s_{\ell})}}
	\sfrac{1}{\sqrt{2 \pi} }
	\sfrac{1}{( \s_{\ell-1} - \s_{\ell}) b_{\ell-1}t \s_{\ell-1} }
	+K \eee^{ - \sqrt{2} \frac{1}{\s_{\ell-1}} t^{\b \delta/2} }
	\eee^{- 2 \sqrt{2} D}
	b_{\ell-1}t
	\right) \nonumber \\
	&\le \, \eee^{ - \sqrt{2} \Lambda^{\mi_{\ell-2}}(t) 
		\left( \frac{\s_{\ell-1} - \s_{\ell}}{\s_{\ell-1} \s_{\ell}} + \frac{1}{\s_{\ell}} \right) }
	\eee^{- \sqrt{2} \frac{ t^{\beta} }{\s_{\ell-1}} \left( D + t^{\delta/2} \right) }
	P(t),
\end{align}
where $P$ is a function satisfying $P(t) \le t^c$ for some $c>0$.
In the last inequality, we used that $( \s_{\ell-1} - \s_{\ell} )^{-1} \gg t^{\b}$.
\end{proof}

\begin{proof}[Proof of Lemma \ref{lem:above_first_moment_first}]
The proof is similar to that of Lemma~\ref{lem:above_middle_first_moment}. 
The main difference lies in the use of a simplified localisation result and a different starting point for the Brownian bridge on the interval $[t^{\b}, b_1t]$.
Since $\tilde{x}_{j^*}^j \in \tilde{\LL}_{t^*, 1 \tilde{x}_j(t^{\b})-\sqrt{2}\s_1 t^{\b},B,D}$, $|\sqrt{2} \s_1 t^{\b} - \tilde{x}_j(t)| \le 2 \sqrt{2} \s_1 t^{\b}$ and $t^{\b} \ll (\s_1 - \s_2)^{-1}$, we know that, for any $D^*>D, B^*<B$ and for $t$ large enough, $\sqrt{2} \s_1 t^* - \tilde{x}_{j^*}^j(t^*) \in (D^*, B^*)(\s_1 - \s_2)^{-1}$.
By the many-to-one lemma, 
\begin{align} \label{eq:above_first_first_many}
	\E \, \Big[ \widetilde{\YY}_j(t) \, \vert \, \FF_{t^{\b}} \Big]
	= \, &
	( b_1 t )^{3/2} ( \s_1 - \s_2 )^3
	\eee^{t^*}
	\int_I
	\sfrac{ \d \o }{ \sqrt{2 \pi \s_1^2 t^* } }
	\exp  \left( - \sfrac{ \o^2 }{ 2 \s_1^2 t^* } - \sqrt{2} ( \sqrt{2} \s_1 t^* - \o )\right) \nonumber \\
	& \times 
	\P \, \bigg( \zet_{ 
		- ( \sqrt{2} \s_1 t^{\b} - \tilde{x}_j (t^{\b}) ),
		\o
	}^{t^*} (s) 
	\le \sqrt{2} \s_1 s - t^{\b \delta/2} \, \forall  s \in [0,t^*] \bigg),
\end{align}
where
\begin{equation}
	I = \sqrt{2} \s_1 t^* + (\sqrt{2} \s_1 t^{\b} - x_j(t^{\b})) - \sfrac{(D^*,B^*)}{\s_1 - \s_2}.
\end{equation}
With the change of variables $\o = -y (\s_1 - \s_2)^{-1} + \sqrt{2} \s_1 t^* + (\sqrt{2} \s_1 t^{\b} - x_j(t^{\b}))$, we see that the right-hand side of \eqref{eq:above_first_first_many} equals
\begin{align}  \label{eq:above_first_first_cov}
	&\sfrac{(b_1t)^{3/2} (\s_1 - \s_2)^2}{\s_1 \sqrt{2 \pi t^*} }
	\int_{D^*}^{B^*} \d y \,
	\eee^{
		- \sfrac{ y^2}{2 \s_1^2 (\s_1 - \s_2)^2 t^*}
	}
	\eee^{- \sqrt{2} \sfrac{y}{\s_1 (\s_1 - \s_2)}} \nonumber \\
	&\times
	\P \, \bigg( \zet_{ 
		- \frac{1}{\s_1} ( \sqrt{2} \s_1 t^{\b} - \tilde{x}_j (t^{\b}) ),
		- \frac{ y }{ \s_1 (\s_1 - \s_2)}
	}^{t^*} (s) 
	\le - \sfrac{t^{\b \delta/2}}{\s_1}\, \forall  s \in [0,t^*] \bigg).
\end{align}
By Assumption~\ref{as:above3} and $t^* \sim b_1t$, it follows that $y^2 /(2 \s_1^2 (\s_1 - \s_2)^2 t^*)$ converges to zero, uniformly in $y$, as $t \uparrow \infty$.
By Lemma~\ref{lem:BB_line}, 
\begin{align} \label{eq:above_first_first_BB}
	&\P \, \bigg( \zet_{ 
		- \frac{1}{\s_1} ( \sqrt{2} \s_1 t^{\b} - \tilde{x}_j (t^{\b}) ),
		- \frac{ y }{ \s_1 (\s_1 - \s_2)}
	}^{t^*} (s) 
	\le - \sfrac{t^{\b \delta/2}}{\s_1}\, \forall  s \in [0,t^*] \bigg) \nonumber \\
	&= \, \sfrac{2}{ \s_1^2 t^* } 
	\Big( \sqrt{2} \s_1 t^{\b} - \tilde{x}_j ( t^{\b}) - t^{\b \delta/2} \Big)
	\left( \sfrac{y}{\s_1 (\s_1 - \s_2)} - t^{\b \delta/2} \right)  \, (1+o(1)),
\end{align}
provided the second line of \eqref{eq:above_first_first_BB} converges to zero as $t \uparrow \infty$.
This is the case since $\sqrt{2} \s_1 t^{\b} - \tilde{x}_j(t)) \in (t^{\b \delta}, 2 \sqrt{2} \s_1 t^{\b})$, $ t^{\b} \ll (\s_1 - \s_2)^{-1}$ and $y^2 /(2 \s_1^2 (\s_1 - \s_2)^2 t^*) \to 0$ as $t \uparrow \infty$.
Since $ t^{\b} \ll (\s_1 - \s_2)^{-1}$, $\delta \in (0,1/2)$ and $t^{\b \delta /2} \ll t^{\b \delta}$, we can drop the terms $t^{\b \delta/2}$ in the last line of \eqref{eq:above_first_first_BB}.
Using that $b_1t \sim t^*$, we see that \eqref{eq:above_first_first_cov} is equal to
\begin{equation}
	\sqrt { \sfrac{2}{\pi} }
	\left( \sqrt{2} \s_1 t^{\b} - \tilde{x}_j \left( t^{\b} \right) \right) 
	\sfrac{ \s_1 - \s_2 }{ \s_1^3 }
	\int_{D^*}^{B^*}
	\d z \, 
	\eee^{ - \sqrt{2} \frac{y}{ \s_1 \s_2}}
	y \, (1+o(1)).
\end{equation}
The integral in the last line converges to $1/2$ as first $t \uparrow \infty$ and then $D^* \downarrow 0, B^* \uparrow \infty$.
The conditional expectation on $\widetilde{\ZZ}_j(t)$ is computed in an analogous way.
\end{proof}

\begin{proof}[Proof of Lemma \ref{prop:above_second_moment_first}]
The proof follows by the same arguments as the proof of Lemma~\ref{prop:above_middle_second_moment}. 
The only difference is that we use here the simplified localisation condition $\sqrt{2} \s_1 t^* - \tilde{x}_{j^*}^{j}(t^*)$ $\in (D^*, B^*)(\s_1 - \s_2)^{-1}$, where $D^*>D, B^*<B$.
\end{proof}

\begin{proof}[Proof of Lemma~\ref{lem:below_FirstMomentsMiddlePart}]
Recall that
\begin{align*}
	\YY_{i_1}(t) =
	\sum_{\substack{
			i_2 \le n^{i_1}(b_2t), \dots, i_{\ell-1} \le n^{\mi_{\ell-2}}(b_{\ell-1}t), \\
			\sum_{k=1}^{\ell-1} \multixb{k} \in \TT_{b_1t, (1-b_\ell)t, 0, \g}
	}}
	C \Delta^{\mi_{\ell-1}}(t)
	\exp\kl{
		b_\ell t
		-\sfrac{
			\kl{\Delta^{\mi_{\ell-1}}(t) + \sqrt{2} b_\ell 	t}^2
		}{
			2b_\ell t
		}
	},\tag{\ref{eq:below_DefY}}
\end{align*}
where
\begin{align}
	\label{eq:below_midfm0}
	\Delta^{\mi_{\ell-1}}(t) 
	+ \sqrt{2} b_\ell t
	= 
	\sfrac{1}{\s_\ell}
	\bbbkl{
		\sqrt{2}t 
		-\sum_{k=2}^{\ell-1}
		\multixb{k}(b_k t) 
		- \bar{x}_{i_{1}}(b_{1}t)
		+ L(t)+y
	}.
\end{align}
Since $\sum_{k=2}^{\ell-1}
\multixb{k}(b_k t) $ is centred  with variance \mbox{$(1 - \s_1^2 b_1 -\s_\ell^2 b_\ell)t$},
\begin{align}
	&\E\ekl{
		\mathcal{Y}_{i_1}(t) \,\big\vert\, \FF_{b_1 t}
	}
	=\:
	C\eee^{(1 - b_1) t} 
	\int\limits
	_{I}
	\sfrac{
		\d \omega
	}{
		\sqrt{2 \pi (1 - \s_1^2 b_1 -\s_\ell^2 b_\ell)t}
	}
	\eee^{
		-\sfrac{
			\omega^2
		}{
			2 (1 - \s_1^2 b_1 -\s_\ell^2 b_\ell)t
		}
	}\notag\\
	&\qquad\times
	\sfrac{1}{\s_\ell}
	\kl{
		\sqrt{2}(1-\s_\ell b_\ell)t 
		-\o 
		- \bar{x}_{i_{1}}(b_{1}t)
		+ L(t)+y
	}
	\eee^{
		-\sfrac{
			\kl{
				\sqrt{2}t 
				-\o
				- \bar{x}_{i_{1}}(b_{1}t)
				+ L(t)+y
			}^2
		}{
			2\s_\ell^2 b_\ell t
		}
	}\notag\\
	&\qquad\times
	\P\kl{
		\kl{
			\zet
			_{\bar{x}_{i_{1}}(b_{1}t), \bar{x}_{i_{1}}(b_{1}t)+\o}
			^{(1-\s_1^2b_1 -\s_\ell^2 b_\ell) t}
			\scriptstyle
			\kl{A_t(s/t) t - \s_1^2b_1t}
		}_{s\in [b_1 t, 1-b_\ell t]} 
		\in \TT_{b_1t, (1-b_\ell)t, 0, \g}
	},\label{eq:below_midfm1}
\end{align}
where
\begin{align}\label{eq:appendixIntervalDef}
	I \equiv \kl{
		\sqrt{2}(1-\s_\ell^2b_\ell) t-\bar{x}_{i_{1}}\!(b_{1}t) - (\s_\ell^2 b_\ell t)^{\g},
		\sqrt{2}(1-\s_\ell^2b_\ell) t-\bar{x}_{i_{1}}\!(b_{1}t) + (\s_\ell^2 b_\ell t)^{\g}
	}.
\end{align}
For $\bar{x}_{i_1} \in \LL_{b_1 t}$ and $\o$ in the range of integration of \eqref{eq:below_midfm1}, we find, 
as in \mbox{\eqref{eq:below_bridge4} -- \eqref{eq:below_bridge6}}, that 
\begin{align}
	\label{eq:below_midfm1.5}
	\lim_{t\uparrow\infty}
	\P\kl{
		\kl{
			\zet
			_{\bar{x}_{i_{1}}(b_{1}t), \bar{x}_{i_{1}}(b_{1}t)+\o}
			^{(1-\s_1^2b_1 -\s_\ell^2 b_\ell) t}
			\scriptstyle
			\kl{A_t(s/t) t - \s_1^2b_1t}
		}_{s\in [b_1 t, 1-b_\ell t]} 
		\in \TT_{b_1t, (1-b_\ell)t, 0, \g}
	} = 1,
\end{align}
and 
\begin{align}
	\sfrac{1}{\s_\ell}
	\kl{
		\sqrt{2}(1-\s_\ell b_\ell)t 
		-\o 
		- \bar{x}_{i_{1}}(b_{1}t)
		+ L(t)+y
	}
	&=
	\sqrt{2}(\s_\ell-1) b_\ell t + \OO((\s_\ell^2 b_\ell t)^\g)\notag\\
	&=
	\sfrac{1}{\sqrt{2}}
	b_\ell t^{1-\a_\ell}
	(1+o(1)).\label{eq:below_midfm2}
\end{align}
Inserting \eqref{eq:below_midfm1.5} and \eqref{eq:below_midfm2} into \eqref{eq:below_midfm1}, 
we obtain
\begin{align}
	\E&\ekl{
		\mathcal{Y}_{i_1}(t) \,\big\vert\, \FF_{b_1 t}
	}\notag\\
	=\:&
	\sfrac{C}{\sqrt{2}}
	b_\ell t^{1-\a_\ell} 
	\eee^{(1-b_\ell)t}
	\eee^{
		-
		\sfrac{
			\kl{
				\sqrt{2 t} - \bar{x}_{i_{1}}(b_{1}t) + L(t)+y
			}^2
		}{
			2 (1- \s_1^2 b_1) t
		}
	}
	\int\limits
	_{I}
	\sfrac{
		\d \omega
	}{
		\sqrt{2 \pi (1 - \s_1^2 b_1 -\s_\ell^2 b_\ell)t}
	}\label{eq:below_midfm3}
	\\
	&\qquad\times
	\exp\bbbbkl{
		{-} (1-\s_1^2 b_1)\sfrac{
			\kl{
				\o
				\,-\,
				\kl{1-\s_1^2 b_1}^{-1}
				\kl{
					1-\s_1^2 b_1 - \s_\ell^2 b_\ell
				}
				\kl{
					\sqrt{2}t - \bar{x}_{i_{1}}(b_{1}t) + L(t)+y
				}
			}^2
		}{
			2 (1 - \s_1^2 b_1 -\s_\ell^2 b_\ell)\s_\ell^2 b_\ell t
		}
	}(1+o(1)).
	\notag
\end{align}
For $\g>1/2$ and $\bar{x}_{i_1} \in \LL_{b_1 t}$, the integral in \eqref{eq:below_midfm3} equals
\begin{align}
	\sfrac{\s_\ell b_\ell^{1/2}}{(1-\s_1^2 b_1)^{1/2}}(1+o(1)).\label{eq:below_midfm4}
\end{align}
Inserting \eqref{eq:below_midfm4} into \eqref{eq:below_midfm3}	completes the proof.
\end{proof}

\begin{proof}[Proof of Lemma~\ref{lem:below_SecondMomentsMiddlePart}]
As in the proof of  Lemma~\ref{prop:above_middle_second_moment}, we write 
\begin{align}\label{eq:below_smmi1}
	\E\ekl{
		\mathcal{Y}^2_{i_1}(t) \,\big\vert\, \FF_{b_1 t}
	} 
	= (T1) + (T2),
\end{align}
with
\begin{align}\label{eq:below_smmi2}
	(T1) &\equiv C^2 \eee^{2b_\ell t}
	\E\bbbbekl{\!\!
		\sum_{\substack{
				i_2 \le n^{i_1}(b_2t), \dots,\\ i_{\ell-1} \le  n^{\mi_{\ell-2}}(b_{\ell-1}t) \\
		}}\!\!\!\!
		\1_{\sum_{k=1}^{\ell-1} \multixb{k} \in\, \TT_{b_1t, (1-b_\ell)t, 0, \g}}
		\\
		&\qquad\qquad\qquad\quad\times\bkl{\Delta^{\mi_{\ell-1}}(t)\!}^{2}
		\!\,\eee^{
			{-}\sfrac{
				\kl{\Delta^{\mi_{\ell-1}}(t) + \sqrt{2} b_\ell 	t}^2
			}{
				b_\ell t
			}
		}
		\Bigg\vert
		\FF_{b_1 t}
	}\!,\notag\\
	(T2)&\equiv C^2 \eee^{2b_\ell t}
	\E\bbbbekl{\!
		\sum_{\substack{
				i_2, j_2 \le n^{i_1}(b_2t), \dots,
				i_{\ell-1} \le  n^{\mi_{\ell-2}}(b_{\ell-1}t),\\
				j_{\ell-1} \le  n^{\bar{j}_{\ell-2}}(b_{\ell-1}t), \mi_{\ell-1} \neq \bar{j}_{\ell-1}\\
		}}
		\!
		\1_{\sum_{k=1}^{\ell-1} \multixb{k}\!,\,
			\sum_{k=1}^{\ell-1} \bar{x}_{j_k}^{\bar{j}_{k-1}} \in\, \TT_{b_1t, (1-b_\ell)t, 0, \g}}
		\notag\\
		&\qquad\qquad\qquad\quad\times\Delta^{\mi_{\ell-1}}(t)
		\Delta^{\bar{j}_{\ell-1}}(t)
		\exp\bbbkl{
			{-}\sfrac{
				\kl{\Delta^{\mi_{\ell-1}}(t) + \sqrt{2} b_\ell t}^2
				+
				\kl{\Delta^{\bar{j}_{\ell-1}}(t) + \sqrt{2} b_\ell t}^2
			}{
				2b_\ell t
			}
		}
		\,\Bigg\vert\,
		\FF_{b_1 t}
	}.\notag
\end{align}
Using that $\TT_{b_1t, (1-b_\ell)t, 0, \g} \subset \TT_{(1-b_\ell)t, (1-b_\ell)t, 0, \g}$, we obtain, as in \eqref{eq:below_midfm0}--\eqref{eq:below_midfm1}, by the many-to-one lemma that
\begin{align}
	\label{eq:below_midsm1}
	(T1)
	&\leq
	C^2 \eee^{(1 -b_1 + b_\ell) t}
	\int\limits		
	_{I}
	\sfrac{\d \o}{\sqrt{2\pi(1-\s_1^2 b_1 - \s_\ell^2 b_\ell) t}}
	\eee^{
		-\sfrac{
			\o^2
		}{
			2 (1-\s_1^2 b_1 - \s_\ell^2 b_\ell) t
		}
	}
	\\
	&\quad\times
	\sfrac{1}{\s^2_\ell}
	\kl{
		\sqrt{2}(1-\s_\ell b_\ell)t 
		-\o 
		- \bar{x}_{i_{1}}(b_{1}t)
		+ L(t)+y
	}^2
	\eee^{
		-\sfrac{
			\kl{
				\sqrt{2}t 
				-\o
				- \bar{x}_{i_{1}}(b_{1}t)
				+ L(t)+y
			}^2
		}{
			\s_\ell^2 b_\ell t
		}
	},\notag
\end{align}
with $I$ as in \eqref{eq:appendixIntervalDef}.
We have seen in \eqref{eq:below_midfm2} that for $\o$ in the range of integration of \eqref{eq:below_midsm1}, 
\begin{align}
	\sfrac{1}{\s^2_\ell}
	\kl{
		\sqrt{2}(1-\s_\ell b_\ell)t 
		-\o 
		- \bar{x}_{i_{1}}(b_{1}t)
		+ L(t)+y
	}^2
	\leq P(t).
\end{align}
Recall that $P$ is universal notation for any function satisfying $P(t) \leq t^c$ for some constant $c>0$ and all $t>0$ large enough.
The exponential terms inside the integral in \eqref{eq:below_midsm1} are equal to 
\begin{align}
	&\exp\kl{
		{-}\sfrac{
			\kl{
				\sqrt{2}t - \bar{x}_{i_{1}}(b_{1}t)
				+ L(t)+y
			}^2
		}{
			2(1-\s_1^2 b_1) t
		}
	}
	\exp\kl{
		{-}\sfrac{
			\kl{
				\sqrt{2}t - \o- \bar{x}_{i_{1}}(b_{1}t)
				+ L(t)+y
			}^2
		}{
			2\s_\ell^2 b_\ell t
		}
	}\notag\\
	&\qquad \times
	\exp\bbbbkl{
		{-} (1-\s_1^2 b_1)\sfrac{
			\kl{
				\o
				\,-\,
				\kl{1-\s_1^2 b_1}^{-1}
				\kl{
					1-\s_1^2 b_1 - \s_\ell^2 b_\ell
				}
				\kl{
					\sqrt{2}t - \bar{x}_{i_{1}}(b_{1}t) + L(t)+y
				}
			}^2
		}{
			2 (1 - \s_1^2 b_1 -\s_\ell^2 b_\ell)\s_\ell^2 b_\ell t
		}
	}.\label{eq:below_midsm2}
\end{align}
We insert \eqref{eq:below_midsm2} back into the right-hand side of \eqref{eq:below_midsm1}, bound the last exponential term by $1$ and shift the integral by $\sqrt{2}(1-\s_\ell^2 b_\ell)t - \bar{x}_{i_{1}}(b_{1}t)$. This gives
\begin{align}
	\label{eq:below_midsm2.4}
	(T1) 
	&\leq
	P(t)
	\eee^{(1 -b_1 + b_\ell) t}
	\!
	\exp\kl{
		{-}\sfrac{
			\kl{
				\!\sqrt{2}t - \bar{x}_{i_{1}}(b_{1}t)
				+ L(t)+y
			}^2
		}{
			2(1-\s_1^2 b_1) t
		}
	}\notag\\
	&\qquad\times
	\int\limits		
	_{
		- (\s_\ell^2 b_\ell t)^{\g}
	}
	^{
		(\s_\ell^2 b_\ell t)^{\g}
	}
	\!\!\!\!
	\sfrac{\d \o}{\!\sqrt{2\pi(1-\s_1^2 b_1 - \s_\ell^2 b_\ell) t}}
	\exp\kl{
		{-}\sfrac{
			\kl{
				\!\sqrt{2}\s_\ell^2 b_\ell t - \o
				+ L(t)+y
			}^2
		}{
			2\s_\ell^2 b_\ell t
		}
	}\!.
\end{align}
For $\o$ in the range of the integral in \eqref{eq:below_midsm2.4},
\begin{align}\label{eq:below_midsm2.42}
	\begin{split}
		b_\ell t
		-
		\sfrac{
			\kl{
				\sqrt{2}\s_\ell^2 b_\ell t - \o
				+ L(t)+y
			}^2
		}{
			2\s_\ell^2 b_\ell t
		}
		&\leq
		(1-\s_\ell^2) b_\ell t + \!\sqrt{2}(\s_\ell^2 b_\ell t)^\g + \!\sqrt{2} L(t)+y= -b_\ell t^{1-\a_\ell} (1+o(1)).
	\end{split}
\end{align}
Assumption~\ref{as:belowLSpeed3} implies that $b_\ell t^{1-\a_\ell} \gg t^{1/2}$, so $(T1)$ is bounded by the right-hand side of \eqref{eq:below_midsm0}.
Dropping the condition $\TT_{b_1t, (1-b_\ell)t, 0, \g}$ except at the endpoint and at the time of the branching gives via the many-to-two lemma
\begin{align}
	(T2) 
	&\leq
	C^2 \eee^{2(1-b_1) t}
	\!
	\int\limits_{b_1 t}^{(1-b_\ell) t}
	\d s \,\eee^{-(s-b_1 t)}
	\int\limits
	_{I_1}
	\!\!
	\sfrac{\d\o_1}{\sqrt{2\pi  \kl{A_t(s/t) - \s_1^2 b_1 }t}}
	\eee^{
		-\sfrac{
			\o_1^2
		}{
			2\kl{A_t(s/t) - \s_1^2 b_1 }t
		}
	}\notag\\
	&\times
	\Bigg(
	\int\limits
	_{I_2}
	\sfrac{\d\o_2}{\sqrt{2\pi  \kl{1- \s_\ell^2 b_\ell - A_t(s/t) } t}}
	\eee^{
		{-}\sfrac{
			\o_2^2
		}{
			2\kl{1- \s_\ell^2 b_\ell - A_t(s/t)} t
		}
	}
	\\\notag
	&\quad\times
	\!\sfrac{1}{\s_\ell}\!
	\kl{
		\!\sqrt{2}(1\!-\!\s_\ell b_\ell)t 
		-\o_1-\o_2 
		- \bar{x}_{i_{1}}(b_{1}t)
		+ L(t)+y
	}
	\eee^{\!
		{-}\sfrac{
			\kl{
				\!\sqrt{2}t 
				-\o_1-\o_2
				- \bar{x}_{i_{1}}(b_{1}t)
				+ L(t)+y
			}^2
		}{
			2\s_\ell^2 b_\ell t
		}
	}\!
	\Bigg)^{\!\! 2}\!,\label{eq:below_midsm2.5}
\end{align}
where 
\begin{align}
	f_{t,\g}(s) &\equiv \bkl{A_t(s/t) \wedge \kl{1 - A_t(s/t)}\!}^\g t^\g, \\
	I_1 &\equiv  \kl{
		\bar{x}_{i_{1}}\!(b_{1}t)  - \!\sqrt{2}t A_t(s/t) - f_{t,\g}(s), 
		\bar{x}_{i_{1}}\!(b_{1}t)  - \!\sqrt{2}t A_t(s/t) + f_{t,\g}(s)
	},\notag\\
	I_2&\equiv \kl{
		\bar{x}_{i_{1}}\!(b_{1}t) + \o_1 - \!\sqrt{2}(1\!-\!\s_\ell^2 b_\ell)t
		- (\s_\ell^2 b_\ell t)^\g\!,
		\bar{x}_{i_{1}}\!(b_{1}t) + \o_1 - \!\sqrt{2}(1\!-\!\s_\ell^2 b_\ell)t
		+ (\s_\ell^2 b_\ell t)^\g
	}\!.\notag
\end{align}
In the $\o_2$-integral in \eqref{eq:below_midsm2.5},
\begin{align}
	\sfrac{1}{\s_\ell}
	\kl{
		\!\sqrt{2}(1-\s_\ell b_\ell)t 
		-\o_1-\o_2 
		- \bar{x}_{i_{1}}(b_{1}t)
		+ L(t)+y
	}
	\leq P(t),
\end{align}
and the exponential terms are equal to
\begin{align}
	&\exp\bbbkl{
		{-}\sfrac{
			\kl{
				\!\sqrt{2}t 
				-\o_1
				- \bar{x}_{i_{1}}(b_{1}t)
				+ L(t)+y
			}^2
		}{
			2\kl{1-A_t(s/t)} t
		}
	}\notag\\
	&\quad\times\exp\kl{
		{-}\sfrac{
			(1-A_t(s/t))
			\,\bbkl{
				\o_2 - 
				\frac{1-\s_\ell^2 b_\ell-A_t(s/t)}{1-A_t(s/t)}
				\bkl{\!\sqrt{2}t 
					-\o_1
					- \bar{x}_{i_{1}}(b_{1}t)
					+ L(t)+y}
			}^2
		}{
			2\kl{1-\s_\ell^2 b_\ell-A_t(s/t)} \s_\ell^2 b_\ell t
		}
	}.
	\label{eq:below_midsm3}
\end{align}
The first factor in \eqref{eq:below_midsm3} does not depend on $\o_2$ and the integral of the second factor over $\o_2$ is Gaussian and thus can be bounded by $P(t)$. This shows that the $\o_2$-integral in \eqref{eq:below_midsm2.5} is bounded by
\begin{align}
	P(t)\,\eee^{
		{-}\sfrac{
			\kl{
				\sqrt{2}t 
				-\o_1
				- \bar{x}_{i_{1}}(b_{1}t)
				+ L(t)+y
			}^2
		}{
			2\kl{1-A_t(s/t)} t
		}
	}.
\end{align}
Inserting this bound into \eqref{eq:below_midsm2.5} and then proceeding as in \eqref{eq:below_midsm2}--\eqref{eq:below_midsm2.4}, we obtain
\begin{align}
	\label{eq:below_midsm4}
	(T2)
	&\leq
	P(t)\, \eee^{2(1-b_1) t}
	\!
	\int\limits_{b_1 t}^{(1-b_\ell) t}
	\!\d s \,\eee^{-(s-b_1 t)}
	\notag\\
	&\quad\times
	\int\limits		
	_{I_1}
	\!\!
	\sfrac{\d\o_1}{\sqrt{2\pi  \kl{A_t(s/t) - \s_1^2 b_1 }t}}
	\eee^{
		-\sfrac{
			\o_1^2
		}{
			2\kl{A_t(s/t) - \s_1^2 b_1 }t
		}
	}
	\eee^{
		{-}\sfrac{
			\kl{
				\!\sqrt{2}t 
				-\o_1
				- \bar{x}_{i_{1}}(b_{1}t)
				+ L(t)+y
			}^2
		}{
			\kl{1-A_t(s/t)} t
		}
	}\notag\\
	&\leq
	P(t)\, \eee^{(2-b_1) t}
	\eee^{
		{-}\sfrac{
			\kl{
				\!\sqrt{2}t 
				- \bar{x}_{i_{1}}(b_{1}t)
				+ L(t)+y
			}^2
		}{
			2\kl{1-\s_1^2 b_1} t
		}
	}
	\!
	\int\limits_{b_1 t}^{(1-b_\ell) t}
	\!\d s \,\eee^{-s}
	\notag\\
	&\quad\times
	\int\limits_{
		- f_{t,\g}(s) 
	}^{f_{t,\g}(s) }\!\!
	\!\sfrac{\d\o_1}{\sqrt{2\pi  \kl{A_t(s/t) - \s_1^2 b_1 }t}}
	\eee^{
		{-}\sfrac{
			\kl{
				\!\sqrt{2}(1-A_t(s/t))t 
				-\o_1
				+ L(t)+y
			}^2
		}{
			2\kl{1-A_t(s/t)} t
		}
	}.
\end{align}
Bounding $\o_1$ by $f_{t,\g}(s)$, we see that the $\o_1$-integral is not larger than
\begin{align}
	P(t) \eee^{
		{-}\sfrac{
			\kl{
				\sqrt{2}(1-A_t(s/t))t 
				-f_{t,\g}(s)
				+ L(t)+y
			}^2
		}{
			2\kl{1-A_t(s/t)} t
		}
	},
\end{align}
so we have
\begin{align}
	\label{eq:below_midsm7}
	(T2)\leq
	P(t)\, \eee^{(1-b_1) t}
	\eee^{
		{-}\sfrac{
			\kl{
				\sqrt{2}t 
				- \bar{x}_{i_{1}}(b_{1}t)
				+ L(t)+y
			}^2
		}{
			2\kl{1-\s_1^2 b_1} t
		}
	}
	\!
	\int\limits_{b_1 t}^{(1-b_\ell) t}
	\d s \,\eee^{t A_t(s/t) - s + \sqrt{2} f_{t,\g}(s)}(1+o(1)).
\end{align}
For all $s$ in the range of the integral in \eqref{eq:below_midsm7}, $f_{t,\g}(s) \leq t^\g$ and by \eqref{eq:below_midsm-1}, $t A_t(s/t) - s \leq -t^{\tilde{\g}}$ for some $\tilde{\g}>\g$ and $t$ large enough. Thus, the integrand in \eqref{eq:below_midsm7} is bounded from above by $\eee^{-t^{\tilde{\g}}}(1+o(1))$ for $t$ large enough. We conclude that
$(T2)$ is not larger than the right-hand side of \eqref{eq:below_midsm0}.
\end{proof}

\begin{proof}[Proof of Lemma~\ref{lem:below_FirstMomentsFirstPart}]
By the many-to-one lemma,	
\begin{align}
	\E\ekl{
		\tilde{\YY}_{j}(t) \,\big\vert\, \FF_{t^\b}
	}
	&=\sfrac{C}{\sqrt{2(1-\s_1^2 b_1)}}
	b_\ell^{3/2} t^{1-\a_\ell}
	\eee^{t-t^\b}
	\int\limits
	_{J}
	\sfrac{
		\d \omega
	}{
		\sqrt{2 \pi \s_1^2 t^*}
	}
	\eee^{
		-\sfrac{
			\omega^2
		}{
			2 \s_1^2 t^*
		}
	}\eee^{
		{-}\sfrac{
			\kl{
				\!\sqrt{2}t 
				-\o
				- \bar{x}_{j}(t^\b)
				+ L(t)+y
			}^2
		}{
			2 (1-\s_1^2 b_1) t
		}
	}\notag\\
	&\quad\times
	\P
	\kl{
		\zet_{
			\s_1^{-1} \bar{x}_{j}(t^\b),\, \s_1^{-1} \kl{\bar{x}_{j}(t^\b) + \o}
		}^{t^*} 
		(s) 
		< \!\sqrt{2}s - t^{\b\delta/2} \, \forall_{s \in [0,t^*]} 
	}
	,\label{eq:below_1fm1}
\end{align}
where 
\begin{align}\label{eq:appendixDefInterval2}
	J\equiv \kl{
		\sqrt{2} \s_1^2 b_1 t- \bar{x}_{j}(t^\b) - (\s_1^2 b_1 t)^{\g},
		\sqrt{2} \s_1^2 b_1 t- \bar{x}_{j}(t^\b) + (\s_1^2 b_1 t)^{\g}
	}.
\end{align}
By Lemma~\ref{lem:BB_line}, the probability in \eqref{eq:below_1fm1} satisfies
\begin{align}
	\P&
	\kl{
		\zet_{
			\s_1^{-1} \bar{x}_{j}(t^\b),\, \s_1^{-1} \kl{\bar{x}_{j}(t^\b) + \o}
		}^{t^*} 
		(s) 
		< \!\sqrt{2}s - t^{\b\delta/2} \, \forall_{s \in [0,t^*]} 
	}\notag\\
	=\:&
	\P
	\kl{
		\zet_{
			\s_1^{-1} \kl{\bar{x}_{j}(t^\b) + t^{\b\delta/2}} - \sqrt{2}t^\b,\,
			\s_1^{-1} \kl{\bar{x}_{j}(t^\b) + \o + t^{\b\delta/2}} - \sqrt{2}b_1 t
		}^{t^*} 
		(s) 
		< 0 \, \forall_{s \in [0,t^*]} 
	}\notag\\
	=\:&
	1-\exp\kl{
		-\sfrac{2}{t^*} 
		\kl{
			\sqrt{2}t^\b -  \sfrac{\bar{x}_{j}(t^\b) + t^{\b\delta/2}}{\s_1}
		}
		\kl{
			\sqrt{2}b_1 t -  \sfrac{\bar{x}_{j}(t^\b) + \o + t^{\b\delta/2}}{\s_1}
		}
	}
	\notag\\
	=\:&
	\sqrt{2}t^{-\a_1} \kl{
		\sqrt{2}\s_1t^\b - \bar{x}_{j}(t^\b) - t^{\b\delta/2}
	}(1+o(1)).
	\label{eq:below_1fm2}
\end{align}
In the last step we used that for $\bar{x}_{j} \in\LL_{t^\b}$, 
\begin{align}
	0 \leq \sqrt{2}t^\b -  \sfrac{\bar{x}_{j}(t^\b) + t^{\b\delta/2}}{\s_1} = \OO(t^\b),
\end{align}
and that for $\o$ in the range of integration in \eqref{eq:below_1fm1},
\begin{align}
	\sqrt{2}b_1 t -  \sfrac{\bar{x}_{j}(t^\b) + \o + t^{\b\delta/2}}{\s_1} = 2^{-1/2} b_1 t^{1-\a_1} (1+o(1)).
\end{align}
Inserting \eqref{eq:below_1fm2} into \eqref{eq:below_1fm1}, we get that $\E\ekl{
	\tilde{\YY}_{j}(t) \,\big\vert\, \FF_{t^\b}
}$ is up to an error term $(1+o(1))$ equal to
\begin{align}\label{eq:below_1fm3}
	&
	C
	b_\ell^{3/2} t^{1-\a_1 -\a_\ell}
	\kl{
		\sqrt{2}\s_1t^\b - \bar{x}_{j}(t^\b) - t^{\b\delta/2}
	}
	\eee^{t-t^\b}
	\eee^{
		-\sfrac{
			\kl{
				\sqrt{2}t 
				- \bar{x}_{j}(t^\b))
				+ L(t)+y
			}^2
		}{
			2 (t - \s_1^2 t^\b)
		}
	}\\
	&\:\:\times\!
	\int\limits
	_{J}
	\!
	\sfrac{
		\d \omega
	}{
		\sqrt{2 \pi (1-\s_1^2 b_1) \s_1^2 t^*}
	}
	\exp\kl{
		-\sfrac{
			t-\s_1^2t^\b
		}{
			2 \s_1^2 (1-\s_1^2b_1) t t^*
		}
		\bbkl{
			\o 
			- \sfrac{
				\s_1^2 t^*
			}{
				t - \s_1^2 t^\b
			}
			\bkl{
				\sqrt{2}t - \bar{x}_{j}(t^\b)
				+ L(t)+y
			}
		}^2
	}\!.\notag
\end{align}
A Gaussian tail bound shows that for $\g>1/2$ and $\bar{x}_{j} \in \LL_{t^\b}$,  the integral in \eqref{eq:below_1fm3} is of order~$1$. Recalling the definition of $L(t)$ in \eqref{eq:below_defL}, we find that \eqref{eq:below_1fm3} is equal to the right-hand side of \eqref{eq:below_FMFP_Result}.
\end{proof}

\begin{proof}[Proof of Lemma~\ref{lem:below_SecondMomentsFirstPart}]
The structure of this proof is the same as that of Lemma~\ref{lem:below_SecondMomentsMiddlePart}, which we will refer to for explanations.
We write
\begin{align}
	\E\ekl{
		\tilde{\YY}^2_{j}(t) \,\big\vert\, \FF_{t^\b}
	}=(T1)+(T2),
\end{align}
where 
\begin{align}\label{eq:below_smfi1}
	(T1) &\equiv \kl{\sfrac{C	b_\ell^{3/2} t^{1-\a_\ell}}{\sqrt{2(1-\s_1^2 b_1)}}}^2 
	\eee^{2(1-b_1) t}\,
	\E\,\bbbbekl{
		\sum_{\substack{
				j^* \le n^j(t^*), \\
				\bar{x}_{j^*}^{\, j} \in \LL_{t^*}
		}}
		\eee^{
			{-}
			\sfrac{
				\kl{
					\sqrt{2 t} - \bar{x}_{j^*}^{\, j}(t^*) - \bar{x}_j(t^\b) + L(t)+y
				}^2
			}{
				(1- \s_1^2 b_1) t
			}
		}
		\,\Bigg\vert\,
		\FF_{b_1 t}
	},\notag\\
	(T2)&\equiv \kl{\sfrac{C	b_\ell^{3/2} t^{1-\a_\ell}}{\sqrt{2(1-\s_1^2 b_1)}}}^2
	\eee^{2(1-b_1) t}\,
	\E\,\bbbbekl{
		\sum_{\substack{
				j^*\!,\, k^* \le n^j(t^*),\, j^*\neq k^*, \\
				\bar{x}_{j^*}^{\, j},\,
				\bar{x}_{k^*}^{\, j} \in \LL_{t^*}
		}}
		\eee^{
			{-}
			\sfrac{
				\kl{
					\sqrt{2 t} - \bar{x}_{j^*}^{\, j}(t^*) - \bar{x}_j(t^\b) + L(t)+y
				}^2
			}{
				2(1- \s_1^2 b_1) t
			}
		}\notag\\
		&\qquad\qquad\qquad\qquad\qquad\qquad
		\times 
		\eee^{
			{-}
			\sfrac{
				\kl{
					\sqrt{2 t} - \bar{x}_{k^*}^{\, j}(t^*) - \bar{x}_j(t^\b) + L(t)+y
				}^2
			}{
				2(1- \s_1^2 b_1) t
			}
		}
		\,\Bigg\vert\,
		\FF_{b_1 t}
	}.
\end{align}
As in \eqref{eq:below_midsm1}, we obtain
\begin{align}
	\label{eq:below_smfi2}
	(T1) 
	&\leq
	P(t)\, \eee^{2(1-b_1) t + t^*}
	\int\limits		
	_{J}
	\sfrac{\d \o}{\sqrt{2\pi \s_1^2 t^*}}
	\eee^{
		-\sfrac{
			\o^2
		}{
			2 \s_1^2 t^*
		}
	}
	\eee^{
		-\sfrac{
			\kl{
				\sqrt{2}t 
				-\o
				- \bar{x}_{j}(t^\b)
				+ L(t)+y
			}^2
		}{
			(1-\s_1^2 b_1) t
		}
	},
\end{align}
with $J$ as in \eqref{eq:appendixDefInterval2}.
Proceeding as in \eqref{eq:below_midsm2}--\eqref{eq:below_midsm2.42}, we see that the right hand side of \eqref{eq:below_smfi2} is bounded by
\begin{align}
	P(t)\,
	\eee^{-\sqrt{2}\kl{\sqrt{2}\s_1 t^\b - x_j(t^\b)}}
	\eee^{-b_1 t^{1/2}}(1+o(1)).
\end{align}
We get, as in \eqref{eq:below_midsm2.5}, that
\begin{align}
	\label{eq:below_fism3}
	(T2) 
	&\leq
	P(t)\, \eee^{2(1-b_1) t}
	\!
	\int\limits_{0}^{t^*}
	\d s \,\eee^{t^* + s}
	\int\limits
	_{
		-\infty
	}^{
		\sqrt{2}\s_1(b_1 t - s) - t^{\b\delta/2} - x_j(t^\b)
	}
	\!\!
	\sfrac{\d\o_1}{\sqrt{2\pi\s_1^2  \kl{t^*-s}}}
	\eee^{
		-\sfrac{
			\o_1^2
		}{
			2\s_1^2  \kl{t^*-s}
		}
	}
	\notag\\
	&\quad\times
	\kl{\,
		\int\limits
		_{J-\o_1}
		\!
		\sfrac{\d\o_2}{\sqrt{2\pi \s_1^2 s}}
		\eee^{
			{-}\sfrac{
				\o_2^2
			}{
				2\s_1^2 s
			}
		}
		\eee^{
			-\sfrac{
				\kl{
					\sqrt{2}t 
					-\o_1-\o_2
					- \bar{x}_{j}(t^\b)
					+ L(t)+y
				}^2
			}{
				2 (1-\s_1^2 b_1) t
			}
		}
	}^2,
\end{align}
where $\o_2 \in J-\o_1$ denotes $\o_2+\o_1 \in J$.	
As in \eqref{eq:below_midsm3}--\eqref{eq:below_midsm7}, we see that the right-hand side of \eqref{eq:below_fism3} is not larger than
\begin{align}
	P(t)\,
	\eee^{-\sqrt{2}\kl{\sqrt{2}\s_1 t^\b - x_j(t^\b)}}
	\eee^{- \sqrt{2} t^{\b\delta/2}}(1+o(1)).
\end{align}
\end{proof}






\end{document}